\documentclass[a4paper,11pt]{article}
\usepackage[T1]{fontenc}
\usepackage[latin1]{inputenc}

\usepackage{amsmath,amsfonts}
\usepackage{amsthm}
\usepackage{amssymb}
\usepackage[USenglish]{babel}
\usepackage{amsfonts}
\usepackage{enumerate}
\usepackage{verbatim}
\usepackage{graphicx}
\usepackage[left=3cm, right=3cm, top=3cm, bottom=3cm]{geometry}
\usepackage{tikz}
\usepackage{csquotes}
\usepackage{mathtools,stmaryrd}
\usepackage{mathscinet} 
\usepackage{tabularx}
\usepackage{scrextend}
\usepackage{float}

\usepackage{multirow}

\usepackage[final, pdftex, pdfpagelabels, pdfstartview = {FitH}, plainpages = false, linktoc=all,linkcolor=black, citecolor=blue, urlcolor=blue,
filecolor=black]{hyperref}
\usepackage{cleveref}

\usetikzlibrary{matrix,arrows,decorations.pathmorphing,decorations.pathreplacing,decorations.markings}

\theoremstyle{plain}
\newtheorem{theorem}{Theorem}[section]
\newtheorem{lemma}[theorem]{Lemma}
\newtheorem{prop}[theorem]{Proposition}
\newtheorem{cor}[theorem]{Corollary}

\theoremstyle{definition}
\newtheorem{definition}[theorem]{Definition}

\newtheorem{remark}[theorem]{Remark}
\newtheorem{example}[theorem]{Example}

\theoremstyle{remark}
\newtheorem{claim}[theorem]{Claim}

\newtheorem*{notation}{Notation}

\Crefname{prop}{Proposition}{Propositions}
\Crefname{cor}{Corollary}{Corollaries}

\renewcommand{\tilde}{\widetilde}
\renewcommand{\bar}{\overline}
\newcommand{\bbC}{\mathbb{C}}
\newcommand{\bbD}{\mathbb{D}}

\newcommand{\bbN}{\mathbb{N}}

\newcommand{\bbQ}{\mathbb{Q}}
\newcommand{\bbR}{\mathbb{R}}
\newcommand{\bbZ}{\mathbb{Z}}

\newcommand{\raw}{\rightarrow}

\newcommand{\Rra}[1]{\xrightarrow[R]{#1}}
\newcommand{\xra}{\xrightarrow}
\newcommand{\cug}{\subseteq}

\newcommand{\undx}{{\underline{x}}}

\newcommand{\undw}{{\underline{w}}}
\newcommand{\undH}{{\underline{\mathbf{H}}}}

\newcommand{\calA}{\mathcal{A}}

\newcommand{\calC}{\mathcal{C}}
\newcommand{\calD}{\mathcal{D}}
\newcommand{\calF}{\mathcal{F}}

\newcommand{\calH}{\mathcal{H}}

\newcommand{\calK}{\mathcal{K}}

\newcommand{\calS}{\mathcal{S}}

\newcommand{\frh}{\mathfrak{h}}
\newcommand{\frD}{\mathfrak{D}}
\newcommand{\frP}{\mathfrak{P}}
\newcommand{\frQ}{\mathfrak{Q}}
\newcommand{\frR}{\mathfrak{R}}

\newcommand{\sbim}{{\mathbb{S}Bim}}

\newcommand{\realign}{ \hspace*{\dimexpr-\leftmargini-\leftmarginii}}

\newcommand{\isom}{\xrightarrow{\sim}}

\newcommand{\bfH}{\mathbf{H}}
\newcommand{\bfP}{\mathbf{P}}
\newcommand{\bfQ}{\mathbf{Q}}
\newcommand{\bfR}{\mathbf{R}}

\newcommand{\bfS}{\mathbf{S}}
\newcommand{\bfU}{\mathbf{U}}
\newcommand{\bfV}{\mathbf{V}}

\newcommand{\bfZ}{\mathbf{Z}}

\newcommand{\cus}{{\subset\hspace*{-0.30cm}\raisebox{0.05cm}{\scalebox{0.5}{$\oplus$}}}\hspace{0.05cm}}
\newcommand{\cussmall}{{\subset\hspace*{-0.25cm}\raisebox{0.04cm}{\scalebox{0.5}{$\oplus$}}}\hspace{0.04cm}}

\newcommand{\form}{\langle
-,-\rangle}

\newcommand{\ctop}{c_{top}}

\newcommand{\Address}{
 \bigskip{\footnotesize

 \textsc{Albert-Ludwigs-Universit\"at Freiburg, Freiburg im Breisgau, Germany}\par\nopagebreak
 \textit{E-mail address}: \texttt{leonardo.patimo@math.uni-freiburg.de}
}}

\DeclareMathOperator{\cha}{ch}

\DeclareMathOperator{\End}{End}

\DeclareMathOperator{\grrk}{grrk}
\DeclareMathOperator{\hgt}{ht}
\DeclareMathOperator{\Hom}{Hom}
\DeclareMathOperator{\Iden}{Id}

\DeclareMathOperator{\spa}{span}

\DeclareMathOperator{\Sym}{Sym}
\DeclareMathOperator{\Trace}{Tr}
\DeclareMathOperator{\Peaks}{Peaks}
\DeclareMathOperator{\adm}{Adm}
\DeclareMathOperator{\Conf}{Conf}
\DeclareMathOperator{\rex}{rex}

\newcommand{\carcu}{\curvearrowright}
\newcommand{\carau}{\curvearrowleft}
\newcommand{\carcd}{\mathbin{\rotatebox[origin=c]{180}{$\curvearrowright$}}}
\newcommand{\carad}{\mathbin{\rotatebox[origin=c]{180}{$\curvearrowleft$}}}

\newcommand{\vertcong}{\mathbin{\rotatebox[origin=c]{-90}{$\cong$}}}

\newdimen\boxmini

\newdimen\boxlength
\newcount\counter
\newcommand\tabpath[1]
{
\counter=0
\coordinate[vertex] (current) at (0,0);
\foreach\x in {#1}
{
\global\advance\counter by 1
\if\x +
\def\st{+1}
\else
\if\x 2
\def\st{+1}
\else
\def\st{-1}
\fi
\fi
\draw[edge] (current) -- ++(1,\st) coordinate[vertex] (current);
};
}

\newcommand\tabpathshift[3]
{
	\counter=0
	\coordinate[vertex] (current) at (#2,#3);
	\foreach\x in {#1}
	{
		\global\advance\counter by 1
		\if\x +
		\def\st{+1}
		\else
		\if\x 2
		\def\st{+1}
		\else
		\def\st{-1}
		\fi
		\fi
		\draw[edge] (current) -- ++(1,\st) coordinate[vertex] (current);
	};
}

\newcommand\tabpathc[2]
{
\counter=0
\coordinate[vertex] (current) at (0,0);
\foreach\x in {#1}
{
\global\advance\counter by 1
\if\x +
\def\st{+1}
\else
\if\x 2
\def\st{+1}
\else
\def\st{-1}
\fi
\fi
\draw[color=#2,edge] (current) -- ++(1,\st) coordinate[vertex] (current);
};

}

\newcommand{\dpath}{\searrow}
\newcommand{\upath}{\nearrow}
\newcommand{\barF}{\bar{F}}

\newcommand{\duarrow}{\mathrlap{\nwarrow}\nearrow}
\newcommand{\ddarrow}{\mathrlap{\swarrow}\searrow}

\DeclareMathOperator{\Gr}{Gr}
\title{Bases of the Intersection Cohomology of Grassmannian Schubert Varieties}

\author{Leonardo Patimo}

\begin{document}

\maketitle

\begin{abstract}
	The parabolic Kazhdan--Lusztig polynomials for Grassmannians can be computed by counting Dyck partitions. We ``lift'' this combinatorial formula to the corresponding category of singular Soergel bimodules to obtain bases of the Hom spaces between indecomposable objects. In particular, we
	 describe bases of intersection cohomology of Schubert varieties in Grassmannians parametrized by Dyck partitions which extend (after dualizing) the classical Schubert basis of the ordinary cohomology.
\end{abstract}

\tableofcontents 

\section*{Introduction}
In some sense the cohomology of the Grassmannians has been studied since the late 19th century. The original motivation was given by Schubert calculus: it turns out that many basic questions in enumerative geometry (for example: ``How many lines intersect four given lines in a $3$ dimensional space?'') can be approached via computations in the cohomology ring.

We denote by $\Gr(i,n)$ the Grassmannians of vector spaces of dimension $i$ inside $\bbC^n$. 
The Schubert cells $\calC_\lambda\subset \Gr(i,n)$ are parameterized by piece-wise linear paths $\lambda$ in $\bbR^2$ from $(0,i)$ to $(n,n-i)$ which can be obtained by joining segments of length $\sqrt{2}$ and direction $\dpath$ or $\upath$. The closure of the Schubert cells $X_\lambda:=\bar{\calC_\lambda}$ are called Schubert varieties.
The Schubert cells form a cell decomposition of $\Gr(i,n)$ and, by taking characteristic classes of Schubert varieties, we obtain a distinguished basis of the cohomology ring $H^\bullet(\Gr(i,n),\bbQ)$, called the Schubert basis. 
With the help of classical results such as Pieri's formula (which describes how to multiply the class of a Schubert variety with the Chern class of a line bundle) and the more general Littlewood--Richardson rule (which describes how to multiply two arbitrary Schubert classes) the Schubert basis allows us to fully understand the ring structure of $H^\bullet(\Gr(i,n),\bbQ)$. The same applies if we consider the cohomology ring of a Schubert variety $H^\bullet(X_\lambda,\bbQ)$: it can be obtained as a quotient of $H^\bullet(\Gr(i,n),\bbQ)$ and the elements smaller or equal than $\lambda$ in the Schubert basis descend to a distinguished basis of $H^\bullet(X_\lambda,\bbQ)$.

A related, yet less well-understood, object is the intersection cohomology of a Schubert variety
 $IH^\bullet(X_\lambda,\bbQ)$ which encodes a lot of data relevant from the point of view of a representation theorist. Since Schubert varieties are in general singular, their ordinary singular cohomology embeds as a submodule in the intersection cohomology. It is then natural to ask whether it is possible to extend the Schubert basis in a natural way to $IH^\bullet(X_\lambda,\bbQ)$, that is if it is possible find a distinguished basis for the intersection cohomology of a Schubert variety $X_\lambda$.
The study of bases of the intersection cohomology of Schubert varieties of a Grassmannian is the main subject of the present paper.

Our approach to this problem is mostly algebraic,
using the technology of singular Soergel bimodules as in \cite[Definition 7.1.1]{W4}. We call Grassmannian Soergel bimodules the singular Soergel bimodules corresponding to the maximal parabolic subgroup $S_i\times S_{n-i}$ of the symmetric group $S_n$. 
Then, we work with the modules $IH^\bullet(X_\lambda,\bbQ)$ by reinterpreting them as indecomposable Grassmannian Soergel bimodules. 
 Finding basis of the intersection cohomology is equivalent to finding bases of the Hom spaces between Grassmannian Soergel bimodules. This reinterpretation is very convenient as it allows us to have the rich technology of Soergel bimodules at disposal. For example, morphisms between singular Soergel bimodules can be depicted using diagrams and the dimension of the spaces of morphisms can be computed in the corresponding Hecke algebra.
The present paper can be regarded as a first step towards establishing bases in singular Soergel calculus. We hope that some of our constructions, e.g. the partial order on the Dyck partitions introduced in \Cref{PartialOrder}, can set the path to build similar bases between the Hom spaces in the singular world with respect to an arbitrary parabolic subgroup, at least in type $A$.

Dimensions of indecomposable singular Soergel bimodules are given by the parabolic Kazhdan--Lusztig polynomials which, as usual for these type of polynomials, can be computed via a recursive formula. However, Schubert varieties in Grassmannians are very special among Schubert varieties in that they all admit small resolution of singularities \cite[Theorem 1]{Zel}. At the level of Kazhdan--Lusztig polynomials this is reflected in the existence of combinatorial non-recursive formulas.

The first version of such a formula is due to Lascoux and Sch\"utzenberger, who described a combinatorics for these polynomials involving ``binary trees'' \cite[Th\'eor\`eme 7.8]{LS}. More recently, Shigechi and Zinn-Justin gave an equivalent formulation of this combinatorics, this time involving Dyck partitions \cite[Corollary 2]{SZJ}. A major advantage of using Dyck partitions is that in this setting it is also possible to describe formulas for the inverse Kazhdan--Lusztig polynomials (this was originally shown by Brenti in \cite[Theorem 5.1]{Br}).

We recall what a Dyck partition is. If $\lambda$ and $\mu$ are paths with $\lambda\leq \mu$ (i.e. $\lambda$ lies completely below $\mu$), a Dyck partition between $\lambda$ and $\mu$ is a partition of the region between $\lambda$ and $\mu$ into Dyck strips (or thickened Dyck paths, see \Cref{Dyck}) as in \Cref{figDyck}. 
\begin{figure}[ht]
	\begin{center}
		\begin{tabular}{c c c}
			\begin{tikzpicture}[x=\boxmini,y=\boxmini]
			\tikzset{vertex/.style={}}
			\tikzset{edge/.style={very thick}}
			\tabpath{+,-,+,-,-,-,+,+,-,+,+}
			\tabpath{-,-,+,+,-,-,+,+,-,+,-}
			\tabpathc{+,-,+,+,-,-,+,+,-,+,-,-}{blue}
			\tabpathc{-,-,+,-,-,-,+,+,+,-,+,+}{red}
			\node[blue] at (12,2) {$\mu$};
			\node[red] at (12,-2) {$\lambda$};
			\end{tikzpicture}
		\end{tabular}
	\end{center}
\caption{An example of a Dyck partition containing $5$ Dyck strips between the paths $\lambda$ at the bottom and $\mu$ at the top, respectively depicted in red and blue.}\label{figDyck}
\end{figure}

We can recover the Kazhdan--Lusztig polynomial $h_{\lambda,\mu}(v)$ by counting the number of Dyck partitions satisfying a combinatorial condition (see \Cref{TypeIandII}). More precisely, the coefficient of $v^k$ in $h_{\lambda,\mu}(v)$ is equal to the number of such Dyck partitions with $k$ strips.

The main goal of the present paper is to ``lift'' the combinatorics of Dyck partitions to the category of Grassmannian Soergel bimodules, obtaining a basis of the intersection cohomology parameterized by Dyck partitions.


The basic idea is to reinterpret each Dyck strip $D$ as a morphism of degree one (denoted by $f_D$) between the corresponding singular Soergel bimodules. A key feature of indecomposable Grassmannian Soergel bimodules is that the space of degree one morphisms between indecomposable bimodules is one-dimensional, therefore the morphism $f_D$ is actually uniquely determined up to a scalar. (In \Cref{diagramsec} we will fix a precise choice for the morphisms $f_D$.)

For an arbitrary Dyck partition $\bfP=\{D_1,D_2,\ldots,D_k\}$ we can consider the morphism $f_\bfP=f_{D_1}\circ f_{D_2}\circ \ldots \circ f_{D_k}$. Unfortunately, this morphism is not well defined: different orders of the elements in $\bfP$ (i.e. different orders for the composition of the morphisms $f_{D_i}$) may lead to different morphisms.

A crucial technical point is to define a partial order $\succ$ on the set of Dyck partitions. We then show that the morphism $f_\bfP$ is well defined, up to a scalar and up to smaller morphisms in the partial order $\succ$. As a corollary we conclude that, after fixing for any Dyck partition an order of its strips arbitrarily (with the sole condition that the order must be admissible, see \Cref{admissible}), the set $\{f_\bfP\}$ gives us a basis of the morphisms between singular Soergel bimodules. Moreover, the basis we obtain is cellular and this makes the Grassmannian Soergel bimodules a strictly object-adapted cellular category in the sense of \cite[Definition 2.4]{ELa}.
By evaluating these morphisms on the unit of the cohomology ring, we also get bases of the (equivariant) intersection cohomology of Schubert varieties.

As it turns out, the basis we obtain is highly not canonical and there does not seem to exist any canonical or distinguished way of choosing the orders of the strips in the Dyck partitions! 
Only in very simple cases, e.g. when the corresponding tableau has only two 
rows, we do get a canonical answer.

In \Cref{DegreeTwoMorphisms,RQexplicit}, we build on the theory of singular Rouquier complexes and we assume some familiarity with the results in  \cite{Pat4} where singular Rouquier complexes are studied in some detail.
Unfortunately, in this paper we have to limit ourselves  to deal with rational coefficients since singular Rouquier complexes are not yet developed over the integers, not even in the Grassmannian case. We believe, however, that most of our methods can be directly extended and that the bases we describe are also bases of the intersection cohomology with integral coefficients.

We strongly recommend to read this paper in color.


\subsection*{Acknowledgments}

This material is based upon work supported by the National Science Foundation under Grant No. DMS-1440140 while the author was in residence at the Mathematical Sciences Research Institute in Berkeley, California, during the Spring 2018 semester.
I wish to thank Ben Elias and Geordie Williamson for useful discussions.
I am also very grateful to the referee for a careful reading and many
suggestions which have significantly improved this paper.

\section{Singular Soergel bimodules}\label{chap1}

We first recall some notations about Hecke algebras from \cite[\S 3.3]{EW1} and \cite[\S 2]{W4}.

We fix $n$ and we assume that $W$ is the symmetric group $S_n$. For $1\leq k \leq n-1$, we denote by $s_k\in W$ the simple transposition $(k\;\; k+1)$. 
We denote by $S=\{s_1,\ldots, s_{n-1}\}$ denote the set of simple transpositions, so that 
 $(W,S)$ is a Coxeter system. We denote by $\leq$ the  Bruhat order and by $\ell$ the length function of $W$.

 Let $\calH$ be the Hecke algebra of $(W,S)$. We denote by $\{\bfH_{x}\}_{x \in W}$ the standard basis and by $\{\undH_x\}_{x\in W}$ the Kazhdan--Lusztig basis. Both are bases of $\calH$ as a $\bbZ[v,v^{-1}]$-module. 

If $I\subset S$, we denote by $W_I$ the subgroup of $W$ generated by $I$. We denote by $w_I$ its longest element and by $\ell(I)$ the length of $w_I$. Let $\undH_I:=\undH_{w_I}$. We denote by $W^I$ the set of minimal right coset representatives in $W/W_I$, i.e. $W^I=\{x \in W \mid xs>x$ for all $s\in I\}$.

Let $\calH^I:=\calH \undH_I$. This is a left ideal of $\calH$. For $x\in W^I$ we define $\bfH_x^I=\bfH_x\undH_I$ and $\undH_x^I=\undH_{xw_I}$. Both $\{\bfH^I_{x}\}_{x \in W^I}$ and $\{\undH^I_{x}\}_{x \in W^I}$ are bases of $\calH^I$ as a $\bbZ[v,v^{-1}]$-module, called respectively the \emph{$I$-standard} and the \emph{$I$-parabolic Kazhdan--Lusztig basis}. The \emph{$I$-parabolic Kazhdan--Lusztig polynomials} $h_{x,y}^I(v)\in v\bbN[v]$ are the coefficients of the change-of-basis matrix between these two bases, namely
\[ \undH_x^I=\bfH_x^I+\sum_{W^I\ni y<x} h_{x,y}^I(v) \bfH_y^I.\]
The \emph{inverse $I$-parabolic Kazhdan--Lusztig polynomials} $g_{x,y}^I\in v\bbN[v]$ are the coefficients of the inverse matrix (up to a sign renormalization), i.e. we have
\[ \sum_{y\in W^I}(-1)^{\ell(y)-\ell(x)} g_{x,y}^I(v) h_{y,z}^I(v) = \delta_{x,z}.\]


Let $\frh^*$ be the $(n-1)$-dimensional geometric representation of $S_n$ over $\bbQ$ (cf. \cite[\S 5.3]{Hum}), that is, $\frh^*$ is the $\bbQ$-vector space
\[\frh^*:=\left\{(x_i)_{i=1}^n\in \bbQ^n \mid \sum_{i=1}^n x_i=0\right\}\]
on which $S_n$ acts by permuting the coordinates.
%
Let $R=\Sym^\bullet(\frh^*)$ be the symmetric algebra of $\frh^*$. We regard it as a graded algebra with $\deg (\frh^*)=2$. We define $R^I$ to be the subalgebra of $W_I$-invariants in $R$. We denote by $(1)$ the grading shift. 
For $1\leq j\leq n-1$, the Demazure operator $\partial_j:R\raw R(-2)$ is defined as \[\partial_j(f)=\frac{f-s_j(f)}{x_{j}-x_{j+1}}.\]
The Demazure operators satisfy the braid relations, hence for any $x\in W$ with reduced expression $\undx=s_{i_1}s_{i_2}\ldots s_{i_k}$ we can define $\partial_x:=\partial_{i_1}\partial_{i_2}\ldots \partial_{i_k}:R\raw R(-2k)$. 

Let $I\cug J$ be subsets of $S$. The ring $R^I$ is a Frobenius extension of $R^J$ with Frobenius trace $\partial_{w_Iw_J}:R^I\raw R^J$. In other words, this means that $R^I$ is free and finitely generated as an $R^J$-module and the morphism $\partial_{w_Iw_J}:R^I\raw R^J$ is non-degenerate, i.e. 
there exist dual bases $\{x_i\}$ and $\{y_i\}$ of $R^I$ over $R^J$ such that $\partial_{w_Iw_J}(x_iy_j)=\delta_{ij}$.
The comultiplication $R^I\raw R^I\otimes_{R^J}R^I$ is the map of $R^I$-bimodules which sends $1$ to $\Delta_J^I:=\sum x_i\otimes y_i$. The element $\Delta_J^I$ does not depend on the choice of the dual bases. 

\begin{definition}
	Let $\vec{I}=(I_0,\ldots,I_{k})$ and $\vec{J}=(J_1,\ldots, J_k)$ be sequences of subsets of $S$. We say that $(\vec{I},\vec{J})$ is a \emph{translation pair} if for any $h$ we have $I_{h-1}\subseteq J_h\supseteq I_h$. 
	
	If $(\vec{I},\vec{J})$ is a translation pair let \[\ell(\vec{I},\vec{J}):=\displaystyle \sum_{i=1}^k \ell(J_i)-\ell(I_i).\]
	The \emph{generalized Bott--Samelson bimodule} $BS(\vec{I},\vec{J})$ is the graded $(R^{I_0},R^{I_k})$-bimodule 
	\[BS(\vec{I},\vec{J})=R^{I_0}\otimes_{R^{J_1}}R^{I_1}\otimes_{R^{J_2}}\ldots \otimes_{R^{J_k}}R^{I_k}(\ell(\vec{I},\vec{J})).\]
	(If $k=0$ let $BS((I),\emptyset)$ be the graded $R^{I}$-bimodule $R^I$.)
	
\end{definition}

\begin{definition}
	Let $J,I$ be subsets of $S$. We define the category ${}^J\sbim^I$ of $(J,I)$-singular Soergel bimodules as the full subcategory of $(R^J,R^I)$-bimodules whose objects are direct sums of direct summands of shifts of generalized Bott--Samelson bimodules of the form $BS(\vec{I},\vec{J})$, where $(\vec{I},\vec{J})$ is a translation pair with $I_0=J$ and $I_k=I$.

	If $J=\emptyset$ we call ${}^\emptyset\sbim^I$ the category of $I$-singular Soergel bimodules, and we denote it by $\sbim^I$.
\end{definition}

There is a duality functor $\bbD$ on ${}^J\sbim^I$ defined by $B\mapsto \Hom^\bullet_{R^J-}(B,R^J)$, where $\Hom^\bullet_{R^J-}(-,-)$ denotes the space of morphisms of left $R^J$-modules.\footnote{The choice of the shift in the duality functor is consistent with \cite[\S 3]{Pat4}, but not with \cite[\S 6.3]{W4} (see also \cite[Remark 3.4]{Pat4}).} 
The indecomposable self-dual bimodules in ${}^J\sbim^I$ are in bijection with the double cosets in $W_J\backslash W/W_I$, and we denote by ${}^JB_x^I$ the indecomposable self-dual bimodule corresponding to $x\in W_J\backslash W/W_I$.

There exists an isomorphism of $\bbZ[v,v^{-1}]$-modules
\[\cha: [{}^J\sbim^I] \raw {}^J\calH^I:= \undH_J\calH \cap \calH\undH_I.\]
Moreover, for $B\in {}^K\sbim^J$ and $B'\in {}^J\sbim^I$ we have
\[\cha(B\otimes_{R^J}B')=\cha(B)*_J\cha(B')\]
where $*_J$ is the multiplication in $\calH$ rescaled by the Poincar\'e polynomial of $J$ (cf. \cite[\S 2.3]{W4}).

If $x\in W_J\backslash W/ W_I$ is a double coset, we denote by $x_+$ (resp. $x_-$) the longest (resp. shortest) element in $x$. Let ${}^J\undH^I_x:= \undH_{x_+}$. 
The following theorem is Soergel's conjecture for singular Soergel bimodules. 

\begin{theorem}[{\cite[Conjecture 7.5.2]{W4} and  \cite[Theorem 1.1]{EW1}}]\label{sconj}
	For any $x\in W_J\backslash W/W_I$ we have	$\cha({}^JB_x^I)={}^J\undH^I_x=\undH_{x_+}$. 
\end{theorem}

The following result explains what happens when we induce an indecomposable singular Soergel bimodule to a one-sided singular bimodule.

\begin{prop}[{\cite[Prop. 7.4.3]{W4}}]\label{piusing}
	Let $x \in W_J\backslash W/W_I$ and let $w\in W^I$ be such that the coset $wW_I$ is the unique maximal right $W_I$-coset contained in $x$. Then
	\[R\otimes_{R^J}{}^JB_x^I\cong B_w^I(\ell(x_-)-\ell(w)).\]
\end{prop}

When $w$ is as in \Cref{piusing} then $ww_I$ is the maximal element in the coset $W_JwW_I$, hence 
\begin{equation}\label{maxindouble}\undH_w^I=\undH_{ww_I}= {}^J\undH_x^I\in {}^J\calH^I.\end{equation}
Going in the opposite direction, we can easily decide if an indecomposable singular Soergel bimodule can be realized as the induction of another Soergel bimodule. For $w\in W^I$ we define $S_w=\{s \in S \mid sw\leq w$ in $W/W_I\}$. (Notice that the subgroup $W_{S_w}$ is  the stabilizer of the Schubert variety corresponding to $w$.) Then, for any $J\cug S_w$, we have 
\begin{equation}\label{stabil}
B_w^I=R\otimes_{R^{J}}{}^{J}B_w^I.
\end{equation}

If $M$ is a graded free left $R^J$-module, we denote its graded rank by $\grrk(M)$ and regard it as an element of $\bbZ[v,v^{-1}]$.
\begin{theorem}[Soergel's Hom Formula for Singular Soergel Bimodules {\cite[Theorem 7.4.1]{W4}}]\label{SHF}
	Let $B_1,B_2\in {}^J\sbim^I$. Then $\Hom^\bullet(B_1,B_2)$ is a free graded left $R^J$-module and we have
	\[\grrk(\Hom^\bullet(B_1,B_2)) = \langle \bar{\cha(B_1)},\cha(B_2)\rangle,\]
	where $\langle-,-\rangle$ is the pairing in the Hecke algebra defined in \cite[\S 2.3]{W4}.
\end{theorem}

We can bundle together all the categories of $(J,I)$-singular Soergel bimodules into a single $2$-category.

\begin{definition}
	The \emph{$2$-category of singular Soergel bimodules} $\calS\sbim$ is the $2$-category whose objects are the subsets of $S$ and such that, given two subsets $J,I\cug S$, we have
	\[\Hom_{\calS\sbim}(J,I)={}^J\sbim^I.\]
\end{definition}

Reduced translating sequences \cite[Definition 1.3.1]{WPhD} are the analogue of reduced expression for double cosets of Coxeter groups. Since we are only interested in one-sided singular Soergel bimodules, we give a simpler definition which is valid in our setting.

\begin{definition}
	Let $(\vec{I},\vec{J})$ be a translation pair with $\vec{I}=(\emptyset,I_1,I_2,\ldots,I_k)$. Let 
	$v_1=w_{J_1}w_{I_1}$ and $v_{h}=v_{h-1}w_{J_h}w_{I_h}$ for every $h\leq k$.
	We call $v_k$ the \emph{end-point} of $(\vec{I},\vec{J})$.

	We say that $(\vec{I},\vec{J})$ is \emph{reduced} if for every $h<k$ we have 
	\[\ell(v_hw_{J_{h+1}}) = \ell(v_h)+\ell(w_{J_{h+1}}).\]
\end{definition}

If $(\vec{I},\vec{J})$ is a reduced translation pair with end-point $w\in W^I$, then $B_w^I$ is a direct summand of $BS(\vec{I},\vec{J})$ with multiplicity $1$.

\begin{remark}
	Our main results in this paper are in type A, so we decided to
	only discuss this case in our exposition, However, the results in this section remain valid for an arbitrary Coxeter system $(W,S)$ under the assumption that all the subsets $I\cug S$ considered are finitary, i.e. that $W_I$ is a finite subgroup. 
\end{remark}

\subsection{Diagrammatic singular Soergel bimodules}

There is an alternative and powerful approach (discussed in \cite{ESW,E3,EL}) to the $2$-category of singular Soergel bimodules via diagrams, which we recall in this section.

\begin{definition}
A \emph{singular $S$-diagrams} is a finite collection of oriented $1$-manifolds with boundary, colored by elements of $S$, embedded in the strip $\bbR\times [0,1]$ and whose boundary is embedded in $\bbR\times\{0,1\}$, the boundary of the strip.

Two $1$-manifolds associated to colors $s\neq t$ can intersect only transversely and cannot intersect on the boundary of the strip. Moreover, two $1$-manifolds associated to the same color cannot intersect at all.

We refer to connected components of the complement of the $1$-manifolds in the strip as \emph{regions}. Each region is labeled by a subset $I\cug S$ compatibly with the following rule:
if two subsets are on two regions bordering the same $1$-manifold labeled by $s$, then they differ by the single element $\{s\}$ with the one on the right side of the $1$-manifold (while looking in the direction of its orientation) being larger. 

Moreover, a region with label $I$ can be decorated by a polynomial $f\in R^I$.
\begin{example}
Let $S=\{{\color{red} r},{\color{blue} b},{\color{green} g},{\color{violet}v}\}$. The following is an example of a singular $S$-diagram where we have labeled some of the regions in gray. In this example $f_1\in R^{v}$ and $f_2\in R^{\{r,g,b\}}$.
\begin{center}
 \begin{tikzpicture}[scale=0.4] 
 \draw (12,0) to (-4,0);
 \draw(12,8) to (-4,8);
\draw[thick,red, -triangle 45] (10,0) .. controls (10,5) and (0,0) .. (0,8);
\draw[thick,blue, -triangle 45] (1,0) to [out=90,in=-135] (2,4) to [out=45,in=-90] (5,8);
\draw[thick,green, -triangle 45, bend right] (4,0) to (3,8);
\draw[thick,violet,-triangle 45] (8,0) to [out=90,in=0] (5,7) to [out=180,in=90] (3,0);
\node[gray] at (10,4) {\scriptsize $\{r,b,g,v\}$};
\node[gray] at (6.5,3.5) {\scriptsize $\{r,b,g\}$};
\node[gray] at (6,1.5) {\scriptsize $\{g,b\}$};
\node[gray] at (4,2) {\scriptsize $\{b\}$};
\node[gray] at (2,2) {\scriptsize $\{v,b\}$};
\node[gray] at (-1,3) {\scriptsize $\{v\}$};

\node at (-1,6) {$f_1$};
\node at (6,5) {$f_2$};
\end{tikzpicture}
\end{center}

\end{example}

To a singular $S$-diagram we associate a bottom sequence and a top sequence of subsets, respectively given by the sequence of subsets appearing on the lower and on the upper boundary of the strip.
\end{definition} 

The only sequences of sets which can appear as bottom (or top) sequence of some $S$-diagram are those in which two consecutive sets in the sequence differ exactly by a single element of $S$. We call these sequences \emph{strict translation sequences}.

Singular $S$-diagrams will always be considered up to isotopy. If the top sequence of a diagram $\calD_1$ coincides with the bottom sequence of a diagram $\calD_2$ then we can stack $\calD_2$ on top of $\calD_1$ to obtain a new diagram which denoted by $\calD_2\circ \calD_1$.

Similarly, if the last element of the bottom (and top) sequence of $\calD_1$ coincides with the first element of the bottom sequence of $\calD_2$ we can stack horizontally two diagrams $\calD_1$ and $\calD_2$ to obtain a new diagram which denoted by $\calD_1\otimes \calD_2$.

Each $S$-diagram can be obtained, up to isotopy, by stacking together the ``building boxes'' of \Cref{boxes}. In every box $J$ denotes the largest subset and $I$ the smallest.
The degree of the building boxes is defined as in the third column of \Cref{boxes}.
We define the degree of a diagram to be the sum of the degree of its building boxes.

\begin{table}[ht]
		\begin{tabularx}{\textwidth}{c c c}
			$\frD$ & $\calS\sbim$ & degree \\ 
			\hline
			\begin{minipage}{2cm}
				\begin{tikzpicture}
					\draw (-1,0) rectangle (1,2);
					\draw[thick,red, -triangle 45] (0,0) to (0,2);
					\node at (0.5,1) {J};
					\node at (-0.5,1) {I};
					\node[red] at (0.2,0.2) {$s$};
				\end{tikzpicture} 
			\end{minipage}	
			& $R^I\otimes_{R^J}R^J(\eta)\xra{id}R^I\otimes_{R^J}R^J(\eta)$ & 0\\
			\begin{minipage}{2cm}				
				\begin{tikzpicture}
					\draw (-1,0) rectangle (1,2);
					\draw[thick,red, triangle 45-] (0,0) to (0,2);
					\node at (0.5,1) {I};
					\node at (-0.5,1) {J};
					\node[red] at (0.2,1.8) {$s$};
				\end{tikzpicture}
			\end{minipage}
			& ${}_JR^I\xra{id}{}_JR^I$ & 0\\ 
 			\begin{minipage}{2cm}
 				\begin{tikzpicture}
					\draw (-1,0) rectangle (1,2);
					\draw[thick,red, triangle 45-] (0.5,0) to [out=90,in=0] (0,1) to [out=180,in=90] (-0.5,0);
					\node at (0,1.5) {I};
					\node at (0,0.25) {J};
					\node[red] at (-0.7,0.2) {$s$};
				\end{tikzpicture}
			\end{minipage}	
			& 
			\begin{tabular}{c}
				$f_{\carcu}:R^I\otimes_{R^J} R^I(\eta) \raw R^I$\\ 
				$f\otimes g\mapsto fg$
			\end{tabular}
			& $\ell(J)-\ell(I)$\\
			\begin{minipage}{2cm}
				\begin{tikzpicture}
					\draw (-1,0) rectangle (1,2);
					\draw[thick,red, -triangle 45] (0.5,2) to [out=-90,in=0] (0,1) to [out=180,in=-90] (-0.5,2);
					\node at (0,1.75) {J};
					\node at (0,0.5) {I};
					\node[red] at (0.7,1.8) {$s$};
				\end{tikzpicture} 
			\end{minipage}			
			&
			\begin{tabular}{c}
				$f_{\carcd}:R^I\raw R^I\otimes_{R^J} R^I(\eta) $\\ 
				$1\mapsto \Delta_J^I$
			\end{tabular} 
			& $\ell(J)-\ell(I)$\\
			\begin{minipage}{2cm} 
 				\begin{tikzpicture}
					\draw (-1,0) rectangle (1,2);
					\draw[thick,red, -triangle 45] (0.5,0) to [out=90,in=0] (0,1) to [out=180,in=90] (-0.5,0);
					\node at (0,1.5) {J};
					\node at (0,0.25) {I};
					\node[red] at (.7,0.2) {$s$};
				\end{tikzpicture} 
			\end{minipage}				
			&
			\begin{tabular}{c}
				$f_{\carau}:{}_J R^I_J(\eta)\raw R^J$\\
				$f\mapsto \partial_J^I(f)$
			\end{tabular}
			& $\ell(I)-\ell(J)$\\
			\begin{minipage}{2cm} 			
	 			\begin{tikzpicture}
					\draw (-1,0) rectangle (1,2);
					\draw[thick,red, triangle 45-] (0.5,2) to [out=-90,in=0] (0,1) to [out=180,in=-90] (-0.5,2);	
					\node at (0,1.75) {I};
					\node at (0,0.5) {J};
					\node[red] at (-0.7,1.8) {$s$};
				\end{tikzpicture}
			\end{minipage}
			 &
			\begin{tabular}{c}
				$f_{\carad}:R^J\raw {}_J R^I_J(\eta)$\\
				$f\mapsto f$
			\end{tabular} 
			& $\ell(I)-\ell(J)$\\
			\begin{minipage}{2cm}
			 	\begin{tikzpicture}
					\draw (-1,0) rectangle (1,2);
					\draw[thick,red, -triangle 45] (0.76,0) to (-0.75,2);
					\draw[thick,blue, triangle 45-] (0.75,2) to (-0.75,0);
					\node at (0.5,1) {J};
					\node at (-0.5,1) {I};
					\node at (0,0.5) {K};
					\node at (0,1.5) {L};
					\node[red] at (0.8,0.3) {$s$};
					\node[blue] at (-0.8,0.3) {$t$};
				\end{tikzpicture}
			\end{minipage}
 			&
			\begin{tabular}{c}
				 $f_{\duarrow}:R^I\otimes_{R^K} R^K\otimes_{R^J}R^J(\eta)\raw R^I\otimes_{R^L} R^L\otimes_{R^J}R^J(\eta)$\\
				$f\otimes 1\otimes 1\mapsto f\otimes 1\otimes 1$
			\end{tabular} 
			& 0
		\end{tabularx}
		\caption{The building boxes.}\label{boxes}
	\end{table}

\begin{definition}
We define the $2$-category $\frD$ as follows: the objects are subsets of $S$, $1$-morphisms between $I$ and $J$ are the strict translation sequences starting in $I$ and ending $J$, and $2$-morphisms between two sequences $K_1$ and $K_2$ are isotopy classes of singular $S$-diagrams with bottom sequence $K_1$ and top sequence $K_2$.
\end{definition}

Now we define a $2$-functor $\Xi:\frD\raw \calS\sbim$. The functor $\Xi$ is the identity on objects. On $1$-morphisms $\Xi$ sends a strict translation sequence to the corresponding generalized Bott--Samelson bimodule, i.e. $\Xi(I)=R^I$ and 
\[\Xi(I_1,\ldots,I_k,I_{k+1})=\begin{cases} \Xi(I_1,\ldots,I_k)\otimes_{R^{I_{k}}}R^{I_{k+1}}(\ell(I_k)-\ell(I_{k+1}))&\text{if }I_{k+1}\subset I_{k}\\
\Xi(I_1,\ldots,I_k)\otimes_{R^{I_{k+1}}}R^{I_{k+1}}&\text{if }I_{k}\subset I_{k+1}.\end{cases}\]
To define $\Xi$ on $2$-morphisms we just need to specify the image of the generating $2$-morphisms. This is done in the second column of \Cref{boxes}.

\begin{theorem}[{\cite{ESW}}]
The $2$-functor $\Xi:\frD\raw \calS\sbim$ is well-defined and it is full on $2$-morphisms.
\end{theorem}

The functor $\Xi$ is not an equivalence of $2$-category, even after taking the Karoubi envelope of $\frD$. To achieve an equivalence, one would need to impose several relations on the category $\frD$. Some of these relations are general for any \emph{hypercube} of Frobenius extensions and are discussed in \cite{ESW}. In (finite or affine) type A the full list of relations that we need to add to $\frD$ to obtain an equivalence with $\calS\sbim$ are described in \cite[\S 2.5]{EL} (at the time of writing a proof has not been yet published).
In this paper we will not actually need any of this as we will only make use of the diagrammatics since they provide an insightful and convenient way to draw and describe morphisms between singular Soergel bimodules.

The following definition is the diagrammatic counterpart of the duality functor $\bbD$.
\begin{definition}
	Given a $S$-diagram $\calD$ we define $\bar{\calD}$ the $S$-diagram obtaining by reflecting $\calD$ over the line $\{y=\frac12\}$ and then changing the orientation of all the $1$-manifolds.
	
	We say that $\bar{\calD}$ is the \emph{flip (upside down)} of $\calD$.

\begin{figure}[h]
\begin{center}
\begin{tikzpicture}
\draw (-1,0) rectangle (2.5,2);
\draw[thick,red, -triangle 45] (0.76,0) to (-0.75,2);
\draw[thick,blue, -triangle 45] (-0.75,0) to [in=219] (0.5,1.5) to [out=39,in= 180] (1.2,1.75) to [out=0,in=90] (2,0);
\node at (1,1) {J};
\node at (-0.5,1) {I};
\node at (0,0.5) {K};
\node at (0,1.5) {L};
\node[red] at (0.8,0.3) {$s$};
\node[blue] at (-0.8,0.3) {$t$};
\node at (-1.8,1) {$\calD=$};

\begin{scope}[xshift=6cm, yshift=2cm, yscale = -1]
\node at (-2,1) {$\rightsquigarrow \bar{\calD}=$};
\draw (-1,0) rectangle (2.5,2);
\draw[thick,red, triangle 45-] (0.76,0) to (-0.75,2);
\draw[thick,blue, triangle 45-] (-0.75,0) to [in=219] (0.5,1.5) to [out=39,in= 180] (1.2,1.75) to [out=0,in=90] (2,0);
\node at (1,1) {J};
\node at (-0.5,1) {I};
\node at (0,0.5) {K};
\node at (0,1.5) {L};
\node[red] at (0.8,0.3) {$s$};
\node[blue] at (-0.8,0.3) {$t$};
\end{scope}
\end{tikzpicture}
\caption{An example of a flip.}
\end{center}
\end{figure}
\end{definition}

If $f$ is a $2$-morphism in $\calS\sbim$, there exists a $S$-diagram $\calD$ such that $\Xi(\calD)=f$.
Then we can define the flip of $f$ as $\bar{f}:=\Xi(\bar{\calD})$. 

\begin{lemma}
	The flip $\bar{f}$ of a $2$-morphism $f$ is well defined.
\end{lemma}
\begin{proof}
We need to show that if $\calD'$ is another $S$-diagram such that $\Xi(\calD')=f$, then we also have $\Xi(\bar{\calD})=\Xi(\bar{\calD'})$.
We can see this in two steps by viewing flip along the horizontal axis as the composition of a flip along the vertical axis and a rotation by $180^\circ$.
First notice that if we flip a diagram $\calD$ along the vertical axis, then the corresponding morphism is $\Phi\circ \Xi(\calD) \circ \Phi$ where $\Phi$, is the auto-equivalence of $\calS\sbim$ which swaps left and right (i.e., it sends ${}^J\sbim^I$ to ${}^I\sbim^J$). 

Then, using the Frobenius relations as in the picture below, we deduce that if $\Xi(\calD)=\Xi(\calD')$, then also $\Xi(\rotatebox[origin=c]{180}{$\calD$})=\Xi(\rotatebox[origin=c]{180}{$\calD'$})$ (here $\rotatebox[origin=c]{180}{$\calD$}$ denotes the rotated diagram).
\[
\begin{tikzpicture}[scale=0.7]
\draw (-1,0) rectangle (1,1);
\draw (-0.8,0) to (-0.8,-0.2) to [out=-90,in=0] (-1,-0.4) to [out=180,in=-90]  (-1.2,-0.2) to (-1.2,2.3);
\draw (-0.5,0) to (-0.5,-0.2) to [out=-90,in=0] (-1,-0.7) to [out=180,in=-90]  (-1.5,-0.2) to (-1.5,2.3);
\draw (.8,0) to (.8,-0.2) to [out=-90,in=0] (-1,-1.2) to [out=180,in=-90]  (-2.5,-0.2) to (-2.5,2.3);
\node at (0,0.5) {\Large $\calD$};
\node at (0.15,1.15) {$\cdots$};
\node at (-2,0.5) {$\cdots$};
\node at (3.2,0.5) {$=$};
\begin{scope}[yscale=-1,xscale=-1,yshift=-1cm]
\draw (-0.8,0) to (-0.8,-0.2) to [out=-90,in=0] (-1,-0.4) to [out=180,in=-90]  (-1.2,-0.2) to (-1.2,2.3);
\draw (.5,0) to (.5,-0.1) to [out=-90,in=0] (-1,-0.9) to [out=180,in=-90]  (-2.2,-0.1) to (-2.2,2.3);
\draw (.8,0) to (.8,-0.2) to [out=-90,in=0] (-1,-1.2) to [out=180,in=-90]  (-2.5,-0.2) to (-2.5,2.3);
\node at (-0.15,1.2) {$\cdots$};
\node at (-1.7,0.5) {$\cdots$};
\end{scope}
\begin{scope}[xshift= 5cm]
\draw (-1,0) rectangle (1,1);
\draw (-0.8,0) to (-0.8,-1.3);
\draw (0.5,0) to (0.5,-1.3);
\draw (.8,0) to (.8,-1.3);
\draw (-0.8,1) to (-0.8,2.3);
\draw (0.5,1) to (0.5,2.3);
\draw (.8,1) to (.8,2.3);
\node at (0,0.5) {\Large \rotatebox{180}{$\calD$}};
\node at (-0.15,1.3) {$\cdots$};
\node at (-0.15,-0.3) {$\cdots$};
\end{scope}
\end{tikzpicture}\qedhere\]
\end{proof}

In \Cref{adjoint} we show that taking the flip of a morphism amounts to taking the adjoint with respect to the intersection forms on generalized Bott--Samelson bimodules.

\section{Singular Soergel calculus for Grassmannians}\label{Grass}
In this chapter we specialize to the case of the complex Grassmannians, that is the case when $I$ is a maximal proper subset of $S$.
We fix now once for all $i$ with $1\leq i \leq n-1$ and $I=S\setminus\{s_i\}$. We have $W_I\cong S_{i}\times S_{n-i}\cug S_n=W$.

The Grassmannian $\Gr(i,n)$ is a smooth complex projective variety, isomorphic to $SL_n(\bbC)/P$, where $P$ is the maximal parabolic subgroup of $SL_n(\bbC)$ corresponding to $I$:
\[ P = \left\{ \begin{pmatrix}
A & B \\ 0 & C
\end{pmatrix}\in SL_n(\bbQ) \mid A\in M_{i\times i}(\bbQ),\; B\in M_{n-i\times i}(\bbQ),\; C \in M_{n-i\times n-1}(\bbQ) \right\}.\]
Let $T\cug SL_n(\bbC)$ be the maximal torus of diagonal matrices. The $T$-equivariant cohomology of $\Gr(i,n)$ (cf. \cite[Prop. 1]{Brion}) is
\[H_T^\bullet(\Gr(i,n),\bbQ)\cong R\otimes_{R^{W}}R^{I}.\]

Let $\calD^b_T(\Gr(i,n),\bbQ)$ denote the $T$-equivariant derived category of constructible sheaves of $\bbQ$-vector spaces on $\Gr(i,n)$. If $\calF\in \calD^b_T(\Gr(i,n))$, the cohomology $H^\bullet(\Gr(i,n),\calF)$ is in a natural way a module over $H^\bullet_T(\Gr(i,n),\bbQ)$, hence in particular a $(R,R^I)$-bimodule.

 Let $\calK$ be the full subcategory of $\calD^b_T(\Gr(i,n))$ whose objects are direct sums of shifts of simple $T$-equivariant $\bbQ$-perverse sheaves on $\Gr(i,n)$. Then the cohomology induces an equivalence $H^\bullet:\calK\raw \sbim^I$ 	(cf. \cite[Erweiterungssatz 5]{S1} and \cite[Theorem (1.1)]{Gi}). If $X_w\cug \Gr(i,n)$ denotes the Schubert variety associated to $w\in W^I$ and $IC(X_w,\bbQ)\in \calK$ the intersection cohomology sheaf of $X_w$, then in particular 
$IH^\bullet_T(X_w,\bbQ):=H^\bullet(\Gr(i,n),IC(X_w,\bbQ))$ is isomorphic to the indecomposable Soergel bimodule $B_w^I$ as a $(R,R^I)$-bimodule.

\subsection{Paths}

Consider the following set of $n$-tuples:
\[\Lambda_{n,i}:=\{\lambda=(\lambda_1,\ldots,\lambda_n) \in \{\upath,\dpath\}^n\mid \#\{k \mid \lambda_k = \dpath\}=i\}\] 
The natural action of $W=S_n$ on $n$-tuples induces a bijection
\[W^I\longleftrightarrow \Lambda_{n,i}\]
\[w\mapsto \lambda^w=(\lambda^w_1,\lambda^w_2,\ldots,\lambda^w_n):=w(\dpath,\dpath,\ldots,\dpath,\upath,\upath,\ldots,\upath)\]

\begin{definition}\label{pathdef}
By \emph{path} we mean a piecewise linear path in $\bbR^2$ which is union of segments of the form $(x,y)\rightarrow (x+1,y+1)$ or $(x,y)\rightarrow (x+1,y-1)$, for $x,y\in \bbZ$.
	Every element $\lambda\in \Lambda_{n,i}$ can be thought as a path from $(0,i)$ to $(n,n-i)$, where to every $\upath$ in $\lambda$ it corresponds a segment $(x,y)\rightarrow (x+1,y+1)$ and to every $\dpath$ corresponds a segment $(x,y)\rightarrow (x+1,y-1)$. Hence, we can identify $\Lambda_{n,i}$ with the set of paths from $(0,i)$ to $(n,n-i)$. By a slight abuse of terminology we also call the elements in $\Lambda_{n,i}$ paths.
\end{definition}

In the following we will often identify an element $w\in W^I$ with its corresponding path $\lambda^w\in \Lambda_{n,i}$. For instance, we will denote an indecomposable singular Soergel bimodule by $B_w^I$ or by $B_{\lambda^w}^I$ indistinctly.

We can deduce many properties of an element $w\in W^I$ directly in terms of the corresponding path $\lambda^w$. For example, the length $\ell(w)$ is half the area of the region between $\lambda^w$ and $\lambda^{id}$.

\begin{definition}
Let $j$ be an integer with $0\leq j\leq n$ and $\lambda$ be a path. The \emph{height} of $\lambda$ at $j$ is the integer $\hgt_j(\lambda)$ such that the point $(j,\hgt_j(\lambda))$ belongs to the path $\lambda$.
\end{definition}

We have $v\leq w$ if and only if the path $\lambda^v$ lies completely below $\lambda^w$, i.e. if $\hgt_j(\lambda^v)\leq \hgt_j(\lambda^w)$ for every $j$.
 In this case we will simply write $\lambda^v\leq \lambda^w$.
 If $\lambda$ and $\mu$ are paths with $\lambda\leq \mu$ we denote by $\calA(\lambda,\mu)$ the region of the plane delimited by $\lambda$ and $\mu$. 
 If $\lambda=\lambda^{id}$ we denote $\calA(\lambda^{id},\mu)$ simply by $\calA(\mu)$.

From the path representation $\lambda^w$ of an element $w$ it is easy to recover its left descent set. We have:
\begin{itemize}
\item $s_jw\in W^I$ and $\ell(s_jw)=\ell(w)-1$ if $j$ is a \emph{peak} for $\lambda^w$, i.e. if $(\lambda^w_j,\lambda^w_{j+1})=(\upath,\dpath)$, 
\item $s_jw\in W^I$ and $\ell(s_jw)=\ell(w)+1$ if $j$ is a \emph{valley} for $\lambda^w$, i.e. if $(\lambda^w_j,\lambda^w_{j+1})=(\dpath,\upath)$, 
\item $s_jw\not\in W^I$ and $s_jwW_I=wW_I$ if $j$ is on a \emph{slope} of $\lambda^w$, i.e. if $(\lambda^w_j,\lambda^w_{j+1})=(\upath,\upath)$ or $(\lambda^w_j,\lambda^w_{j+1})=(\dpath,\dpath)$.
\end{itemize}

\begin{definition}\label{boxdef}
By \emph{box} we mean a square rotated by $45^\circ$ with side length $\sqrt{2}$ and whose center $(x,y)$ is a point with integral coordinates such that $x+y+i$ is odd.

 We fill the region $\calA(\lambda^w)$ with boxes and we label each boxes by the horizontal coordinate of its highest point as in \Cref{figfill}.
\end{definition}

\begin{figure}[H]
	\begin{center}
		\begin{tikzpicture}[x=\boxmini,y=\boxmini]
		\tikzset{vertex/.style={}}
		\tikzset{edge/.style={very thick}}
		\tabpath{+,-,+,-,+,+,-,-}
		\tikzset{edge/.style={}}
		\tabpath{-,-,-,-,+,+,+,+}
		\tikzset{edge/.style={dotted}}
		\tabpath{-,+,-,+,+,-,-}
		\tabpath{-,-,+,+,-,+,+}
		\tabpath{-,-,-,+,+,-}
		\tabpath{-,-,+,-,-}
		\node at (1,0) {$1$};
		\node at (2,-1) {$2$};
		\node at (3,-2) {$3$};
		\node at (3,0) {$3$};
		\node at (4,-3) {$4$};
		\node at (4,-1) {$4$};
		\node at (5,-2) {$5$};
		\node at (5,0) {$5$};
		\node at (6,-1) {$6$};
		\node at (6,1) {$6$};
		\node at (7,0) {$7$};
		\node at (7.5,2) {$\lambda^w$};
		\node at (-1,0) {\scriptsize $(0,4)$};
		\node at (9,0) {\scriptsize $(8,4)$};
		\end{tikzpicture}
		\caption{The thick line denotes the path $\lambda^w$ for $w=s_6s_1s_3s_5s_7s_2s_4s_6s_3s_5s_4\in S_8$. From the picture we immediately see that the set of peaks for $\lambda^w$ is $\{1,3,6\}$, the set of valleys is $\{2,4\}$ and the set of slopes is $\{5,7\}$.  }\label{figfill}
	\end{center}
\end{figure}

Any order in which we can remove a box from $\calA(\lambda^w)$ so that at any step the upper boundary of the remaining region is still a path in $\Lambda_{n,i}$ gives rise to a reduced expression. Moreover, all the reduced expressions of $w$ arise in this way.

From the description above of reduced expression it is easy to see that no reduced expression can contain a subword of the form $s_js_{j+1}s_j$.
In particular, the function $\rex_w:\{1,2,\dots,\}\raw \bbN$ given by 
\[\rex_w(j)= \left\lvert\left\{\begin{array}{c|c}\multirow{2}{*}{$k$} & i_k=j\text{ for a reduced expression }\\ &s_{i_1}s_{i_2}\ldots s_{i_k}\text{ of }w \end{array}\right\}\right\rvert
=\left\lvert\left\{\begin{array}{c}\text{boxes labeled by $j$ between} \\ \lambda^w\text{ and }\lambda^{id}\end{array}\right\}\right\rvert\]
is well-defined.
We have 
\[ \hgt_{\lambda^w}(j)=2\rex_j(w) + |i-j|.\]
Notice that for $v,w\in W^I$ we have $v\geq w$ if and only if $\rex_v(j)\geq \rex_w(j)$ for all $j$.

Moreover, if $v\geq w$ there exists $x\in W$ such that $v=xw$ and $\ell(v)=\ell(x)+\ell(w)$, therefore the Bruhat order on $W^I$ is generated by $sx>x$ with $s\in S$ and $\ell(sx)>\ell(x)$.

To simplify the notation, we will often use $J\cug \{1,2,\ldots,n\}$ to denote the set $\{s_j \mid j \in J\}\subset S$, e.g. we use $W^J$ and $W_J$ to denote $W^{\{s_j\mid j\in J\}}$ and $W_{\{s_j\mid j\in J\}}$ respectively.

We have now all the tools to prove the following simple Lemmas.

\begin{lemma}\label{valley}
	Assume that $w\in W^I$ has a valley in $j$ and that $\lambda^w_a=\ldots=\lambda^w_j=\dpath$ and $\lambda^w_{j+1}=\ldots=\lambda^w_b=\upath$. Assume that $v<w$.
	\begin{enumerate}[i)]
		\item If $x\in W_{[a,b-1]}$, then $xv\not\geq w$.
		\item Let $\hat{J}=[a,b-1]\setminus \{j\}$. Then, if $x\in W_{[a,b-1]}^{\hat{J}}$ we have $xw\in W^I$ and $\ell(xw)=\ell(x)+\ell(w)$.
	\end{enumerate}
\end{lemma}

Roughly, the second point is stating that we can stack a ``small tableau'' on top of $\lambda^w$ to obtain the tableau of $\lambda^{xw}$.

\begin{figure}[H]
	\begin{center}
		\begin{tikzpicture}[x=\boxmini,y=\boxmini]
		\tikzset{vertex/.style={}}
		\tikzset{edge/.style={very thick}}		
		\tabpath{+,+,-,-,+,+,+,-,-}
		\tabpathshift{+,-,+,-}{-8}{2}
		\tikzset{edge/.style={}}
		\tabpath{-,-,-,-,+,+,+,+,+}
		\tabpathshift{-,-,+,+}{-8}{2}
		\tikzset{edge/.style={dotted}}
		\tabpath{-,+,-,+,+,-,-}
		\tabpath{-,-,+,+,-,+,+}
		\tabpath{-,-,-,+,+,-}
		\tabpath{-,-,+,-,-}
		\tabpath{+,-,+,-,+,+,-,-}
		\tabpath{+,-,+,-,+,+,-,+}
		\tabpathshift{-,+,-,+}{-8}{2}
		\node at (1,0) {$1$};
		\node at (2,-1) {$2$};
		\node at (2,1) {$2$};
		\node at (3,-2) {$3$};
		\node at (3,0) {$3$};
		\node at (4,-3) {$4$};
		\node at (4,-1) {$4$};
		\node at (5,-2) {$5$};
		\node at (5,0) {$5$};
		\node at (6,-1) {$6$};
		\node at (6,1) {$6$};
		\node at (7,0) {$7$};
		\node at (7,2) {$7$};
		\node at (8,1) {$8$};
		\node at (9.1,2) {$\lambda^w$};
		\node at (-6,1) {$4$};
		\node at (-5,2) {$5$};
		\node at (-7,2) {$3$};
		\node at (-8.5,2.7) {$\lambda^x_{\hat{J}}$};
		\draw[->,very thick] (-4,3) to [out=20,in=180] (0,4) to [out=0,in=110] (4,2) ;
		\end{tikzpicture}
		\caption{In this example $\lambda^w$ has a  valley in $4$ and we have $a=3$ and $b=7$. So we can stack a tableau corresponding to an element $x\in W^{\{3,5,6,7\}}_{[3,7]}$ on top of it.}
	\end{center}
\end{figure}

\begin{proof}
	Let $\lambda=\lambda^w$ and $\mu=\lambda^v$. 
	Since $\mu<\lambda$ there exists $k$ with $\hgt_k(\mu)<\hgt_k(\lambda)$.
	Actually, we can always find such a $k$ with $k<a$ or $k\geq b$. 
	Otherwise, we have $\hgt_{a-1}(\lambda)=\hgt_{a-1}(\mu)$ and $\hgt_{b}(\lambda)=\hgt_{b}(\mu)$. Then, since $\lambda^w_a=\ldots=\lambda^w_j=\dpath$ and $\lambda^w_{j+1}=\ldots=\lambda^w_b=\upath$ we also have
	$\hgt_k(\mu)=\hgt_k(\lambda)$ for all $k\in [a,b-1]$.

	Let now $k$ with $k<a$ or $k\geq b$ be such that 
	$\hgt_k(\mu)<\hgt_k(\lambda)$.
	 Since $\rex_k(w)>\rex_k(v)=\rex_k(xv)$ we have $xv\not\geq w$.

	ii) To any element $x\in W_{[a,b-1]}^{\hat{J}}$ we can associate a path $\lambda_{\hat{J}}^x\in \Lambda_{b-a,j-a}$.	
	We claim that the path $\lambda^{xw}\in \Lambda_{n,i}$ corresponding to $xw$ can be obtained by ``stacking'' the path $\lambda^x_{\hat{J}}$ on top of $\lambda^w$. In formulas, we have 
	\[\lambda^{xw}_k=\begin{cases}
	\lambda^w_k &\text{ if }k< a\text{ or }k\geq b-1\\
	(\lambda^w_{\hat{J}})_{k-a} & \text{ if }a\leq k\leq b.
	\end{cases}.\]
	If $x$ is a simple reflection (i.e. if $x=s_j$) the claim is clear. If $\ell(x)>1$ then we can write $x=s_kx'$ with $x'\in W^{\hat{J}}_{[a,b-1]}$ and $\ell(x')=\ell(x)-1$. By induction on the length we have $\ell(x'w)=\ell(x')+\ell(w)$, and $k$ is a valley for $\lambda^{x'w}$ since it is a valley for $\lambda_{\hat{J}}^{x'}$. We conclude that $\ell(s_kx'w)=\ell(x')+\ell(w)+1$.
\end{proof}

\begin{lemma}\label{crucial}
	
	Let $w\in W^I$ and let $a\leq j<b$ such that $\lambda^w_a=\ldots=\lambda
	^w_j=\dpath$ and $\lambda^w_{j+1}=\ldots=\lambda^w_b=\upath$.
	Let $\hat{J}=[a,b-1]\setminus \{j\}$ and let $x\in W_{[a,b-1]}^{\hat{J}}$. 
	
	Consider the $(R,R^I)$-bimodule $B:=B_x^{\hat{J}}\otimes_{R^{\hat{J}}}{}^{\hat{J}}B_w^I$ and assume $B_z^I(m_z)$ is a direct summand of $B$. 
	If $z\geq w$, then $z=xw$ and $m_z=0$.
\end{lemma} 
\begin{proof}
	By Soergel's conjecture (\Cref{sconj}), it is sufficient to prove the correspondent statement in the Hecke algebra. We can write
	\begin{equation}\label{Hdec}
	\cha(B)=\undH_x^{\hat{J}}*_{\hat{J}}{}^{\hat{J}}\undH^I_w=\sum_z p_z(v) \undH_z^I
	\end{equation}
	for some $p_z(v)\in \bbZ[v,v^{-1}]$. 
	Assume there exists $z\geq w$ such that $p_z(v)\neq 0$. Since $B$ is self-dual, we have $p_z(v)=p_z(v^{-1})$. So, if we write $\cha(B)$ in the $I$-standard basis of $\calH^I$, the term 
	\[\tau:=v^{-\deg p_z(v)}\bfH_z^I\]
	 must occur with a positive coefficient.

	Note that $k$ is on a slope of $\lambda^w$ for all $k\in {\hat{J}}$. It follows that for any $u\in W_{\hat{J}}$ we have $uwW_I=wW_I$, and so $W_{\hat{J}}wW_I=wW_I$. By 	\eqref{maxindouble}, this implies $\undH_w^I={}^{\hat{J}}\undH_w^I$ and, in particular, $\undH_w^I\in \undH_{\hat{J}}\calH$.
	We can write
	\[\undH_x^{\hat{J}}=\bfH_x^{\hat{J}}+\sum_{\substack{r\in W^{\hat{J}}\\r<x}}h^{\hat{J}}_{r,x}(v)\bfH_r^{\hat{J}}=
	\bfH_x\undH_{{\hat{J}}}+\sum_{\substack{r\in W^{\hat{J}}\\r<x}}h^{\hat{J}}_{r,x}(v)\bfH_r\undH_{{\hat{J}}}.\]
	Since $\undH_w^I\in \undH_{\hat{J}}\calH$ and $\undH_{\hat{J}}*_{\hat{J}} \undH_{\hat{J}}=\undH_{\hat{J}}$, we have
	\begin{multline}\label{Hdec2}
	\undH_x^{\hat{J}}*_{\hat{J}}\undH^I_w=\bfH_x\undH_w^I+\sum_{\substack{r\in W^{\hat{J}} \\ r<x}}h_{r,x}^{\hat{J}}(v)\bfH_r\undH_w^I=\\ =\bfH_x\bfH_w^I+\sum_{\substack{r\in W^{\hat{J}}\\r<x}}h_{r,x}^{\hat{J}}(v)\bfH_r\bfH_w^I+\sum_{\substack{t\in W^I\\t<w}}h_{t,w}^I(v)\bfH_x\bfH_t^I+\sum_{ \substack{r\in W^{\hat{J}},\;t\in W^I\\r<x,\;t<w}}h_{r,x}^{\hat{J}}(v)h_{t,w}^I(v)\bfH_r\bfH_t^I.
	\end{multline}
	
	Recall that for any $r<x$ we have $h^{\hat{J}}_{r,x}(v)\in v\bbN[v]$. Using \Cref{valley}(ii) we can see that $\bfH_r\bfH_w^I=\bfH_{rw}^I$ for every $r<x$. Hence the first sum in RHS of \eqref{Hdec2} does not contain the term $\tau$.

	By \Cref{valley}(i) if $t<w$ we have $rt\not\geq w$ for any $r\in W^{\hat{J}}_{[a,b-1]}$. It follows that the term $\tau$ cannot occur neither in the second or in the third sum of \eqref{Hdec2}. The only remaining possibility is $\tau=\bfH_z^I=\bfH_x\bfH_w^I$, hence $z=xw$ and $\deg p_z(v)=0$.
\end{proof}

\subsection{Dyck strips}
Dyck paths are a very classical combinatorial object studied, for example, for their connection to Catalan numbers. 
In this section we recall the main result from \cite{SZJ} which allows us to express Kazhdan--Lusztig polynomials in terms of Dyck paths.

\begin{definition}\label{Dyck}
A \emph{Dyck path} is a path (in the sense of \Cref{pathdef}) from $(x_0,y_0)$ to $(x_0+2l,y_0)$, with $x_0,y_0\in \bbZ$ and $l\in \bbZ_{\geq0}$, and such that it is contained in the half-plane $\{(x,y)\in\bbR^2 \mid y\leq y_0\}$.

To any Dyck path we associate a \emph{Dyck strip}: this is the set of boxes (in the sense of \Cref{boxdef}) given by squares with center in the integral coordinates of the Dyck path. If $D$ is a Dyck strip corresponding to a Dyck path going from $(x_0,y_0)$ to $(x_0+2l,y_0)$ we define its \emph{height} to be $\hgt(D):=y_0$ and its length to be $\ell(D):=2l+1$, i.e. the length of a Dyck strip is the number of its boxes.

Let $\lambda,\mu$ be paths with $\lambda\leq \mu$.
 We call \emph{Dyck partition} a partition of $\calA(\lambda,\mu)$ into Dyck strips.
Given a Dyck partition $\bfP$, we denote by $|\bfP|$ the number of Dyck strips in $\bfP$. We denote by $\hgt(\lambda,\mu)$ the maximal height of a Dyck strip in $\bfP$, for any Dyck partition of $\calA(\lambda,\mu)$.
\end{definition}

Our convention is to denote Dyck partitions by bold capital letters and Dyck strips by capital letters, so a typical Dyck partition would be denoted as $\bfP=\{D_1,\ldots,D_k\}$.

Given a box $B$ centered in $(x,y)$ we call the box centered in $(x,y+2)$ (resp. $(x,y-2)$, $(x+1,y+1)$, $(x+1,y-1)$, $(x-1,y-1)$, $(x-1,y+1)$) the box North (resp. South, NE, SE, SW, NW) of $B$.

\begin{figure}[H]
\begin{center}
\begin{tikzpicture}[x=.9cm,y=.9cm]
\draw[very thick] (2,0) to (3,1) to (4,0) to (3,-1) to (2,0);
\tikzset{vertex/.style={}}
\tikzset{edge/.style={}}
\tabpath{+,+,+,-,-,-}
\tabpath{-,-,-,+,+,+}
\tikzset{edge/.style={dotted}}
\tabpath{+,+,-,+}
\tabpath{+,-,-,+,+}
\tabpath{-,+,+,-,-}
\tabpath{-,-,+,-}
\node at (2,1) {NW};
\node at (3,2) {North};
\node at (4,1) {NE};
\node at (4,-1) {SE};
\node at (2,-1) {SW};
\node at (3,-2) {South};
\node at (3,0) {$\bullet$};
\node at (3,0.3) {$(x,y)$};
\node at (3,-0.5) {$B$};
\end{tikzpicture}
\end{center}
\end{figure}

There are two special type of Dyck partitions of $\calA(\lambda,\mu)$ which describe the $I$-parabolic (or Grassmannians) Kazhdan--Lusztig polynomials. 
\begin{definition}\label{TypeIandII}
	We say that a Dyck partition $\bfP$ is of type 1 or of type 2 if it satisfies one of the following two rules:
	\begin{description}
	\item[\textbf{Type 1:}] For any Dyck strip $D\in \bfP$ if there exists a a Dyck strip $D'\in \bfP$ containing a box North of a box in $D$, then every box North of a box in $D$ lies in $D'$.
	\item[\textbf{Type 2:}] For any Dyck strip $D\in \bfP$, if there exists a Dyck strip $D'\in \bfP$ containing a box South, SW or SE of a box in $D$ then every box South, SW or SE of a box in $D$ belongs either to $D$ or $D'$.
	\end{description}
\end{definition}

\begin{figure}[h]
	\begin{center}
	\begin{tabular}{c c c}
		\begin{tikzpicture}[x=\boxmini,y=\boxmini]
		\tikzset{vertex/.style={}}
		\tikzset{edge/.style={very thick}}
		\tabpath{+,-,+,+,-,-,+,+,-,+,-,-}
		\tabpath{+,-,+,-,-,-,+,+,-,+,+}
		\tabpath{-,-,+,-,-,-,+,+,+,-,+,+}
		\tabpath{-,-,+,+,-,-,+,+,-,+,-}
		\tikzset{edge/.style={very thick,blue}}
		\tabpath{+,-,+,+,-,-,+,+,-,+,-,-}
		\tikzset{edge/.style={very thick,red}}
		\tabpath{-,-,+,-,-,-,+,+,+,-,+,+}
		\node[blue] at (12,2) {$\mu$};
		\node[red] at (12,-2) {$\lambda$};
		\end{tikzpicture}
	&\qquad &
		\begin{tikzpicture}[x=\boxmini,y=\boxmini]
		\tikzset{vertex/.style={}}
		\tikzset{edge/.style={very thick}}
		\tabpath{+,-,+,+,-,-,+,+,-,+,-,-}
		\tabpath{+,-,+,-,-,-,+,+,+,-,+,-}
		\tabpath{-,-,+,-,-,-,+,+,+,-,+,+}	
		\tikzset{edge/.style={very thick,blue}}
		\tabpath{+,-,+,+,-,-,+,+,-,+,-,-}
		\tikzset{edge/.style={very thick,red}}
		\tabpath{-,-,+,-,-,-,+,+,+,-,+,+}
		\node[blue] at (12,2) {$\mu$};
		\node[red] at (12,-2) {$\lambda$};
		\end{tikzpicture}
	\end{tabular}
	\end{center}
\caption{Two Dyck partitions of the region $\calA(\lambda,\mu)$. The partition on the left is of Type 1 while the partition on the right is of Type 2.}
\end{figure}

\begin{definition}
	Let $\lambda,\mu\in \Lambda_{n,i}$ with $\lambda\leq \mu$. For $x\in \{1,2\}$ we denote by $\Conf^x(\lambda,\mu)$ be the set of Dyck partitions of Type $x$ in $\calA(\lambda,\mu)$.
 We define the following polynomials
\[Q^{(x)}_{\lambda,\mu}(v)=\sum_{\bfP\in \Conf^x(\lambda,\mu)}v^{|\bfP|}.\]

\end{definition}

Recall from \Cref{chap1} the I-parabolic KL polynomials $h^I_{x,y}(v)$ and the inverse I-parabolic KL polynomials $g_{x,y}^I(v)$.
\begin{theorem}[{\cite[Theorem 5 and Corollary 2]{SZJ}}]\label{SZJ}

For any $x,y\in W$ with $x\leq y$ we have \[Q^{(1)}_{\lambda^{x},\lambda^{y}}(v)=h^I_{x,y}(v)\quad\text{ and }\quad Q^{(2)}_{\lambda^{x},\lambda^y}(v)=g^I_{x,y}(v).\]

Moreover, for $\lambda,\mu \in \Lambda_{n,i}$ with $\lambda\leq \mu$, the set $\Conf^2(\lambda,\mu)$ has at most one element, i.e. $g_{\lambda,\mu}^I(v)$ is a monomial.
\end{theorem}

We introduce some additional notation.
Let $\lambda,\mu\in \Lambda_{n,i}$ with $\lambda\leq \mu$ be such that the region $\calA(\lambda,\mu)$ consists of a single Dyck strip $D$. Then we say that $D$ \emph{can be added to} $\lambda$ and we write $\lambda + D:=\mu$. Vice versa, we say that $D$ \emph{can be removed from} $\mu$ and we write $\mu - D:=\lambda$.

\subsection{Small Resolutions of Schubert varieties}

A small resolution of singularities of a complex algebraic variety $X$ is a morphism $p:\tilde{X}\raw X$ where $\tilde{X}$ is smooth variety such that for every $r>0$ we have
\[\mathrm{codim}\{x\in X \mid \dim (p^{-1}(x))\geq r\}>2r.\]
For every small resolution we have that $H^\bullet(\tilde{X},\bbQ)\cong IH^{\bullet-\dim X}(X,\bbQ)$. 

Zelevinski{\u\i} \cite{Zel} showed that all Schubert varieties in a Grassmannian admit a small resolution. In our setting this means that any indecomposable Soergel bimodule $B_x^I$ is isomorphic to a generalized Bott--Samelson bimodule, i.e. there exists a reduced translation pair $(\vec{I},\vec{J})$ such that 
\begin{equation}\label{genBS}
B_x^I\cong BS(\vec{I},\vec{J}).
\end{equation}
 Following Zelevinski{\u\i}, we describe how to construct such a translation pair. 
 For $\lambda\in \Lambda_{n,i}$, we denote by $\Peaks(\lambda)=\{p\in [1,n]\mid (\lambda_p,\lambda_{p+1})=(\upath,\dpath)\}$ the set of peaks of $\lambda$.
Choose a peak $p\in \Peaks(\lambda)$. Let $a$ be the maximum index with the property that $a< p$ and $\lambda_a=\dpath$ and let $b$ be the minimum index with the property that $b>p$ and $\lambda_b=\upath$.
We obtain a new path $\lambda^p\in \Lambda_{n,i}$ by setting
\[\lambda^p_i=\begin{cases}
\lambda_{a+b-i}& \text{ for all }i\in [a+1,b-1],\\
\lambda_i & \text{ otherwise.}\end{cases}\]

\begin{figure}[h]
	\begin{center}
	\begin{tabular}{c c c}
		\begin{tikzpicture}[x=\boxmini,y=\boxmini]
		\tikzset{vertex/.style={}}
		\tikzset{edge/.style={very thick}}
		\tabpath{+,-,+,+,+,-,-,+,-}
		\tikzset{edge/.style={}}
		\tabpath{-,-,-,-,+,+,+,+,+}
		\tikzset{edge/.style={dotted}}
		\tabpath{-,+,+,-,+}
		\tabpath{-,+,-,+,+,-,-}
		\tabpath{-,-,+,+,-,+,+}
		\tabpath{-,-,-,+,+,-}
		\tabpath{-,-,+,-,-}
		\tabpath{-,+,+,+,-,+,-,-}
		\node at (5.7,3.5) {$p$};
		\node at (0,2) {$\lambda$};
		\end{tikzpicture}
	& \qquad & 
	\begin{tikzpicture}[x=\boxmini,y=\boxmini]
	\tikzset{vertex/.style={}}
	\tikzset{edge/.style={very thick}}
	\tabpath{+,-,-,-,+,+,+,+,-}
	\tikzset{edge/.style={}}
	\tabpath{-,-,-,-,+,+,+,+,+}
	\tikzset{edge/.style={dotted}}
	\tabpath{-,+,-,-,-}
	\tabpath{-,-,+,-,+,-}
	\tabpath{-,-,-,+,+,+,-}
	\tabpath{+,-,-,-,+,+,+,-}
	\tikzset{edge/.style={very thick, dotted}}
	\tabpath{+,-,+,+,+,-,-,+,-}
	\node at (4,-1) {$\bar{p}$};
	\node at (0,2) {$\lambda^p$};
	\end{tikzpicture}
	\end{tabular}
\end{center}
	\caption{Passing from $\lambda$ to $\lambda^p$. In this example we have $a=2$, $p=5$, $b=8$ and $\bar{p}=4$.}
\end{figure}

 Clearly $\lambda^p< \lambda$ and $\Peaks(\lambda^p)=\Peaks(\lambda)\setminus\{p\}$.
Given $x\in W^I$ such that $\lambda=\lambda^x$, let $S_\lambda:=\{s \in S \mid sx\leq x$ in $W/W_I\}$.
In other words, $S_\lambda$ is the set of simple reflections indexed by the peaks and slopes of $\lambda$.
 Notice that $S_{\lambda^p}=(S_{\lambda}\cup \{a,b-1\})\setminus \{\bar{p}\}$, where $\bar{p}:=a+(b-1)-p$.

\begin{definition}
An ordering $(p_1,p_2,\ldots,p_k)$ of the elements of $\Peaks(\lambda)$ is called \emph{neat} if 
\begin{itemize}
	\item the height of $p_1$ is less or equal than the height of the peaks adjacent to $p_1$ in $\lambda$,
	\item $(p_2,p_3,\ldots,p_k)$ is a neat ordering of $\Peaks(\lambda^{p_1})$. 
\end{itemize}
\end{definition}

In \cite{Zel}, Zelevinski{\u\i}  associated to each neat ordering of $\Peaks(\lambda)$ a small resolution of the Schubert variety $X_\lambda$. The resolution is constructed by induction on the number of peaks. Let $(p_1,\ldots,p_k)$ be a neat ordering of $\Peaks(\lambda)$ and let $p=p_1$. Assume we have already constructed a small resolution $\widehat{X}_{\lambda^{p}}$ of $X_{\lambda^{p}}$. Let $a,b$ and $\bar{p}$ be as above. Then, multiplication induces a small resolution of singularities
\[\widehat{X}_\lambda:=L(a+1,b-2)\times_{P(\bar{p})} \widehat{X}_{\lambda^{p_1}}\raw X_\lambda,\]
where $L(a+1,b-2)$ is the Levi subgroup corresponding to the subset $[a+1,b-2]$ and $P(\bar{p})\subset L(a+1,b-2)$ is the maximal parabolic subgroup associated to the subset $[a+1,b-2]\setminus \{\bar{p}\}$ (see \cite[\S 5]{Per} for more details on this construction).

 Recall from \eqref{stabil} that we have $B_\lambda^I=R\otimes_{R^{S_\lambda}}{}^{S_\lambda}B_\lambda^I$. 
 Zelevinski{\u\i}'s result (\cite[Theorem 1]{Zel}) can be translated in the language of Soergel bimodules as follows.
\begin{theorem}\label{small}
Let $\lambda\in \Lambda_{n,i}$ and let $(p_1,\ldots,p_k)$ be a neat ordering of $\Peaks(\lambda)$. Fix $p=p_1$ and let $a,b$ and $\bar{p}$ be defined as above. Let $\tilde{S}:=S_{\lambda^p}\setminus \{a,b-1\}=S_\lambda\setminus\{\bar{p}\}$.
Then we have:
\[{}^{S_\lambda}B_\lambda^I\cong \prescript{}{S_\lambda}{\left(R^{\tilde{S}}\otimes_{R^{S_{\lambda^p}}}{}^{S_{\lambda^p}}B_{\lambda^p}^I\right)}\left(\ell(S_\lambda))-\ell(\tilde{S})\right).
\]
\end{theorem}

In terms of $S$-diagrams we can obtain ${}^{S_\lambda}B_\lambda^I$ from ${}^{S_{\lambda^p}}B_{\lambda^p}^I$ as follows:
\begin{center}
 \begin{tikzpicture}[scale=1]
\draw[thick,red,triangle 45-] (-1,0) to (-1,2);
\draw[thick,blue,-triangle 45] (0,0) to (0,2);
\draw[thick,green,-triangle 45] (1,0) to (1,2);
\draw (5,0) to (2,0) to (2,2) to (5,2);
\node at (3.5,1) {${}^{S_{\lambda^p}}B^I_{\lambda^p}$};
\node at (-2.5,1) {${}^{S_\lambda}B^I_\lambda=$};
\node[red] at (-0.7,0.2) {$\bar{p}$};
\node[blue] at (0.3,0.2) {$a$};
\node[green] at (1.5,0.2) {$b-1$};
\end{tikzpicture}
\end{center}
where the arrow labeled by $a$ (resp. by $b-1$) is neglected if $a\leq 0$ (resp. $b\geq n+1$).

In general, there are several possible neat orders for a path $\lambda$, thus the algorithm of \Cref{small} gives rise to several translation pairs, any of which leads to a generalized Bott--Samelson bimodule isomorphic to $B_\lambda^I$. By Soergel's Hom formula (\Cref{SHF}) we have $\End^0(B_\lambda^I)\cong \bbQ$. Hence, we can identify these bimodules via the unique isomorphism sending $1^\otimes:=1\otimes 1\otimes \ldots 1$ to $1^\otimes$.

\begin{example}
	Consider the path $\lambda=(\upath,\dpath,\upath,\dpath)\in \Lambda_{4,2}$. 	
	\begin{center}
	\begin{tikzpicture}[x=\boxmini,y=\boxmini]
	\tikzset{vertex/.style={}}
	\tikzset{edge/.style={very thick}}
	\tabpath{+,-,+,-}
	\tikzset{edge/.style={}}
	\tabpath{-,-,+,+}
	\tikzset{edge/.style={dotted}}
	\tabpath{-,+,-,+}
	\node at (4,1) {$\lambda$};
	\node at (1,0) {$1$};
	\node at (2,-1) {$2$};
	\node at (3,0) {$3$};
	\end{tikzpicture}
	\end{center}
Then $(1,3)$ and $(3,1)$ are both neat ordering of $\Peaks(\lambda)$. They give rise to the two following different diagrammatic description of the bimodule $B_\lambda^I$:
\begin{center}
	\begin{tikzpicture}[scale=1]
	\draw[thick,green,-triangle 45] (-1,0) to (-1,2);
	\draw[thick,blue,-triangle 45] (0,0) to (0,2);
	\draw[thick,blue,triangle 45-] (1,0) to (1,2);
	\draw[thick,red,-triangle 45] (2,0) to (2,2);
	\draw[thick,red,triangle 45-] (3,0) to (3,2);
	\draw[thick,blue,-triangle 45] (4,0) to (4,2);
	\node[green] at (-0.7,0.2) {$3$};
	\node[blue] at (0.3,0.2) {$1$};
	\node[blue] at (1.3,0.2) {$1$};
	\node[red] at (2.3,0.2) {$2$};
	\node[red] at (3.3,0.2) {$2$};
	\node[blue] at (4.3,0.2) {$1$};
	\end{tikzpicture} 
	\begin{tikzpicture}[scale=1]
	\draw[thick,green,-triangle 45] (-1,0) to (-1,2);
	\draw[thick,blue,-triangle 45] (0,0) to (0,2);
	\draw[thick,green,triangle 45-] (1,0) to (1,2);
	\draw[thick,red,-triangle 45] (2,0) to (2,2);
	\draw[thick,red,triangle 45-] (3,0) to (3,2);
	\draw[thick,green,-triangle 45] (4,0) to (4,2);
	\node[green] at (-0.7,0.2) {$3$};
	\node[blue] at (0.3,0.2) {$1$};
	\node[green] at (1.3,0.2) {$3$};
	\node[red] at (2.3,0.2) {$2$};
	\node[red] at (3.3,0.2) {$2$};
	\node[green] at (4.3,0.2) {$3$};
	\node at (-2,1) {$\cong$};
	\end{tikzpicture}
\end{center}
In terms of bimodules, this corresponds to the following isomorphisms of generalized Bott--Samelson.
\[ B_\lambda^I\cong R\otimes_{R^{1,3}} R^3\otimes_{R^{2,3}}R^3\otimes_{R^{1,3}}R^{1,3}(3)\cong R\otimes_{R^{1,3}} R^1\otimes_{R^{1,2}}R^1\otimes_{R^{1,3}}R^{1,3}(3). \]
\end{example}

\subsection{Morphisms of degree one}\label{diagramsec}
From \Cref{SZJ} and Soergel's Hom formula \cite[Theorem 7.4.1]{W4} we obtain that
\[\Hom^1(B_\lambda^I,B_\mu^I)=\begin{cases}\bbQ & \text{if $\lambda=\mu\pm D$ for some Dyck strip $D$,}\\
0 & \text{otherwise.}\end{cases}\]
This means that to any Dyck strip we can associate a degree $1$ morphisms of bimodules that is unique up to a scalar.

Assume that $\lambda,\mu$ are two paths such that $\lambda=\mu +D$ for some Dyck strip $D$. Most of the peaks of $\lambda$ are also peaks for $\mu$. 
Let us be more precise: 
there are exactly two integers $a$ and $b+1$ with $a<b+1$ such that $\lambda_a\neq \mu_a$ and $\lambda_{b+1}\neq \mu_{b+1}$. (In fact, we have $\lambda_a\in \upath$, $\mu_a=\dpath$ and $\lambda_{b+1}=\dpath$, $\mu_{b+1}=\upath$.) We have $a,b\in \Peaks(\lambda)\setminus\Peaks(\mu)$. On the other hand we have $(a-1)\not\in \Peaks(\lambda)$ and \[(a-1)\in \Peaks(\mu) \iff \lambda_{a-1}=\mu_{a-1}=\upath.\]
Similarly, $(b+1)\not \in \Peaks(\lambda)$ and $(b+1)\in \Peaks(\mu)\iff \lambda_{b+2}=\mu_{b+2}=\dpath$.

Since $D$ is a Dyck strip, we have $\hgt_a(\lambda)=\hgt_b(\lambda)$ and all the peaks of $\lambda$ between $a$ and $b$ are of height smaller or equal than $\hgt_a(\lambda)$.
This means that we can find a neat ordering of $\Peaks(\lambda)$ of the form 
\begin{equation}\label{superneat}
(p_1,p_2,\ldots,p_m,q_1,q_2,\ldots q_{m'},a,b,r_1,r_2,\ldots,r_{m''})
\end{equation}
satisfying the following three rules:
\begin{itemize}
	\item for every peak $p\in \Peaks(\lambda)$ such that $p\not \in [a,b]$ and $\hgt_p(\lambda)<\hgt(D)$ we have $p=p_i$ for some $i\in [1,m]$,
	\item for every peak $p\in \Peaks(\lambda)$ such that $p\in [a,b]$ we have $p=q_i$ for some $i\in [1,m']$,
	\item for every peak $p\in \Peaks(\lambda)$ such that $p\not \in [a,b]$ and $\hgt_p(\lambda)\geq \hgt(D)$ we have $p=r_i$ for some $i\in [1,m'']$,
\end{itemize}

 Using \Cref{small}, from the neat order \eqref{superneat} we can construct a diagram representing the object $B_\lambda^I$ of the form
\begin{center}
 \begin{tikzpicture}[scale=1]
\draw[thick,blue,-triangle 45] (0,0) to (0,2);
\draw[thick,green,-triangle 45] (0.5,0) to (0.5,2);
\draw[thick,red,-triangle 45] (1.5,0) to (1.5,2);
\node at (1,1) {$\ldots$};
\draw (2,0) rectangle (4.5,2);
\node at (3.25,1) {$\frP$};
\draw (5,0) rectangle (7.5,2);
\node at (6.25,1) {$\frQ$};
\draw[thick,blue,triangle 45-] (8,0) to (8,2);
\draw[thick,green,-triangle 45] (8.5,0) to (8.5,2);
\draw[thick,red,-triangle 45] (9,0) to (9,2);
\draw[thick,blue,triangle 45-] (9.5,0) to (9.5,2);
\draw[thick,green,-triangle 45] (10,0) to (10,2);
\draw[thick,red,-triangle 45] (10.5,0) to (10.5,2);
\draw (11,0) rectangle (13.5,2);
\node at (12.25,1) {$\frR$};
\node[gray] at (14,1) {\scriptsize $I$};
\node[gray] at (-0.5,1) {\scriptsize $\emptyset$};
\draw [decorate,decoration={brace,amplitude=8pt,mirror},yshift=0pt]
(0,-0.1) -- (1.5,-0.1) node [black,midway,yshift=-16pt] {\footnotesize $S_\lambda$};
\draw [decorate,decoration={brace,amplitude=8pt,mirror},yshift=0pt]
(8,-0.1) -- (9,-0.1) node [black,midway,yshift=-16pt] {\footnotesize $a$};
\draw [decorate,decoration={brace,amplitude=8pt,mirror},yshift=0pt]
(9.5,-0.1) -- (10.5,-0.1) node [black,midway,yshift=-16pt] {\footnotesize $b$};
\end{tikzpicture}
\end{center}
where the boxes labeled by $\frP$, $\frQ$ and $\frR$ correspond respectively to the peaks $p_i$'s, $q_i$'s and $r_i$'s.

Similarly, we obtain a neat ordering $(p_1,\ldots,p_m,q_{1},\ldots q_{m'},\widehat{a-1},\widehat{b+1},r_1,\ldots,r_{m''})$ of $\Peaks(\mu)$ (here the notation $\widehat{j}$ means that $j$ is neglected if $j\not \in \Peaks(\mu)$) and we obtain a diagrammatic presentation of $B_\mu^I$ as follows.
\begin{center}
	\begin{tikzpicture}[scale=1]
	\draw[thick,blue,-triangle 45] (0,0) to (0,2);
	\draw[thick,green,-triangle 45] (0.5,0) to (0.5,2);
	\draw[thick,red,-triangle 45] (1.5,0) to (1.5,2);
	\node at (1,1) {$\ldots$};
	\draw (2,0) rectangle (4.5,2);
	\node at (3.25,1) {$\frP'$};
	\draw (5,0) rectangle (7.5,2);
	\node at (6.25,1) {$\frQ'$};
	\draw[thick,blue,triangle 45-] (8,0) to (8,2);
	\draw[thick,green,-triangle 45] (8.5,0) to (8.5,2);
	\draw[thick,red,-triangle 45] (9,0) to (9,2);
	\draw[thick,blue,triangle 45-] (9.5,0) to (9.5,2);
	\draw[thick,green,-triangle 45] (10,0) to (10,2);
	\draw[thick,red,-triangle 45] (10.5,0) to (10.5,2);
	\draw (11,0) rectangle (13.5,2);
	\node at (12.25,1) {$\frR$};
	\node[gray] at (14,1) {\scriptsize $I$};
	\node[gray] at (-0.5,1) {\scriptsize $\emptyset$};
	\draw [decorate,decoration={brace,amplitude=8pt,mirror},yshift=0pt]
	(0,-0.1) -- (1.5,-0.1) node [black,midway,yshift=-16pt] {\footnotesize $S_\mu$};
	\draw [decorate,decoration={brace,amplitude=8pt,mirror},yshift=0pt]
	(8,-0.1) -- (9,-0.1) node [black,midway,yshift=-16pt] {\footnotesize $\widehat{a-1}$};
	\draw [decorate,decoration={brace,amplitude=8pt,mirror},yshift=0pt]
	(9.5,-0.1) -- (10.5,-0.1) node [black,midway,yshift=-16pt] {\footnotesize $\widehat{b+1}$};
	\end{tikzpicture}
\end{center}

The singular $S$-diagrams in $\frP$ and $\frP'$ coincide, except that they may have regions with different labels. The same holds for $\frQ$ and $\frQ'$.

Let $\delta=(S_\lambda \cup S_\mu)\setminus (S_\lambda\cap S_\mu)$ be the symmetric difference of $S_\lambda$ and $S_\mu$. Then $\delta\cug \{a-1,a,b,b+1\}$. Any $j\in \delta$ is distinct from the label of any arrow appearing in $\frP$ and $\frQ$, and it is distant from any label of a downward arrow in $\frP$ and $\frQ$. Hence, up to an isomorphism which sends $1^\otimes$ to $1^\otimes$, we can draw $B_\lambda^I$ as follows.
\begin{center}
	\begin{tikzpicture}[scale=0.88]
	\draw[thick,blue,-triangle 45] (0,0) to (0,2);
	\draw[thick,green,-triangle 45] (0.5,0) to (0.5,2);
	\draw[thick,red,-triangle 45] (1.5,0) to (1.5,2);
	\node at (1,1) {$\ldots$};
	\draw (2,0) rectangle (4.5,2);
	\node at (3.25,1) {$\frP''$};
	\draw (5,0) rectangle (7.5,2);
	\node at (6.25,1) {$\frQ''$};
	\draw[thick,blue,-triangle 45] (8,0) to (8,2);
	\draw[thick,green,-triangle 45] (8.5,0) to (8.5,2);
	\draw[thick,red,-triangle 45] (9.5,0) to (9.5,2);
	\node at (9,1) {$\ldots$};
	\draw[thick,blue,triangle 45-] (10,0) to (10,2);
	\draw[thick,green,-triangle 45] (10.5,0) to (10.5,2);
	\draw[thick,red,-triangle 45] (11,0) to (11,2);
	\draw[thick,blue,triangle 45-] (11.5,0) to (11.5,2);
	\draw[thick,green,-triangle 45] (12,0) to (12,2);
	\draw[thick,red,-triangle 45] (12.5,0) to (12.5,2);
	\draw (13,0) rectangle (15.5,2);
	\node at (14.25,1) {$\frR$};
	\node[gray] at (16,1) {\scriptsize $I$};
	\node[gray] at (-0.5,1) {\scriptsize $\emptyset$};
	\draw [decorate,decoration={brace,amplitude=8pt,mirror},yshift=0pt]
	(0,-0.1) -- (1.5,-0.1) node [black,midway,yshift=-16pt] {\footnotesize $S_\lambda\cap S_\mu$};
	\draw [decorate,decoration={brace,amplitude=8pt,mirror},yshift=0pt]
	(8,-0.1) -- (9.5,-0.1) node [black,midway,yshift=-16pt] {\footnotesize $S_\lambda \setminus S_\mu$};
	\draw [decorate,decoration={brace,amplitude=8pt,mirror},yshift=0pt]
	(10,-0.1) -- (11,-0.1) node [black,midway,yshift=-16pt] {\footnotesize $a$};
	\draw [decorate,decoration={brace,amplitude=8pt,mirror},yshift=0pt]
	(11.5,-0.1) -- (12.5,-0.1) node [black,midway,yshift=-16pt] {\footnotesize $b$};
	\end{tikzpicture}
\end{center}

The isomorphism between these two diagrammatic presentations of $B_\lambda^I$ is simply given by crossing all the arrows in $S_\lambda\setminus S_\mu$ over $\frP$ and $\frQ$.
The box $\frP''$ is equal to $\frP$ up to a relabeling of the regions.

We can similarly do the same for $\mu$ and obtain the following diagram representing $B_\mu^I$:

\begin{center}
	\begin{tikzpicture}[scale=0.88]
	\draw[thick,blue,-triangle 45] (0,0) to (0,2);
	\draw[thick,green,-triangle 45] (0.5,0) to (0.5,2);
	\draw[thick,red,-triangle 45] (1.5,0) to (1.5,2);
	\node at (1,1) {$\ldots$};
	\draw (2,0) rectangle (4.5,2);
	\node at (3.25,1) {$\frP''$};
	\draw (5,0) rectangle (7.5,2);
	\node at (6.25,1) {$\frQ''$};
	\draw[thick,blue,-triangle 45] (8,0) to (8,2);
	\draw[thick,green,-triangle 45] (8.5,0) to (8.5,2);
	\draw[thick,red,-triangle 45] (9.5,0) to (9.5,2);
	\node at (9,1) {$\ldots$};
	\draw[thick,blue,triangle 45-] (10,0) to (10,2);
	\draw[thick,green,-triangle 45] (10.5,0) to (10.5,2);
	\draw[thick,red,-triangle 45] (11,0) to (11,2);
	\draw[thick,blue,triangle 45-] (11.5,0) to (11.5,2);
	\draw[thick,green,-triangle 45] (12,0) to (12,2);
	\draw[thick,red,-triangle 45] (12.5,0) to (12.5,2);
	\draw (13,0) rectangle (15.5,2);
	\node at (14.25,1) {$\frR$};
	\node[gray] at (16,1) {\scriptsize $I$};
	\node[gray] at (-0.5,1) {\scriptsize $\emptyset$};
	\draw [decorate,decoration={brace,amplitude=8pt,mirror},yshift=0pt]
	(0,-0.1) -- (1.5,-0.1) node [black,midway,yshift=-16pt] {\footnotesize $S_\lambda\cap S_\mu$};
	\draw [decorate,decoration={brace,amplitude=8pt,mirror},yshift=0pt]
	(8,-0.1) -- (9.5,-0.1) node [black,midway,yshift=-16pt] {\footnotesize $S_\mu \setminus S_\lambda$};
	\draw [decorate,decoration={brace,amplitude=8pt,mirror},yshift=0pt]
	(10,-0.1) -- (11,-0.1) node [black,midway,yshift=-16pt] {\footnotesize $\widehat{a-1}$};
	\draw [decorate,decoration={brace,amplitude=8pt,mirror},yshift=0pt]
	(11.5,-0.1) -- (12.5,-0.1) node [black,midway,yshift=-16pt] {\footnotesize $\widehat{b+1}$};
	\end{tikzpicture}
\end{center}

 Now, to specify a morphism between $B_\lambda^I$ and $B_\mu^I$ it is enough to do it ``locally'' via some diagram 
 of the form:
 \begin{equation}\label{canIdothis}
 	\begin{tikzpicture}[scale=0.88, baseline=(current bounding box.center)]
 	\draw[thick,blue,-triangle 45] (8,0) to (8,1);
 	\draw[thick,green,-triangle 45] (8.5,0) to (8.5,1);
 	\draw[thick,red,-triangle 45] (9.5,0) to (9.5,1);
 	\node at (9,0.5) {$\ldots$};
 	\draw[thick,blue,triangle 45-] (10,0) to (10,1);
 	\draw[thick,green,-triangle 45] (10.5,0) to (10.5,1);
 	\draw[thick,red,-triangle 45] (11,0) to (11,1);
 	\draw[thick,blue,triangle 45-] (11.5,0) to (11.5,1);
 	\draw[thick,green,-triangle 45] (12,0) to (12,1);
 	\draw[thick,red,-triangle 45] (12.5,0) to (12.5,1);
 	\draw [decorate,decoration={brace,amplitude=8pt},yshift=0pt]
 	(8,1.1) -- (9.5,1.1) node [black,midway,yshift=16pt] {\footnotesize $S_\lambda \setminus S_\mu$};
 	\draw [decorate,decoration={brace,amplitude=8pt},yshift=0pt]
 	(10,1.1) -- (11,1.1) node [black,midway,yshift=16pt] {\footnotesize $a$};
 	\draw [decorate,decoration={brace,amplitude=8pt},yshift=0pt]
 	(11.5,1.1) -- (12.5,1.1) node [black,midway,yshift=16pt] {\footnotesize $b$};
 		\draw (7.5,0) rectangle (13,-1.5);
 	\node at (10.25,-.75) {$f_D$};
 	
 	 	\draw[thick,blue,-triangle 45] (8,-2.5) to (8,-1.5);
 	\draw[thick,green,-triangle 45] (8.5,-2.5) to (8.5,-1.5);
 	\draw[thick,red,-triangle 45] (9.5,-2.5) to (9.5,-1.5);
 	\node at (9,-2) {$\ldots$};
 	\draw[thick,blue,triangle 45-] (10,-2.5) to (10,-1.5);
 	\draw[thick,green,-triangle 45] (10.5,-2.5) to (10.5,-1.5);
 	\draw[thick,red,-triangle 45] (11,-2.5) to (11,-1.5);
 	\draw[thick,blue,triangle 45-] (11.5,-2.50) to (11.5,-1.5);
 	\draw[thick,green,-triangle 45] (12,-2.5) to (12,-1.5);
 	\draw[thick,red,-triangle 45] (12.5,-2.5) to (12.5,-1.5);
 	\draw [decorate,decoration={brace,amplitude=8pt,mirror},yshift=0pt]
 	(8,-2.6) -- (9.5,-2.6) node [black,midway,yshift=-16pt] {\footnotesize $S_\mu \setminus S_\lambda$};
 	\draw [decorate,decoration={brace,amplitude=8pt,mirror},yshift=0pt]
 	(10,-2.6) -- (11,-2.6) node [black,midway,yshift=-16pt] {\footnotesize $\widehat{a-1}$};
 	\draw [decorate,decoration={brace,amplitude=8pt,mirror},yshift=0pt]
 	(11.5,-2.6) -- (12.5,-2.6) node [black,midway,yshift=-16pt] {\footnotesize $\widehat{b+1}$};
 	\end{tikzpicture}
 \end{equation}

Following \cite[Definition 4.4]{Per}, we define a partial order $\lhd$ on the set of boxes in $\calA(\lambda)$. The order $\lhd$ is generated by the relation $\theta\vartriangleleft \theta'$ if $\theta$ is SW or SE of $\theta$.
If $p\in \Peaks(\lambda)$ we denote by $\theta(p)$ the unique box containing $p$. Clearly $\theta(p)$ is a maximal element with respect to $\lhd$. Define:
\[U:=\left\{\theta \text{ box in }\calA(\lambda)\;\middle|
\begin{tabular}{c}
$\exists q\in\{q_1,\ldots, q_{m'},a,b-1\}\text{ such that }\theta\vartriangleleft \theta(q)$\\
and $\theta\not\vartriangleleft \theta(r_i)\text{ for any }i\in [1,m'']$
\end{tabular}
\right\}.\]

Let us call $u_s$, $u_e$ and $u_w$ respectively the labels of the southernmost, easternmost and westernmost boxes of $U$.
\begin{center}
\begin{tikzpicture}[x=\boxmini,y=\boxmini]
\tikzset{vertex/.style={}}
\tikzset{edge/.style={very thick}}
\draw[fill=yellow] (1,-1) -- (4,2) -- (6,0) -- (7,1) -- (10,-2) -- (12,0) -- (13,-1) -- (7,-7) -- cycle;
\draw[fill=orange] (4,2) -- (5,3) -- (6,2) -- (7,3) -- (10,0) -- (12,2) -- (13,1) -- (15,3) -- (16,2) -- (13,-1) -- (12,0) -- (10,-2) -- (7,1) -- (6,0) -- (4,2) -- cycle;
\tabpathc{+,-,+,+,+,-,+,-,-,-,+,+,-,+,+,-,+,+,+,-,-,-,+,-}{blue}
\begin{scope}[shift={(0,-0.1)}]
\tabpathc{+,-,+,+,-,-,+,-,-,-,+,+,-,+,+,+,+,+,+,-,-,-,+,-}{red}
\end{scope}
\tabpath{-,-,-,-,-,-,-,-,-,-,-,+,+,+,+,+,+,+,+,+,+,+,+,+}
\node at (5,3.5) {$a$};
\node at (7,3.5) {$q_2$};
\node at (12,2.5) {$q_1$};
\node at (15,3.5) {$b$};
\node at (1,1.5) {$p_1$};
\node at (19,5.5) {$r_1$};
\node at (5,3.5) {$a$};
\node at (23,3.5) {$r_2$};
\node at (7,-3) {$U$};
\node at (9,0) {$D$};
\node[red] at (15,-0.5) {$\mu$};
\node[blue] at (9,2) {$\lambda$};
\node at (7,-6) {$u_s$};
\node at (2,-1) {$u_w$};
\node at (15,2) {$u_e$};
\end{tikzpicture}
\end{center}

\setlength{\tabcolsep}{20pt}

We proceed to draw the morphisms $f_D$ as in \eqref{canIdothis} explicitly. To do this, we need to divide into 8 cases labeled from $1a)$ to $2d)$, which depend on the precise form of $D$ and $U$.
All the crossing involving a black arrow are isomorphism of degree $0$. The black arrows are meant to be neglected whenever their label is not an element of $\{1,\ldots,n\}$.

\begin{enumerate}[{\bfseries Case }\bfseries 1:]
	\item We have $a=b$, i.e. $D$ consists of a single box.
		\begin{enumerate}[{Case 1}a)]
			\item $a-1$ and $b+1$ are both peaks for $\mu$.
			
 \realign\begin{minipage}{0.35\textwidth}\begin{center}
 \begin{tikzpicture}[x=\boxmini,y=\boxmini]
 \tikzset{vertex/.style={}}
\draw[fill=yellow] (0,0) -- (2,2) -- (3,1) -- (4,2) -- (7,-1) -- (4,-4) -- cycle;
\draw[fill=orange] (2,2) -- (3,3) -- (4,2) -- (3,1) -- cycle;
\tikzset{edge/.style={very thick}}
\tabpath{+,+,+,-,-,-,-}
\tabpath{+,+,-,+,-,-,-}
\tabpath{-,-,-,-,+,+,+}
\node at (3,2) {$a$};
\node at (1,0) {$u_w$};
\node at (4,-3) {$u_s$};
\node at (6,-1) {$u_e$};
\end{tikzpicture}\end{center}\end{minipage}\quad\begin{minipage}{0.55\textwidth}\begin{center}\begin{tikzpicture}
\draw[-triangle 45, thick, blue] (3,0) to (3,0.5) ;
\draw[-triangle 45 reversed, thick, blue] (-1,0) to (-1,0.5) ;
\draw[thick, blue] (3,0.5) to [out=90,in=0] (1,2) to [out=180,in=90] (-1,0.5) ;
\draw[triangle 45-, thick, red] (2,0) to [out=90,in=-90] (1,3);
\draw[-triangle 45, thick, green] (0,0) to (0,3);
\draw[thick,-triangle 45] (1,0) to [out=90,in=-90] (2,3);
\draw[-triangle 45, thick] (4.5,0) to (4.5,3);
\node at (5.2,2.2) {$u_e+1$};
\node at (2.4,2.2) {$u_w\hspace{-3pt}-\hspace{-3pt}1$};
\node[blue] at (3.2,0.5) {$u_w$};
\node[red] at (2.2,0.5) {$u_s$};
\node[green] at (0.2,0.5) {$a$};
\node at (6.5,1.5) {$\ldots$};
\node at (-1.5,1.5) {$\ldots$};
\end{tikzpicture}\end{center}\end{minipage}
\item $a-1\in \Peaks(\mu)$ and $b+1\not \in \Peaks(\mu)$. In this case $u_w=u_s$ and $u_e=a$

 \realign\begin{minipage}{0.35\textwidth}\begin{center}
		\begin{tikzpicture}[x=\boxmini,y=\boxmini]
 \draw[fill=yellow] (0,0) -- (2,2) -- (3,1) -- (1,-1) -- cycle;
 \draw[fill=orange] (2,2) -- (3,3) -- (4,2) -- (3,1) -- cycle;
 \tikzset{vertex/.style={}}
\tikzset{edge/.style={very thick}}
\tabpath{+,+,+,-}
\tabpath{+,+,-}
\tabpath{-,+,+,+}
\node at (3,2) {$a$};
\node at (1,0) {$u_s$};
\end{tikzpicture}\end{center}\end{minipage}\quad\begin{minipage}{0.55\textwidth}\begin{center}\begin{tikzpicture}
\draw[triangle 45-, thick, red] (0,0) to [out=90, in=-135] (1,1.5) to [out=45, in=-90] (2,3);
\draw[-triangle 45, thick, green] (2,0) to [out=90, in=-45] (1,1.5) to [out=135, in=-90] (0,3);
\draw[-triangle 45, thick] (4,0) to (4,3);

\node at (-1.5,0.5) {$a+1$};
\node at (4.7,2.2) {$u_s-1$};
\node[red] at (0.4,0.5) {$u_s$};
\node[green] at (2.2,0.5) {$a$};
\draw[thick,-triangle 45] (-1,0) to [out=90,in=-90] (3,1.5) to [out=90,in=-90] (3,3);
\node at (5.5,1.5) {$\ldots$};
\node at (-2.5,1.5) {$\ldots$};
\end{tikzpicture}\end{center}\end{minipage}

\item $a-1\not\in \Peaks(\mu)$ and $b+1\in \Peaks(\mu)$. This is completely analogous to case 1b).

\item $a-1,b+1\not\in \Peaks(\mu)$. In this case $U=D$ is a single box labeled by $a$.

 \realign\begin{minipage}{0.35\textwidth}\begin{center}
		\begin{tikzpicture}[x=\boxmini,y=\boxmini]
 \draw[fill=orange] (0,0) -- (1,1) -- (2,0) -- (1,-1) -- cycle;
 \tikzset{vertex/.style={}}
\tikzset{edge/.style={very thick}}
\tabpath{+,-}
\tabpath{-,+}
\node at (1,0) {$a$};
\end{tikzpicture}\end{center}\end{minipage}\quad\begin{minipage}{0.55\textwidth}\begin{center}\begin{tikzpicture}
\draw[-triangle 45, thick, green] (2,3) to [out=-90, in=0] (1,1.5) to [out=180, in=-90] (0,3);
\draw[-triangle 45, thick] (4.5,0) to (4.5,3);
\draw[-triangle 45, thick] (3,0) to (3,3);
\node at (5.1,2.2) {$a+1$};
\node at (3.6,2.2) {$a-1$};
\node[green] at (2.2,2.5) {$a$};
\node at (6.5,1.5) {$\ldots$};
\node at (-1.5,1.5) {$\ldots$};
\end{tikzpicture}\end{center}\end{minipage}
\end{enumerate}

\item We have $a<b$, i.e. the Dyck strip $D$ has length $\ell(D)>1$.

\begin{enumerate}[{Case 2}a)]

\item $a-1$ and $b+1$ are both peaks

 \realign\begin{minipage}{0.35\textwidth}
 	\begin{center}
			\begin{tikzpicture}[x=\boxmini,y=\boxmini]
				\draw[fill=yellow] (0,0) -- (3,3) -- (5,1) -- (6,2) -- (7,1) -- (9,3) -- (11,1) -- (5,-5) -- cycle;
				\draw[fill=orange] (3,3) -- (4,4) --(5,3) -- (6,4) -- (7,3) -- (8,4) -- (9,3) -- (7,1) -- (6,2) -- (5,1) -- cycle;
				\tikzset{vertex/.style={}}
				\tikzset{edge/.style={very thick}}
				\tabpath{+,+,+,+,-,+,-,+,-,-,-}
				\tabpath{+,+,+,-,-,+,-,+,+}
				\tabpath{-,-,-,-,-,+,+,+,+,+,+}
				\draw[dashed,gray] (3,-3) to (7,1);
				\draw[dashed,gray] (5,1) to (6,0);
				\node at (1,0) {$u_w$};
				\node at (4,3) {$a$};
				\node at (8,3) {$b$};
				\node at (5,-4) {$u_s$};
				\node at (10,1) {$u_e$};
				\node at (3,-2) {$v$};
				\node at (6,1) {$v'$};
			\end{tikzpicture}\end{center}\end{minipage}\quad\begin{minipage}{0.55\textwidth}\begin{center}
			\begin{tikzpicture}[scale=.98]
				\node (a) at (1.5,-2) {};
				\draw[-triangle 45, thick, red] (0,0) to (0,-1) ;
				\draw[triangle 45 reversed-, thick, red] (4,-1) to (4,0);
				\draw[-triangle 45 reversed, thick, green] (1,0) to (1,-1) ;
				\draw[triangle 45-, thick, blue] (3,-1) to (3,0);
				\draw[thick, red] (0,0) to (0,-1) to [out=-90,in=180] (2,-2.5) to [out=0,in=-90] (4,-1) to (4,0) ;
				\draw[thick, blue] (3,0) to (3,-4);
				\draw[thick, green] (1,0) to (1,-4);
				\draw[thick,green,triangle 45-] (1,-3) to (1,-4);
				\draw[thick,blue,triangle 45 reversed-] (3,-3) to (3,-4);
				\node[red] at (4.5,-0.5) {$v-1$};
				\node[blue] at (3.2,-0.5) {$u_s$};
				\node[green] at (1.2,-0.5) {$v'$};
				\node[violet] at (4.2,-3.5) {$v$};
				\draw[thick, violet, -triangle 45 reversed] (0,-4) to (0,-3);
				\draw[thick, violet, triangle 45-] (4,-3) to (4,-4) ;
				\draw[thick, violet] (0,-4) to (0,-3) to [out=90,in=180] (2,-1.5) to [out=0,in=90] (4,-3) to (4,-4) ;
				\draw[thick,-triangle 45] (2,-4) to [out=90,in=-90] (2,0);
				\draw[thick,-triangle 45] (5.3,-4) to (5.3,0);
				\node at (2.5,-3.5) {$u_w\hspace{-3pt}-\hspace{-3pt}1$};
				\node at (5.9,-3) {$u_e+1$};
				\node at (-1,-2) {$\ldots$};
				\node at (6.2,-2) {$\ldots$};
			\end{tikzpicture}\end{center}\end{minipage}
		
where $v=u_s-u_e+b$ and $v'=v-u_e+b+a-u_w$. 
\item\label{2b} $a-1\in \Peaks(\mu)$ and $b+1\not\in\Peaks(\mu)$. In this case $u_e=b$.

 \realign\begin{minipage}{0.35\textwidth}\begin{center}
 \begin{tikzpicture}[x=\boxmini,y=\boxmini]
 \draw[fill=orange] (2,2) -- (3,3) --(4,2) -- (5,3) -- (6,2) -- (7,3) -- (8,2) -- (6,0) -- (5,1) -- (4,0) -- cycle;
 \draw[fill=yellow] (0,0) -- (2,2) -- (4,0) -- (5,1) -- (6,0) -- (3,-3) -- cycle;
 \tikzset{vertex/.style={}}
\tikzset{edge/.style={very thick}}
\tabpath{+,+,+,-,+,-,+,-}
\tabpath{+,+,-,-,+,-,+,+}
\tabpath{-,-,-,+,+,+,+,+}
\draw[dashed,gray] (4,0) to (5,-1);
\node at (1,0) {$u_w$};
\node at (3,2) {$a$};
\node at (7,2) {$b$};
\node at (3,-2) {$u_s$};
\node at (5,0) {$v$};
\end{tikzpicture}\end{center}\end{minipage}\quad\begin{minipage}{0.55\textwidth}
 \begin{tikzpicture}[scale=1]
\draw[-triangle 45, thick, red] (0,0) to (0,-1) ;
\draw[triangle 45 reversed-, thick, red] (4,-1) to (4,0);
\draw[-triangle 45 reversed, thick, green] (1,0) to (1,-1) ;
\draw[triangle 45-, thick, blue] (3,-1) to (3,0);
\draw[thick, red] (0,0) to (0,-1) to [out=-90,in=180] (2,-2.5) to [out=0,in=-90] (4,-1) to (4,0) ;
\draw[thick, blue] (3,0) to (3,-1) to [out=-90,in=30] (2,-2) to [out=-150,in=90] (1,-3.5);
\draw[thick, green] (1,0) to (1,-1) to [out=-90,in=150] (2,-2) to [out=-30,in=90] (3,-3.5);
\draw[thick,blue,triangle 45 reversed-] (1,-3.5) to (1,-4);
\draw[thick,green,triangle 45-] (3,-3.5) to (3,-4);
\draw[thick,-triangle 45] (-0.5,-4) to [out=90,in=-90] (2,0);
\draw[thick,-triangle 45] (5.3,-4) to (5.3,0);
\node[red] at (4.5,-0.5) {$u_s-1$};
\node[blue] at (3.2,-0.5) {$u_s$};
\node[green ] at (1.2,-0.5) {$v$};
\node at (-1,-3) {$u_w-1$};
\node at (5.8,-3) {$b+1$};
\node at (-1,-2) {$\ldots$};
\node at (6.3,-2) {$\ldots$};
\end{tikzpicture}
\end{minipage}

where $v=u_s+a-u_w$.
\item $a-1\not\in \Peaks(\mu)$ and $b+1\in \Peaks(\mu)$. This is completely analogous to case \ref{2b}).
\item $a-1,b+1\not\in \Peaks(\mu)$.

 \realign\begin{minipage}{0.35\textwidth}\begin{center}
 \begin{tikzpicture}[x=\boxmini,y=\boxmini]
 \draw[fill=orange] (0,0) -- (1,1) --(2,0) -- (3,1) -- (4,0) -- (5,1) -- (6,0) -- (4,-2) -- (3,-1) -- (2,-2) -- cycle;
 \draw[fill=yellow](3,-3) -- (4,-2) -- (3,-1) -- (2,-2) -- cycle;
 \tikzset{vertex/.style={}}
\tikzset{edge/.style={very thick}}
\tabpath{+,-,+,-,+,-}
\tabpath{-,-,+,-,+,+}
\tabpath{-,-,-,+,+,+}
\node at (1,0) {$a$};
\node at (5,0) {$b$};
\node at (3,-2) {$u_s$};
\end{tikzpicture}\end{center}\end{minipage}\quad\begin{minipage}{0.55\textwidth}\begin{center}
 \begin{tikzpicture}[scale=1]
\draw[-triangle 45, thick, red] (0,0) to (0,-1) ;
\draw[triangle 45 reversed-, thick, red] (4,-1) to (4,0);
\draw[-triangle 45 reversed, thick, blue] (1,0) to (1,-1) ;
\draw[triangle 45-, thick, blue] (3,-1) to (3,0);
\draw[thick, red] (0,0) to (0,-1) to [out=-90,in=180] (2,-2.5) to [out=0,in=-90] (4,-1) to (4,0) ;
\draw[thick, blue] (1,0) to (1,-1) to [out=-90,in=180] (2,-1.75) to [out=0,in=-90] (3,-1) to (3,0) ;
\node[red] at (4.5,-0.5) {$u_s-1$};
\node[blue] at (3.2,-0.5) {$u_s$};
\draw[thick,-triangle 45] (4,-3) to [out=90,in=-90] (2,0);
\draw[thick,-triangle 45] (5.5,-3) to (5.5,0);
\node at (4.5,-2.4) {$a-1$};
\node at (6,-2) {$b+1$};
\node at (-1,-2) {$\ldots$};
\node at (7,-2) {$\ldots$};
\end{tikzpicture}\end{center}
\end{minipage}
\end{enumerate}
\end{enumerate}

\begin{prop}\label{nontrivial}
	Assume that $\lambda,\mu$ are two paths such that $\lambda=\mu +D$ for some Dyck strip $D$. Hence, we are in one of the eight cases labeled from 1a) to 2d) discussed above.
	\begin{enumerate}[i)]
		\item\label{pointA} 	The corresponding $S$-diagram is a non-trivial morphism of degree one between $B_\lambda^I$ and $B_\mu^I$. In particular, it forms a basis of the one-dimensional space $\Hom^1(B_\lambda^I,B_\mu^I)$. 
		\item Similarly, by taking the flips of the diagram in \ref{pointA}) we obtain a non-trivial morphism of degree one which gives a basis of $\Hom^1(B_\mu^I,B_\lambda^I)$. 
		\item\label{pointC} These morphisms do not depend on the different possible choice of a neat order of the form as in \eqref{superneat} if we identify all the different bimodules so obtained via the isomorphisms which send $1^\otimes$ to $1^\otimes$.
	\end{enumerate}
\end{prop}
\begin{proof}
It is a straightforward computation to check that all the listed morphisms are of degree $1$. We need to show that they describe morphisms which are not trivial. 
We show it in the case 2a, the other cases are similar. We have 
\setlength{\tabcolsep}{.2cm}

\begin{center}
\begin{tabular}{r l}
			\begin{tikzpicture}[scale=1]

\node (a) at (1.5,-2) {};
\draw[-triangle 45, thick, red] (0,0) to (0,-1) ;
\draw[triangle 45 reversed-, thick, red] (4,-1) to (4,0);
\draw[-triangle 45 reversed, thick, green] (1,0) to (1,-1) ;
\draw[triangle 45-, thick, blue] (3,-1) to (3,0);
\draw[thick, red] (0,0) to (0,-1) to [out=-90,in=180] (2,-2.5) to [out=0,in=-90] (4,-1) to (4,0) ;
\draw[thick, blue] (3,0) to (3,-4);
\draw[thick, green] (1,0) to (1,-4);
\draw[thick,green,triangle 45-] (1,-3) to (1,-4);
\draw[thick,blue,triangle 45 reversed-] (3,-3) to (3,-4);
\node[red] at (4.5,-0.5) {$v-1$};
\node[blue] at (3.2,-0.5) {$u_s$};
\node[green] at (1.2,-0.5) {$v'$};
\node[violet] at (4.2,-3.5) {$v$};

\draw[thick, violet, -triangle 45 reversed] (0,-4) to (0,-3);
\draw[thick, violet, triangle 45-] (4,-3) to (4,-4) ;
\draw[thick, violet] (0,-4) to (0,-3) to [out=90,in=180] (2,-1.5) to [out=0,in=90] (4,-3) to (4,-4) ;
\draw[thick,-triangle 45] (2,-4) to [out=90,in=-90] (2,0);
\node at (2.5,-3.5) {$u_w\hspace{-3pt}-\hspace{-3pt}1$};
\node at (5,-2) {$=$};

\end{tikzpicture}
&
\begin{tikzpicture}[scale=1]
\draw[thick, violet] (0,-4) to (0,-3) to [out=90,in=180] (2,-1.5) to [out=0,in=90] (4,-3) to (4,-4) ;
\node (a) at (1.5,-2) {};
\draw[-triangle 45, thick, red] (0,0) to (0,-1) ;
\draw[triangle 45 reversed-, thick, red] (4,-1) to (4,0);
\draw[-triangle 45 reversed, thick, green] (1,0) to (1,-1) ;
\draw[triangle 45-, thick, blue] (3,-1) to (3,0);
\draw[thick, red] (0,0) to (0,-1) to [out=-90,in=180] (2,-2.5) to [out=0,in=-90] (4,-1) to (4,0) ;
\draw[thick, blue] (3,0) to (3,-4);
\draw[thick, green] (1,0) to (1,-4);
\draw[thick,green,triangle 45-] (1,-3) to (1,-4);
\draw[thick,blue,triangle 45 reversed-] (3,-3) to (3,-4);
\node[red] at (4.5,-0.5) {$v-1$};
\node[blue] at (3.2,-0.5) {$u_s$};
\node[green] at (1.2,-0.5) {$v'$};
\node[violet] at (4.2,-3.5) {$v$};

\draw[thick, violet, -triangle 45 reversed] (0,-4) to (0,-3);
\draw[thick, violet, triangle 45-] (4,-3) to (4,-4) ;
\draw[thick,-triangle 45] (2,-4) to [out=90,in=-120] (4.8,-2.3) to [out=60,in=-60] (4.8,-1.7) to [out=120,in=-90] (2,0);
\node at (5,-3.3) {$u_w\hspace{-3pt}-\hspace{-3pt}1$};
\draw[gray, dashed] (-0.5,-2.7) -- (5.5,-2.7);
\draw[gray, dashed] (-0.5,-1.3) -- (5.5,-1.3);
\end{tikzpicture}
\end{tabular}
\end{center}
Because the bottom and the top third in diagram on the right are isomorphisms we can neglect the arrow labeled by $u_w-1$.
By isotopy, it is enough to check that morphism corresponding to the following diagram is not trivial.

\begin{center}
	\begin{tabular}{c c}
\begin{tikzpicture}[scale=1]

\node (a) at (1.5,-2) {};
\draw[-triangle 45, thick, red] (0,0) to (0,-1) ;
\draw[triangle 45 reversed-, thick, red] (3,-1) to (3,0);
\draw[-triangle 45 reversed, thick, green] (1,0) to (1,-1) ;
\draw[triangle 45-, thick, blue] (2,-1) to (2,0);
\draw[thick, red] (0,0) to (0,-2) to [out=-90,in=180] (1.5,-3.5) to [out=0,in=-90] (3,-2) to (3,0) ;
\draw[thick, blue] (2,0) to (2,-4);
\draw[thick, green] (1,0) to (1,-4);
\draw[thick,green,triangle 45-] (1,-3) to (1,-4);
\draw[thick,blue,triangle 45 reversed-] (2,-3) to (2,-4);
\draw[thick, violet, -triangle 45 reversed] (-1,0) to (-1,-1);
\draw[thick, violet, -triangle 45] (4,0) to (4,-1) ;
\draw[thick, violet] (-1,-1) to (-1,-3) to [out=-90,in=180] (-0.5,-3.5) to [out=0,in=-90] (0,-3) to [out=90,in=180] (1.5,-1.5) to [out=0,in=90] (3,-3) to [out=-90,in=180] (3.5,-3.5) to [out=0,in=-90] (4,-3) to (4,-1);
\node at (5,-2) {$=$};
\end{tikzpicture} &
\begin{tikzpicture}[scale=1]

\node (a) at (1.5,-2) {};
\draw[-triangle 45, thick, red] (0,0) to (0,-1) ;
\draw[triangle 45 reversed-, thick, red] (3,-1) to (3,0);
\draw[-triangle 45 reversed, thick, green] (1,0) to (1,-1) ;
\draw[triangle 45-, thick, blue] (2,-1) to (2,0);
\draw[thick, red] (0,0) to (0,-2) to [out=-90,in=180] (1.5,-3.5) to [out=0,in=-90] (3,-2) to (3,0) ;
\draw[thick, blue] (2,0) to (2,-4);
\draw[thick, green] (1,0) to (1,-4);
\draw[thick,green,triangle 45-] (1,-3) to (1,-4);
\draw[thick,blue,triangle 45 reversed-] (2,-3) to (2,-4);
\draw[thick, violet, -triangle 45 reversed] (-1,0) to (-1,-1);
\draw[thick, violet, -triangle 45] (4,0) to (4,-1) ;
\draw[thick, violet] (-1,-1) to [out=-90,in=180] (1.5,-2.3) to [out=0,in=-90] (4,-1);
\draw[gray, dashed] (-1.5,-2.5) -- (4.5,-2.5);
\end{tikzpicture}

\end{tabular}
\end{center}

\setlength{\tabcolsep}{20pt}
Now we can check that this defines a non-trivial morphism by computing the image of $1^\otimes$. Notice that the morphism corresponding to the bottom half of the diagram, up to the dashed line, sends $1^\otimes$ to $1^\otimes$. The top half is given by the composition of a clockwise cup and four left pointing crossings. The image of $1^\otimes$ under a clockwise cup is non trivial.
It follows from \cite[Relation 1.12]{ESW} that a left pointing crossing represents an injective morphisms of bimodules. This concludes our proof.

Hence, they are the generators of the one-dimensional space of morphisms of degree $1$.
The second statement immediately follows from the first one.

For the third statement, notice that if we choose a different order of the peaks $p_i$'s, then we will obtain a new box $\hat{\frP}$ replacing $\frP$. However, by \Cref{small}, the bimodules corresponding to the diagrams in the boxes $\frP$ and $\hat{\frP}$ are isomorphic. Similar statements hold when we change the orders of the peaks $q_i$'s or $r_i$'s. 
\end{proof}

\begin{definition}\label{fd}
	If $\lambda=\mu+D$ for a Dyck strip $D$, we denote by $f_D$ the map $f_D:B_\mu^I\raw B_\lambda^I$ defined above. We denote by $g_D:B_\lambda^I\raw B_\mu^I$ obtained by taking the flip of the diagram of $f_D$.
\end{definition}

We have shown in \Cref{nontrivial}.\ref{pointC}) that the morphism $f_D$ does not depend on the choice of a neat order as long as it is of the form as in \eqref{superneat}. The same holds for the flipped morphism $g_D$. To see this, it is enough to show that the flip of an isomorphism which sends $1^\otimes$ to $1^\otimes$, also sends $1^\otimes$ to $1^\otimes.$
 For this purpose, we first need to introduce intersection forms on generalized Bott--Samelson bimodules, generalizing the definition from \cite[\S 3.4]{EW1}.

\begin{definition}
	Let $(\vec{I},\vec{J})$ be a translation pair with $\vec{I}=(\emptyset,I_1,\ldots,I_k)$ and $\vec{J}=(J_1,\ldots,J_k)$. We define
 $\ctop(\vec{I},\vec{J})\in BS(\vec{I},\vec{J})$ to be the image of $1$ under the composition
\begin{equation}\label{tantef}R\xra{f_{\carcd}} R\otimes_{R^{J_k}}R^{I_k}\xra{f_{\carcd}\otimes \Iden} \ldots \raw BS((\emptyset,I_2,\ldots,I_k),(J_2,\ldots,J_k)))\xra{f_{\carcd}\otimes \Iden} BS(\vec{I},\vec{J}).
\end{equation}

The element $\ctop(\vec{I},\vec{J})$ is homogeneous of degree $\ell(\vec{I},\vec{J})$. We can always find a left $R$-basis $\calC$ of  $BS(\vec{I},\vec{J})$ such that  $\ctop(\vec{I},\vec{J})\in \calC$ and all the other elements in $\calC$ have degree less than $\ell(\vec{I},\vec{J})$. 

	We define the \emph{intersection form} on $BS(\vec{I},\vec{J})$ to be the left invariant bilinear form defined by 
	\[\langle x,y\rangle_{(\vec{I},\vec{J})} =\Trace(x\cdot y).\]
	where $\cdot$ is the component-wise multiplication and $\Trace:BS(\vec{I},\vec{J})\raw R$ is the linear operator such that $\Trace(b)$ is the coefficient of $\ctop(\vec{I},\vec{J})$ when we express $b$ in a basis $\calC$ as above. The operator $\Trace$ does not depend on the particular basis $\calC$ chosen.
\end{definition}

	We sketch here the proof of the following Lemma. A more detailed treatment of intersection form on generalized Bott--Samelson can be found in the \Cref{appendix}.

\begin{lemma}\label{flippreserves1}\label{adjoint}
	Let $(\vec{I},\vec{J})$ and $(\vec{I'},\vec{J'})$ be two translation pairs and let $\varphi :BS(\vec{I},\vec{J})\raw BS(\vec{I'},\vec{J'})$ be a morphism. 
	\begin{enumerate}[i)]
		\item The morphism $\bar{\varphi}$ is the adjoint of $\varphi$ with respect to the intersection forms $\form_{(\vec{I},\vec{J})}$ and $\form_{(\vec{I'},\vec{J'})}$
	\item Assume now that $(\vec{I},\vec{J})$ and $(\vec{I'},\vec{J'})$ are reduced with the same endpoint $w\in W^I$ and that $\varphi(1^\otimes)= 1^\otimes$.
	Then also $\bar{\varphi}(1^\otimes)=1^\otimes$.
	\end{enumerate}
\end{lemma}
\begin{proof}
	For the first part,  it is enought to consider the case when $\varphi$ is one of the building boxes in \Cref{boxes}. It is straightforward to check the statement when $\varphi=f_{\carcd},f_{\carad}$ or $f_{\duarrow}$.

For the second part, since $\varphi(1^\otimes)=1^\otimes$, the morphism $\varphi$ must be of degree $0$, and so does $\bar{\varphi}$. Then $\bar{\varphi}(1^\otimes)=c 1^\otimes$ for some scalar $c$.
	 	
	 	 Let $\rho\in (\frh^*)^I$ and $k=\ell(w)$. Then $\Trace(1^\otimes\cdot \rho^k)$ does not depend on the reduced translation pair but only on $w$. In fact, we have $\Trace(1^\otimes \cdot\rho^k)=\partial_w(\rho^k)$. It follows that
	 \[\partial_w(\rho^k)=\langle \varphi(1^\otimes), 1^\otimes\cdot  \rho^k\rangle_{(\vec{I'},\vec{J'})} =\langle 1^\otimes\rho^k, \bar{\varphi}(1^\otimes) \rangle_{(\vec{I},\vec{J})}=c\partial_w(\rho^k).\] 
	 We can choose $\rho\in (\frh^*)^I$ to be ample (cf. \cite[\S 5]{Pat4}). In this case we have $\partial_w(\rho^k)>0$, hence $c=1$.
\end{proof}

\subsection{Singular Rouquier complexes and Morphisms of degree two}\label{DegreeTwoMorphisms}

Let $\calC^b(\sbim^I)$ be the bounded category of complexes of $I$-singular Soergel bimodules and $\calK^b(\sbim^I)$ be the corresponding homotopy category.
For $\mu\in \Lambda_{n,i}$ we recall the definition of the singular Rouquier complex $E_\mu^I$ from \cite[\S 4]{Pat4}.
If $x\in W^I$ is such that $\mu=\lambda^x$ then $E_\mu^I$ is the minimal complex in $\calC^b(\sbim^I)$ of the restriction of the (ordinary) Rouquier complex $E_x\in \calK^b(\sbim)$ to a complex of $(R,R^I)$-bimodules. We can write
\[E_\mu^I=[\ldots {}^{-2}E_\mu^I\raw {}^{-1}E_\mu^I \raw {}^0E_\mu^I(=B_\mu^I)\raw 0]\in \calC^b(\sbim^I).\]
The singular Rouquier complex is perverse \cite[Theorem 4.13]{Pat4}, i.e. for any $i$ the bimodule ${}^{-i}E_\mu^I(i)$ is a direct sum of indecomposable self-dual singular Soergel bimodules. 
Moreover, by \cite[Remark 4.14]{Pat4} the indecomposable Soergel bimodules occurring in each degree are given by the inverse $I$-parabolic KL polynomials:
\[{}^{-i}E_\mu^I\cong \bigoplus_{\substack{\lambda \in \Lambda_{n,i}\\g_{\lambda,\mu}^I(v)=v^i}} B_\lambda^I(-i).\]

Recall by \Cref{SZJ} that $g_{\lambda,\mu}^I(v)$ can be obtained by counting Dyck partitions in $\Conf^2(\lambda,\mu)$.

\begin{remark}\label{firstdiff}
	Following \cite[Lemma 4.15]{Pat4}, we can fully describe the first differential of the Rouquier complex $d_{-1}^\mu:{}^{-1}E_\mu^I\raw B_\mu^I$. For any Dyck strip $D$ that can be removed from $\mu$ there is a summand $B_{\mu-D}^I(-1)\cus {}^{-1}E_\mu^I$.\footnote{We use $B\cussmall B'$ to denote ``$B$ is a direct summand of $B'$.''} Let $f_D:B_{\mu-D}^I\raw B_\mu$ be the morphism of degree $1$ from \Cref{fd}. Then, there exists an isomorphism 
	\[ \phi: \bigoplus_D B_{\mu-D}^I(-1)\xra{\sim} {}^{-1}E_\mu^I,\]
	where $D$ runs over all Dyck strips that can be removed from $\mu$, such that $d_\mu^{-1}\circ \phi=\bigoplus_D f_D$.
\end{remark}

Let $\lambda$ and $\mu$ be two paths with $\lambda<\mu$. Assume that there exists a Dyck partition $\{C,D\}\in \Conf^2(\lambda,\mu)$ with two elements. Notice that this implies that $B_{\lambda}^I(-2)$ is a direct summand of ${}^{-2}E_{\mu}^I$.

From Soergel's Hom formula we have
\begin{equation}\label{Hom2}
\dim\Hom^2(B_\lambda^I,B_\mu^I)=\left(\begin{array}{c}\text{coeff. of }v^2\\\text{ in }h_{\lambda,\mu}^I(v)\end{array}\right)+ \left\lvert\left\{ \begin{array}{c|c}\multirow{3}{*}{$\nu\in \Lambda_{n,i}$}&\nu\leq \lambda\text{ and there exist}\\ &\text{ Dyck strips }T_1\text{ and }T_2\text{ with }\\ &
\nu+T_1=\lambda\text{ and }\nu+T_2=\mu\end{array}\right\}\right\rvert.
\end{equation}
To compute $\dim \Hom^2(B_{\lambda}^I,B_\mu^I)$ we thus need to classify all Dyck partitions in $\Conf^1(\lambda,\mu)$ with two elements and 
all the ways in which one can obtain $\calA(\lambda,\mu)$ as a difference of two Dyck paths.
 We divide this task into cases.

\subsubsection{The case of two overlying strips}
\label{overlyingSec}

Let $\lambda,\mu,C$ and $D$ be as above, that is $\{C,D\}\in \Conf^2(\lambda,\mu)$. Assume that there exists a box of $C$ South of a box of $D$. The condition on partitions of type 2 forces every box South of a box of $D$ is in $C$. Let $\{C',D'\}$ be any other Dyck partition of $\calA(\lambda,\mu)$. Two boxes lying one North of the other cannot belong to the same Dyck strip, so we can assume $D\cug D'$ and $C'\cug C$. 

Assume first $\hgt(C)<\hgt(D)$. Then, since $\hgt(D)\leq \hgt(D')$, we must have $D'=D$ and $C'=C$. In particular there are no Dyck partitions in $\Conf^1(\lambda,\mu)$ with exactly two elements.

\begin{center}
\begin{tikzpicture}[x=\boxmini,y=\boxmini]
\tikzset{vertex/.style={}}
\tikzset{edge/.style={very thick}}
\tabpath{+,+,-,+,-,+,-,-}
\tabpath{-,-,-,+,-,+,+,+}
\tikzset{edge/.style={}}
\tabpath{+,-,-,+,-,+,+,-}

\node at (5,0) {$D$};
\node at (6,-1) {$C$};
\end{tikzpicture}
\end{center}

\begin{lemma}\label{missing}
	Let $\{C,D\}\in \Conf^2(\lambda,\mu)$ be with $\hgt(C)<\hgt(D)$. If there exists a box in $C$ South of a box of $D$ then $f_D\circ f_C=0$
\end{lemma}
\begin{proof}
There exist no strips $T_1$ and $T_2$ such that $T_1\sqcup D\sqcup C=T_2$.\footnote{We use $\sqcup$ to denote the disjoint union.}
It follows from \eqref{Hom2} that $\Hom^2(B_\lambda^I,B_\mu^I)=0$, so $f_D\circ f_C=0$.
\end{proof}

Assume now $\hgt(C)\geq \hgt(D)$. Then there exists a unique Dyck partition of type 1 obtained by taking as $C'$ the set of all boxes South of a box of $D$ and as $D'$ its complement.

\begin{figure}[H]
\begin{center}
\begin{tikzpicture}[x=\boxmini,y=\boxmini]
\tikzset{vertex/.style={}}%
\tikzset{edge/.style={very thick}}%

\tabpath{+,-,-,-,+,-,+,-,+,-,+,+,+,-}
\tabpath{-,-,-,-,-,-,+,-,+,+,+,+,+,+}%
\tikzset{edge/.style={}}%
\tabpath{+,-,-,-,-,-,+,-,+,+,+,+,+,-}%
\tikzset{edge/.style={dotted}}%
\node at (6,-3) {$D$};
\node at (6,-5) {$C$};
\end{tikzpicture}\quad
\begin{tikzpicture}[x=\boxmini,y=\boxmini]
\tikzset{vertex/.style={}}%
\tikzset{edge/.style={very thick}}%

\tabpath{+,-,-,-,+,-,+,-,+,-,+,+,+,-}
\tabpath{-,-,-,-,-,-,+,-,+,+,+,+,+,+}%
\tikzset{edge/.style={}}%
\tabpath{-,-,-,-,+,-,+,-,+,-,+,+,+,+}%
\tikzset{edge/.style={dotted}}%
\node at (4,-3) {$D'$};
\node at (6,-5) {$C'$};
\end{tikzpicture}
\caption{In this example we see how, if $\hgt(C)\geq \hgt(D)$, there exists a unique Dyck partition of Type 1 $\{D',C'\}\in \Conf^1(\lambda,\mu)$ with $D\subset D'$ and $C'\subset C$.}\label{fromIItoI}
\end{center}
\end{figure}

\begin{lemma}\label{overlying}
Let $D,C,D',C'$ be as above. Then the morphism $f_C\circ f_D$ is not trivial and we have $f_C\circ f_D=f_{D'}\circ f_{C'}$.
\end{lemma}
\begin{proof}
From \eqref{Hom2} we have $\dim \Hom^2(B_\mu^I,B_\lambda^I)=1$, so it is clear that $f_D\circ f_C$ is equal to $f_{D'}\circ f_{C'}$ up to a scalar. To see that they actually coincide notice that we can choose a neat ordering of $\Peaks(\lambda)$ of the form
\[(p_1,p_2,\ldots,p_m,q_1,q_2,\ldots q_{m'},a_D,b_D,r_1,r_2,\ldots,r_{m''},a_C,b_C,\tilde{r}_1,\tilde{r}_2,\ldots,\tilde{r}_{m'''})\]
where $a_D,b_D$ are the extreme boxes of $D$ and $a_C$ and $b_C$ are the extreme boxes of $C$ and $\hgt(p_i)<\hgt(a_C)\leq \hgt(r_j)$ for any $i$ and $j$.
Now it follows from the explicit description of the morphisms above that starting from $f_D\circ f_C$, we can apply an isotopy that ``slides'' $f_D$ down and $f_C$ up to obtain $f_{D'}\circ f_{C'}$.

Finally, since $\Hom^2(B_\mu^I,B_\lambda^I)$ is generated by $f_C\circ f_D$($=f_{D'}\circ f_{C'}$) by \cite[Lemma 4.15]{Pat4}, the morphism $f_C\circ f_D$ must be non-trivial.
\end{proof}

\subsubsection{The case of two distant strips}

Assume now that no box in $D$ is North of or South of a box in $C$. In this case we say that the two strips are \emph{distant}. Notice that in this case we can apply $f_D$ and $f_C$ in any order. However, as we will see, in general we have $f_D\circ f_C\neq f_C\circ f_D$.

 It is evident that the only Dyck partition of $\calA(\lambda,\mu)$ with two elements into two strips is $\{C,D\}$.
There exists a Dyck strip $T_1$ that can be removed from $\lambda$ such that $T_2=D\sqcup T_1\sqcup C$ is also a Dyck strip if and only if $\hgt(D)=\hgt(C)$ and there are no peaks in $\lambda$ between $D$ and $C$ of height at least $\hgt(D)$.

\begin{center}
\begin{tikzpicture}[x=\boxmini,y=\boxmini]
\tikzset{vertex/.style={}}%
\tikzset{edge/.style={very thick}}%

\tabpath{+,-,-,+,+,-,-,+,+,-,+,-}
\tabpath{-,-,-,+,+,-,-,+,+,-,+,+}
\tabpath{-,-,-,+,+,+,-,+,-}
\node at (3,-2) {$D$};
\node at (7,-2) {$T_1$};
\node at (10,-1) {$C$};
\end{tikzpicture}
\end{center}

Assume that there does not exist such a strip $T_1$. Then, by \eqref{Hom2}, $\dim\Hom^2(B_\mu^I,B_\lambda^I)=1$ and there exists $c\in \bbQ$ so that $f_D\circ f_C\cong cf_C\circ f_D$. In order to show that $c\in \bbQ^*$ it is enough to show that both $f_D\circ f_C$ and $f_C\circ f_D$ are non-zero. We will show this in \Cref{RQexplicit}. 

Assume now that there exists such a $T_1$. Then $\dim \Hom^2(B_\lambda^I,B_\mu^I)=2$. In general, as the next example shows, the two morphisms $f_D\circ f_C$ and $f_C\circ f_D$ need not even be multiple of each others. 

\begin{example}\label{ex132}
	Let $n=4$ and $i=2$. Let $\lambda=(\dpath,\upath,\dpath,\upath)$ and $\mu=(\upath,\dpath,\upath,\dpath)$ and let $D$ and $C$ be the single-box Dyck strips as in the following picture.
	
	\begin{center}
		\begin{tikzpicture}[x=0.7cm,y=0.7cm]
		\tikzset{vertex/.style={}}
		\tikzset{edge/.style={very thick}}
		\tabpathc{+,-,+,-}{blue}
		\tabpath{-,-,+,+}
		\tabpathc{-,+,-,+}{red}
		\node at (1,0) {$D$};
		\node at (3,0) {$C$};
		\node at (2,-1) {$T_1$};
		\node[blue] at (4,1) {$\mu$};
		\node[red] at (4,-1) {$\lambda$};
		
		\end{tikzpicture}
	\end{center}
	
	We have $\lambda+D=(\upath,\dpath,\dpath,\upath)$ and $\lambda+C=(\dpath,\upath,\upath,\dpath)$. Let $T_1$ be the Dyck strip consisting of the only box in $\calA(\lambda)$. Notice that $T_2:=C\sqcup D\sqcup T_1$ is a Dyck strip.
	From \Cref{firstdiff} we deduce that the Rouquier complex $E_\mu^I$ is isomorphic to the following complex.
	
	\begin{center}
		\begin{tikzpicture}
		\node at (0,3) {${}^{-2}E_\mu^I$};
		\node at (3,3) {${}^{-1}E_\mu^I$};
		\node at (6,3) {${}^{0}E_\mu^I$};
		\node at (-3,3) {${}^{-3}E_\mu^I$};
		\node (a) at (0,0) {$B_\lambda^I(-2)$};
		\node (b1) at (3,1.5) {$B_{\lambda+D}^I(-1)$};
		\node (b2) at (3,0) {$B_{\lambda+C}^I(-1)$};
		\node (b3) at (3,-1.5) {$B_{id}^I(-1)$};
		\node (c) at (6,0) {$B_\mu^I$};
		\node at (3,0.75) {$\oplus$};
		\node at (3,-0.75) {$\oplus$};
		\node (d) at (-3,0) {$0$};
		\node (e) at (9,0) {$0$};
		
		\path[->] (a) edge node[above] {$c_1\cdot f_D\;\;$}(b1);
		\path[->] (a) edge node[above] {$c_2\cdot f_C$}(b2);
		\path[->] (a) edge node[below] {$c_3 \cdot g_{T_1}\;\;$}(b3);
		\path[->] (b1) edge node[above] {$f_C$}(c);
		\path[->] (b2) edge node[above] {$f_D$}(c);
		\path[->] (b3) edge node[above] {$f_{T_2}$}(c);
		\path[->] (d) edge (a);
		\path[->] (c) edge (e);
		\end{tikzpicture}
	\end{center}
	with $c_1,c_2,c_3\in \bbQ$. 
	Since $H^{-2}(E_\mu^I)=0$ we have that $c_1f_D(1^\otimes)+c_2f_C(1^\otimes)+c_3g_{T_1}(1^\otimes)\neq 0$. Since $g_{T_1}(1^\otimes)=0$,
	one between $c_1$ and $c_2$ must be non-zero. We can assume $c_1\neq 0$. Now a simple computation shows that $f_C\circ f_D(1^\otimes)\neq 0$. Since  the differential must square to zero, we see that also $c_2\neq 0$.	We conclude that $c_1 f_D\circ f_C-c_2f_C\circ f_D=c_3f_{T_2}\circ g_{T_1}$ with $c_1,c_2\neq 0$. 
	
	One can actually compute the maps above using the explicit descriptions of \Cref{diagramsec} and check that $f_D\circ f_C$ is not a scalar multiple of $f_C\circ f_D$, i.e. that also $c_3\neq 0$. However, notice that $f_{T_2}\circ g_{T_1}$ factors through $B_{id}^I$.
\end{example}

	For $B,B'\in \sbim^I$ we denote by $\Hom_{\leq \lambda}(B,B')$ (resp. $\Hom_{< \lambda}(B,B')$ ) the subspace of $\Hom(B,B')$ generated by morphisms which factor through  $B_\mu^I$ for $\mu\leq \lambda$ (resp. $\mu<\lambda$). Let $\Hom_{\not< \lambda}(B,B'):=\Hom_{\leq \lambda}(B,B')/\Hom_{< \lambda}(B,B')$.

We record the general behavior, which is well illustrated by the example above, in the following Proposition. A proof will be given in the next section (\S \ref{RQexplicit}).

\begin{prop}\label{distantDp}
Let $D$ and $C$ be two distant Dyck strips that can be removed from a path $\lambda$. 
\begin{enumerate}[i)]
\item\label{dist1} If $C \sqcup D$ cannot be obtained as the difference of two Dyck strips then there exists $c\in \bbQ^*$ such that $f_D\circ f_C=c\cdot f_C\circ f_D$.
\item\label{dist2} If $\hgt(D)=\hgt(C)$ then there exists $c\in \bbQ^*$ such that \[f_D\circ f_C=c\cdot f_C\circ f_D\in \Hom_{\not <\lambda}^\bullet(B_\lambda^I,B_\mu^I).\]
\end{enumerate}
\end{prop}

It is also important to study the composition between a morphism $f_C$ and the flipped morphism $g_D$, when $C$ and $D$ are distant.

\begin{lemma}\label{distneg}
	Assume that $D$ and $C$ are distant strips. Then, there exists $c\in \bbQ^*$ such that \[f_D\circ g_C=c\cdot g_C\circ f_D:B^I_{\lambda+C}\raw B^I_{\lambda+D}.\]

\end{lemma} 
\begin{proof}
	From \cite[Lemma 4.15]{Pat4}, every map of degree $2$ between $B^I_{\lambda+C}$ and $B^I_{\lambda+D}$ can be realized as a sum of maps factoring through $B_{\mu+C}^I\xra{g_R} B^I_{\mu+C-R}\raw B^I_{\mu+D}$ for some Dyck strip $R$. This can happen only if there exists a Dyck strip $T$ such that $\mu+D=\mu+C-R+T$.

	We claim that the only possibility is $T=D$ and $R=C$.	In fact, we have $T\sqcup C= D\sqcup R$, so $D\cug T$ and $C\cug R$. Let $U:=T\setminus D=R\setminus C$.
	We can assume that $C$ lies on the right of $D$, so $U$ lies on the left side of $C$ and on the right side of $D$. 
	In particular the set of boxes $U$ is either empty or connected. Assume that $U$ is not empty (otherwise $T=D$). 
	
	\begin{center}
		\begin{tikzpicture}[x=\boxmini,y=\boxmini]
		\tikzset{vertex/.style={}}%
		\tikzset{edge/.style={very thick}}%
		\draw[fill=orange] (6,0) -- (8,-2) -- (10,0) -- (11,-1) -- (12,0) -- (13,-1) -- (11,-3) -- (10,-2) -- (8,-4) -- (5,-1) -- cycle;
		\tabpath{+,-,-,+,+,-,-,-,+,+,-,+,-}
		\tabpath{-,-,-,+,+,-,-,-,+,+,-,+,+}
		\tabpath{-,-,-,+,+,+,-,-,+,-}
		\tabpathc{-,-,-,+,+,+,-,-,+,-,-,+,+}{blue}
		\node[orange] at (9,-4) {$R$};
		\node at (3,-2) {$D$};
		\node at (7,-2) {$U$};
		\node at (12,-1) {$C$};
		\node[blue] at (0,-1) {$\mu$};
 		\end{tikzpicture}
	\end{center}

	Then the left and the right extreme boxes of $R$ are resp. in $U$ and in $C$, so in particular $\hgt(R)=\hgt(C)$. 
	Similarly, the right extreme box of $U$ must be an extreme box of $T$. But this box has height $\hgt(C)-1$, hence $T$ cannot be a Dyck strip since its two extreme boxes have different heights.
\end{proof}

\subsection{Description of Rouquier complexes}\label{RQexplicit}

The \Cref{ex132} above exhibits how the knowledge of singular Rouquier complexes is a useful tool for computations involving the maps $f_D$'s.
The goal of this section is to generalize this \Cref{ex132} and give a partial description of singular Rouquier complexes. This will allow us to prove \Cref{distantDp} in general.


In \Cref{firstdiff} we have recalled how to describe the first differential $d_{-1}^{\mu}:{}^{-1}E_\mu^I\raw B_{\mu}^I$. 
We now focus on the second differential $d_{-2}^\mu:{}^{-2}E_\mu^I\raw {}^{-1}E_\mu^I$. 

\begin{prop}\label{distantprop}
	Let $\nu<\lambda<\mu \in \Lambda_{n,i}$ be such that $B_\nu^I(-2)\cus {}^{-2}E_\mu^I$ and $B_\lambda^I(-1)\cus {}^{-1}E_\mu^I$. Assume there exists two distant Dyck strips $C$ and $D$ such that $\nu+C=\lambda$ and $\lambda+D=\mu$. Then the second differential $d^{\mu}_{-2}$ of $E_\mu$ induces a non-trivial map between $B_\nu^I(-2)$ and $B_\lambda^I(-1)$.
\end{prop}

\begin{remark}
	\Cref{distantprop} remains valid also when $C$ and $D$ are not distant. However, we only consider distant strips as this is the only case we actually need.
\end{remark}
	


The proof of \Cref{distantprop} will take most of the remainder of this section.
 Our strategy is to construct a new complex $\tilde{E}\in \calC^b(\sbim^I)$ such that its minimal complex is $E_\mu^I$ and it is easier to study than $E_\mu^I$ itself.
 We first need the following simple fact about complexes. This is based on \cite[\S 6.1, Footnote 3]{EW1}.

\begin{lemma}\label{subcomplex}
	Let $E\in \calC^b(\sbim^I)$ and let $E_1$ be a subcomplex of $E$. Assume that $E_1=E_2\oplus T$ with $T$ contractible. Then there exists a decomposition $E=\tilde{E}\oplus T'$, with $T'\cong T$ and such that $E_2$ is a subcomplex of $\tilde{E}$.
\end{lemma}
\begin{proof}
	The complex $T$ is contractible, so in $\calC^b(\sbim^I)$ there exists a decomposition
	$T\cong \bigoplus_{k} [ {B_{x_k}^I(m_k)}\isom B_{x_k}^I(m_k)]$. By removing one summand of $T$ at a time, we can restrict ourselves to the case $T= [{B}\isom B]$ for some $B\in \sbim^I$. So, there exists $n\in \bbZ$ such that ${}^iT=0$ unless $i=n$ or $i=n+1$.
	
	For any $i$ we have a decomposition ${}^iE={}^iE_1\oplus {}^iV={}^iE_2\oplus {}^iT\oplus {}^iV$ for some ${}^iV\in \sbim^I$. With respect to this decomposition, the differential $d_i^E:{}^iE_2\oplus {}^iT\oplus {}^iV\raw {}^{i+1}E_2\oplus {}^{i+1}T\oplus {}^{i+1}V$ has the form:
	\[ d_i^E=\begin{pmatrix}
		d_i^{E_2} & 0 & \alpha_i \\ 0 & d_i^{T} & \beta_i \\ 0 & 0 & \gamma_i
	\end{pmatrix}.\]
	Now $d_i^T=0$ if $i\neq n$ and $d_n^T$ is an isomorphism. If we apply the automorphism
	\[ \begin{pmatrix}
		\Iden& 0 & 0 \\ 0 & \Iden & -(d_n^T)^{-1}\beta_n \\ 0 & 0 & \Iden
	\end{pmatrix}\]
	on ${}^{n}E$, we obtain a new decomposition, say ${}^nE={}^nE_2\oplus {}^{n}T\oplus {}^{n}\tilde{V}$, with ${}^n\tilde{V}\cong {}^nV$.
	With respect to this new decomposition the differential
	$d_n^E$ has the form
	\[ d_i^E=\begin{pmatrix}
		d_n^{E_2} & 0 & \alpha_n \\ 0 & d_n^{T} &0 \\ 0 & 0 & \gamma_n
	\end{pmatrix}.\]
	Now is evident that $T$ splits from $E$, so we can define a complex $\tilde{E}= E_2\oplus \tilde{V}$, so that we have $E= \tilde{E}\oplus T$ and $E_2$ is a subcomplex of $\tilde{E}$.
\end{proof}

\begin{definition}
We call $R(D)$ the intersection of $\calA(\mu)$ with the smallest square (with sides tilted by $45^\circ$ from the axes) containing the Dyck strip $D$.

Let $\tilde{J_D}$ be the set of simple reflections occurring as labels of boxes of $D$ (or, equivalently, of $R(D)$). Let $j_D$ be the label of southernmost box of $R(D)$ and let $J_D=\tilde{J_D}\setminus \{j_D\}$.

	\begin{figure}
		\begin{center}
	\begin{tikzpicture}[x=\boxmini,y=\boxmini]
	\tikzset{vertex/.style={}}%
	\tikzset{edge/.style={very thick}}%
	
	\draw[fill=yellow] (4,2) -- (6,0) -- (7,1) -- (10,-2) -- (12,0) -- (13,-1) -- (10,-4) -- cycle;
	\draw[fill=orange] (4,2) -- (5,3) -- (6,2) -- (7,3) -- (10,0) -- (12,2) -- (13,1) -- (15,3) -- (16,2) -- (13,-1) -- (12,0) -- (10,-2) -- (7,1) -- (6,0) -- (4,2) -- cycle;
	\draw[fill=orange] (20,2) -- (21,3)-- (22,2) -- (23,3) -- (24,2) -- (22,0) -- cycle;
	\tabpathc{+,-,+,+,+,-,+,-,-,-,+,+,-,+,+,-,-,-,+,+,+,-,+,-}{blue}
	\begin{scope}[shift={(0,-0.1)}]
	\tabpathc{+,-,+,+,-,-,-,-,-,-,+,+,+,+,+,+,-,-,+,+,-,-,+,+}{red}
	\end{scope}
	\begin{scope}[shift={(0,-0.2)}]
	\tabpathc{+,-,+,+,-,-,+,-,-,-,+,+,-,+,+,+,-,-,+,+,-,-,+,+}{green}
	\end{scope}
	\tabpath{-,-,-,-,-,-,-,-,-,-,-,+,+,+,+,+,+,+,+,+,+,+,+,+}

	\node at (11,-1) {$R(D)$};
	\node at (9,0) {$D$};
	\node at (23,2) {$C$};
	\node at (26,2) {$=R(C)$};
	\node at (10.1,-3) {$j_D$};
	\node at (22.1,1) {$j_C$};
	\node[red] at (13,-2.5) {$\kappa$};
	\node[green] at (17,-0.5) {$\nu$};
	\node[blue] at (9,2) {$\mu$};
	\end{tikzpicture}
\end{center}
\caption{In this figure the path $\mu$ is depicted in blue, the path $\nu$ in green and the path $\kappa$ in red. The Dyck strips $D$ and $C$ are colored in orange and the region $R(D)$ not covered by $D$ in yellow. In this example we have $R(C)=C$.}  
\end{figure}

\end{definition}

Notice that, since the strips $C$ and $D$ are distant, they can be removed from $\mu$ in any order and we have $\nu=\mu-D-C$. Moreover, $R(D)$ and $R(C)$ are disjoint.

\begin{proof}[Proof of \Cref{distantprop}.]

Let $\kappa$ be the path obtained by removing from $\mu$ all the boxes in $R(D)$ and $R(C)$.

Let $J=J_D\cup J_C$. Notice that every $s \in J$ is on a slope of $\kappa$. It follows that $B_\kappa^I=R\otimes_{R^J}{}^JB_\kappa^I$.

We can regard $R(D)$ (resp. $R(C)$) as a subtableau, and define $x_{R(D)}$ (resp. $x_{R(C)}$) to be the corresponding element in $W_{\tilde{J_D}}^{J_D}\cug W$ ( resp. $W_{\tilde{J_C}}^{J_C}\cug W$).
Consider the following complex: 
\[E:=E_{x_{R(C)}}\otimes_R E_{x_{R(D)}}\otimes_R E_\kappa^I\in \calC^b(\sbim^I).\] In general the complex $E$ is not minimal, and the minimal complex of $E$ is $E_\mu^I$.
For any $i\geq 0$, the bimodule ${}^{-i}E$ is a direct sum of bimodules of the form \[B_{y_{R(C)}}B_{y_{R(D)}}B_{\kappa'}^I(-i),\qquad\text{ with }\quad y_{R(C)}\leq x_{R(C)},\; y_{R(D)}\leq x_{R(D)}\text{ and }\kappa'\leq \kappa.\]
 
Consider the following subcomplex of $E$:
\[E_1:=E_{x_{R(C)}}\otimes_R E_{x_{R(D)}}\otimes_R B_\kappa^I\in \calC^b(\sbim^I).\]
Since $J_D$ and $J_C$ are distant, we have $E^J_{x_{R(D)}}\cong (E^{J_D}_{x_{R(D)}})_J\in \calC^b(\sbim^J)$. Moreover, there exists a complex of $(R^{J_C},R^J)$-bimodules ${}^{J_C}E_{x_{R(D)}}^J$ such that 
\[E_{x_{R(D)}}^J\cong R\otimes_{R^{J_C}}{}^{J_C}E_{x_{R(D)}}^J\in \calC^b(\sbim^J).\]
In fact, for any $i$ we have 
\[{}^{-i}E^J_{x_{R(D)}}\cong \bigoplus_{k} B^J_{z_k}(-i),\]
with $z_k\in W_{\tilde{J}_D}^J=W_{\tilde{J}_D}^{J_D}$ and since for any $z_k\leq x_{R(D)}$ we have $z_kW_J=W_{J_C}z_kW_J$, by \Cref{piusing} there exists an indecomposable bimodule ${}^{J_C}B_{z_k}^J$ such that $B_{z_k}^J\cong R\otimes_{J_C}{}^{J_C}B_{z_k}^J$. Moreover, since by \cite[Theorem 7.4.1]{W4} we have
\[\Hom({}^{J_C}B_{z_k}^J, {}^{J_C}B_{z_k}^J)\cong \Hom(B_{z_k}^J, B_{z_k}^J),\]
 we can construct a complex ${}^{J_C}E^J_{x_{R(D)}}$ such that $R\otimes_{R^{J_C}}{}^{J_C}E_{x_{R(D)}}^J=E^J_{x_{R(D)}}$ and  
	\[{}^{-i}\left({}^{J_C}E^J_{x_{R(D)}}\right)\cong \bigoplus_{k} {}^{J_C}B^J_{z_k}(-i).\]
	
Consider now the complex:
\[E_2:=E_{x_{R(C)}}^{J_C}\otimes_{R^{J_C}} {}^{J_C}E_{x_{R(D)}}^{J}\otimes_{R^J} {}^JB_\kappa^I\in \calC^b(\sbim^I).\]
We have a decomposition $E_1\cong E_2\oplus T$, with $T$ a contractible summand in $\calC^b(\sbim^I)$. By \Cref{subcomplex}, since $E_1$ is a subcomplex of $E$ we can find a decomposition of complexes $E\cong \tilde{E}\oplus T$ such that $E_2$ is a subcomplex of $\tilde{E}$.
So we have $E\cong \tilde{E}\in \calK^b(\sbim^I)$ and the minimal complex of $\tilde{E}$ is again $E_\mu^I$.
For any $i$ we have
${}^{i}\tilde{E}={}^iE_2\oplus ({}^iE/{}^iE_1)$. Notice that $({}^iE/{}^iE_1)$ is a direct sum of bimodules of the form $B_{y_{R(C)}}B_{y_{R(D)}}B_{\kappa'}^I(i)$ with $\kappa'<\kappa$. In particular, by \Cref{valley}(i), all the direct summands of $({}^iE/{}^iE_1)$ are of the form $B^I_\eta(m)$, with $\eta \not \geq \kappa$.

We can now go back to the original question.
Consider the indecomposable bimodule $B_\nu^I(-2)$: it is a direct summand of ${}^{-2}E_\mu^I$ and therefore of ${}^{-2}\tilde{E}$. 

Let $x_{R(D)-D}$ be the element in $W_{\tilde{J_D}}^{J_D}$ corresponding to the tableau $R(D)-D$. Define $x_{R(C)-C}$ similarly.

\begin{claim}
	The indecomposable bimodule $B_\nu^I(-2)$ is a direct summand of 
	\[B_{x_{R(C)-C}}^{J_C}\otimes_{R^{J_C}}{}^{J_C}B_{x_{R(D)-D}}^{J}\otimes_{R^J} {}^JB_{\kappa}^I(-2)\cus {}^{-2}E_2\cus {}^{-2}\tilde{E}.\]
\end{claim}
\begin{proof}[Proof of the claim.]
 Since $\nu\geq \kappa$,  $B_\nu^I(-2)$ is not a summand of ${}^iE/{}^iE_1$. Hence, it must be a summand of ${}^{-2}E_2$. Furthermore, we see from \Cref{crucial} that the only summand of ${}^{-2}E_2$ containing $B_\nu^I(-2)$ is $B_{x_{R(C)-C}}^{J_C}\otimes_{R^{J_C}}{}^{J_C}B_{x_{R(D)-D}}^{J}\otimes_{R^J} {}^JB_{\kappa}^I(-2)$.
\end{proof}

Similarly, we have 
\[B_{\lambda}^I(-1)=B_{\mu-D}^I(-1)\cus B_{x_{R(C)}}^{J_C}\otimes_{R^{J_C}}{}^{J_C}B_{x_{R(D)-D}}^{J}\otimes_{R^J} {}^JB_{\kappa}^I(-1)\cus {}^{-1}E_2\cus {}^{-1}\tilde{E}.\]

In $E_{x_{R(C)}}^{J_C}$ the differential $d_{-1}^{x_{R(C)}}$ induces a non-trivial map between the summands $B_{x_{R(C)-C}}^{J_C}(-1)\cus {}^{-1}E_{x_{R(C)}}^{J_C}$ and $B_{x_{R(C)}}^{J_C}={}^0E_{x_{R(C)}}^{J_C}$. By \Cref{firstdiff}, we can assume that this map is $f_C$. It follows that in the complex $\tilde{E}$ there is a non-trivial map 
\[ f_C \otimes \Iden \otimes \Iden: B_{x_{R(C)-C}}^{J_C}\otimes_{R^{J_C}}{}^{J_C}B_{x_{R(D)-D}}^J\otimes_{R^J}{}^JB_{\kappa}^I(-2)\raw B_{x_{R(C)}}^{J_C}\otimes_{R^{J_C}}{}^{J_C}B_{x_{R(D)-D}}^J\otimes_{R^J}{}^JB_{\kappa}^I(-1).\]

\begin{claim}\label{claim2}
For any choice of the summands 
\begin{align*}B_{\nu}^I\cus B_{x_{R(C)-C}}^{J_C}&\otimes_{R^{J_C}}{}^{J_C}B_{x_{R(D)-D}}^{J}\otimes_{R^J} {}^JB_{\kappa}^I\\
B_{\lambda}^I\cus  B_{x_{R(C)}}^{J_C}&\otimes_{R^{J_C}}{}^{J_C}B_{x_{R(D)-D}}^{J}\otimes_{R^J} {}^JB_{\kappa}^I\end{align*}
the map $f_C \otimes \Iden\otimes \Iden$ induces a non-trivial map between $B_{\nu}^I(-2)$ and $B_{\lambda}^I(-1)$.
\end{claim}
\begin{proof}[Proof of the claim.]
Assume that there exists a decomposition such that the induced map $B_{\nu}^I(-2)\raw B_{\lambda}^I(-1)$ vanishes. Because of \Cref{crucial}, this means that $ f_C \otimes \Iden\otimes \Iden$ would vanish modulo terms smaller than $\nu$. Recall the support functor $\Gamma^I_{<\nu}$ from \cite[\S 4.1]{Pat4}.
Then we have
\[f_C({1^\otimes})\otimes{1^\otimes} \otimes {1^\otimes} =(f_C \otimes \Iden\otimes \Iden)({1^\otimes}\otimes {1^\otimes}\otimes {1^\otimes})\in \Gamma_{<
	\nu}^I\left(B_{x_{R(C)}}^{J_C}\otimes_{R^{J_C}}{}^{J_C}B_{x_{R(D)-D}}^{J}\otimes_{R^J} {}^JB_{\kappa}^I\right),\]
which is impossible since $f_C({1^\otimes})\not \in \Gamma^{J_C}_{<x_{R(C)-C}}(B^{J_C}_{x_{R(C)}})$.
\end{proof}

Now, $B_\nu^I(-2)$ and $B_\lambda^I(-1)$ both occur with multiplicity $1$ resp. in ${}^{-2}\tilde{E}$ and ${}^{-1}\tilde{E}$ and the map between them is not trivial.
It follows from \Cref{claim2} that for any choice of the minimal complex $E_{\mu}^I\cus \tilde{E}$, the map induced by the differentials between the summands $B_\nu^I(-2) \cug E_{\mu}^{-2}$ and $B_{\lambda}^I(-1)\cug E_{\mu}^{-1}$ must be non-trivial.
\end{proof}

\Cref{distantprop} allows us to finally prove \Cref{distantDp}.

\begin{proof}[Proof of \Cref{distantDp}.]
Consider the Rouquier complex $E_{\lambda}^I$. The summand $B_{\nu}^I(-2)$ occurs in ${}^{-2}E_{\lambda}^I$ while $B_{\lambda-D}^I(-1)$ and $B_{\lambda-C}^I(-1)$ occur in ${}^{-1}E_\lambda^I$. As shown above the maps between these summands are non-trivial, i.e. there exists $c_1,c_2\in \bbQ^*$ such that the first two terms of Rouquier complex $E_\lambda^I$ look like 

\begin{center}
\begin{tikzpicture}
	\node (-2) at (0,3) {${}^{-2}E_\lambda^I$};
	\node (-1) at (3,3) {${}^{-1}E_\lambda^I$};
	\node (0) at (6,3) {${}^{0}E_\lambda^I$};
	\node (-3) at (-3,3) {${}^{-3}E_\lambda^I$};
	\node (1) at (9,3) {${}^{-3}E_\lambda^I$};
	\node at (0,2.5) {\rotatebox{90}{$=$}};
	\node at (3,2.5) {\rotatebox{90}{$=$}};
	\node at (6,2.5) {\rotatebox{90}{$=$}};
	\node at (-3,2.5) {\rotatebox{90}{$=$}};
	\node (a) at (0,0) {$B_{\nu}^I(-2)$};
	\node (b1) at (3,1) {$B_{\lambda-D}^I(-1)$};
	\node (b2) at (3,-1) {$B_{\lambda-C}^I(-1)$};
	\node (b3) at (3,-3) {$\vdots$};
	\node (c) at (6,0) {$B_\lambda^I$};
	\node at (3,0) {$\oplus$};
	\node at (3,-2) {$\oplus$};
	\node at (0,-1) {$\oplus$};
	\node at (0,-2) {$\vdots$};

	\node (d) at (-3,0) {$\ldots$};
	\node (e) at (9,0) {$0$};

\path[->] (-3) edge node[above] {$d_{-3}^\lambda$} (-2);
\path[->] (-2) edge node[above] {$d_{-2}^\lambda$} (-1);
\path[->] (-1) edge node[above] {$d_{-1}^\lambda$} (0);
\path[->] (0) edge node[above] {$d_{-0}^\lambda$} (1);
 \path[->] (a) edge node[above] {$c_1\cdot f_D$}(b1);
 \path[->] (a) edge node[above] {$c_2\cdot f_C$}(b2);
 \path[dashed, ->] (a) edge (b3);
 \path[->] (b1) edge node[above] {$f_C$}(c);
 \path[->] (b2) edge node[above] {$f_D$}(c);
 \path[dashed, ->] (b3) edge (c);
	\path[->] (d) edge (a);
	\path[->] (c) edge (e);
\end{tikzpicture}
\end{center}

Since $d^\lambda_{-2}\circ d^{\lambda}_{-1}=0$ and since $\{D,C\}$ is the only Dyck partition of $\calA(\nu,\lambda)$ with two elements, it follows that $\dim \Hom^2_{\not< \nu}(B_{\nu}^I,B_\lambda)=1$ and
\[f_D\circ f_C=\frac{c_2}{c_1}f_C\circ f_D\in \Hom^2_{\not< \nu}(B_{\nu}^I,B_\lambda).\qedhere\]
\end{proof}

\subsection{The construction of the basis}

Let $\lambda<\mu$ be two paths and let $\bfP=\{D_1,D_2,\ldots,D_k\}\in \Conf^1(\lambda,\mu)$.

\begin{definition}\label{admissible}
	We say that an ordering $(D_1,D_2,\ldots,D_k)$ of $\bfP$ is \emph{admissible} if for any $j$ the Dyck strip $D_j$ can be added to $\lambda+D_1+\ldots+D_{j-1}$. We denote by $\adm(\bfP)$ the set of admissible orderings of $\bfP$. 
\end{definition}

To any admissible ordering $o=(D_{\sigma(1)},\ldots,D_{\sigma(k)})\in \adm(\bfP)$ we can associate 
 a map $f_{\bfP,o}=f_{D_{\sigma(k)}}\circ f_{D_{\sigma(k-1)}}\circ \ldots\circ f_{D_{\sigma(1)}}:B_\lambda^I\raw B_\mu^I$. In general, the map $f_{\bfP,o}$ will depend on the choice of the order $o$. The goal of this section is to show that, fixing for any Dyck partition $\bfP$ an order $o\in \adm(\bfP)$, the set $\{f_{\bfP,o}\}_{\bfP\in \Conf^1(\lambda,\mu)}$ gives a basis of $\Hom_{\not<\mu}(B_\lambda^I,B_\mu^I)$.

Given a Dyck partition $\bfP=\{D_1,D_2,\ldots,D_k\}\in \Conf^1(\mu,\lambda)$ we call $\bfP(h)$ the subset of Dyck strips in $\bfP$ of height $h$. Notice that all the strips in $\bfP(h)$ are pairwise distant.

Since $\bfP$ is of type 1, if $\hgt(D_i)>\hgt(D_j)$ then there is no box in $D_j$ which is North, NW or NE of a box in $D_i$. It follows that any ordering $(D_{\sigma(1)},D_{\sigma(2)},\ldots,D_{\sigma(k)})$ for which $\hgt(D_{\sigma(i)})\leq \hgt(D_{\sigma(i+1)})$ for all $i$ is admissible.

\begin{definition}
	Let $\bfP\in \Conf^1(\lambda,\mu)$ and $o\in \adm(\bfP)$. We define $o_{\hgt}$ to be the order given by taking first all the strips of height $1$ (in the same order as they occur in $o$), then all the strips of height $2$
	and so on. We have $o_{\hgt}\in \adm(\bfP)$.
\end{definition}

\begin{lemma}\label{orderedheight}
Let $\bfP\in \Conf^1(\mu,\lambda)$ and let $o\in \adm(\bfP)$. We denote by $o(h)$ the restriction of $o$ to $\bfP(h)$. Then there exists $c\in \bbQ^*$ such that
\[f_{\bfP,o}=c\cdot f_{\bfP(m),o(m)}\circ f_{\bfP(m-1),o(m-1)}\circ \ldots \circ f_{\bfP(1),o(1)}= c\cdot f_{\bfP,o_{\hgt}}.\]
\end{lemma}
\begin{proof}
Let $o_{\hgt}=(D_1,\ldots,D_k)$ and $o=(D_{\sigma(1)},D_{\sigma(2)},\ldots, D_{\sigma(k)})$. Assume that $o\neq o_{\hgt}$, i.e. that $\sigma$ is not the trivial permutation.
 Then, there exists an index $j$ such that $\hgt(D_{\sigma(j)})>\hgt(D_{\sigma(j+1)})$. Since $o\in \adm(\bfP)$, this implies that the strips $D_{\sigma(j)}$ and $D_{\sigma(j+1)}$ can be added in any order to $\mu+D_{\sigma(1)}+\ldots+D_{\sigma(j-1)}$, so in particular they are distant. Hence the ordering \[o':=(D_{\sigma(1)},\ldots D_{\sigma(j-1)},D_{\sigma(j+1)},D_{\sigma(j)},D_{\sigma(j+2)},\ldots D_{\sigma(k)})\]
 is also admissible.

By induction on the length of the permutation $\sigma$ we can assume $f_{\bfP,o'}=c'f_{\bfP,o_{\hgt}}$ for some $c'\in \bbQ^*$. Moreover, since $\hgt(D_{\sigma(j)})\neq \hgt(D_\sigma(j+1))$, by \Cref{distantDp}.\ref{dist1}), we have $f_{\bfP,o}=c'' f_{\bfP,o'}$ for some $c''\in \bbQ^*$.
\end{proof}

\subsubsection{A partial order on Dyck partitions}
\begin{definition}\label{PartialOrder}
Let $\lambda,\mu$ be paths such that $\lambda\leq \mu$ and let $\bfP,\bfQ\in \Conf^1(\lambda,\mu)$. For $h\in \bbN$, we say that $\bfP(h)$ is \emph{finer} that $\bfQ(h)$ if $\bfP(h)\neq \bfQ(h)$ and for any strip in $\bfP(h)$ there exists a strip in $\bfQ(h)$ containing it.

 We write $\bfP\succ \bfQ$ if there exists an integer $h>0$ such that $\bfP(j)=\bfQ(j)$ for any $j>h$ and $\bfP(h)$ is finer than $\bfQ(h)$.
\end{definition}

\begin{lemma}
The relation $\succ$ defines a partial order on Dyck partitions. 
\end{lemma}
\begin{proof}
Assume $\bfP\succ \bfQ$ and $\bfQ\succ \bfP$. Then, for the largest index $h$ for which $\bfP(h)$ and $\bfQ(h)$ differ we would have that $\bfP(h)$ is finer than $\bfQ(h)$ and $\bfQ(h)$ is finer that $\bfP(h)$. This is a contradiction.

Assume $\bfP\succ \bfQ$ and $\bfQ\succ \bfR$. Let $h$ be such that $\bfP(j)=\bfQ(j)$ for all $j>h$ and $\bfP(h)$ is finer than $\bfQ(h)$. Let $k$ be such that $\bfQ(j)=\bfR(j)$ for all $j>k$ and $\bfQ(k)$ is finer than $\bfR(k)$. 

We have $\bfP(j)=\bfR(j)$ for any $j>\max\{h,k\}$.
If $k=h$ then $\bfP(h)$ is also finer than $\bfR(h)$. If $k>h$ then $\bfP(k)=\bfQ(k)$ is finer than $\bfR(k)$. If $k<h$ then $\bfP(h)$ is finer than $\bfQ(h)=\bfR(h)$. In any case we have $\bfP\succ \bfR$.
\end{proof}

\subsubsection{Structure of the proof}

For any two paths $\lambda,\mu$ with $\lambda \leq \mu$ and for any Dyck partition $\bfP\in \Conf^1(\lambda,\mu)$ we arbitrarily fix an admissible order $o_\bfP\in \adm(\bfP)$.

\begin{notation}
	In order to avoid a too cumbersome notation we write $f_\bfP$ instead of $f_{\bfP,o_\bfP}$ where $o_\bfP\in \adm(\bfP)$ is the order that we have just fixed.
\end{notation}

Our main goal is to show that the following.
\begin{theorem}\label{mainthm}
	Let $\lambda,\mu$ be paths with $\lambda\leq \mu$.
	The set $\{f_\bfP\}_{\bfP\in \Conf^1(\lambda,\mu)}$ is a basis of $\Hom_{\not< \lambda}(B_\lambda^I,B_\mu^I)$ as a left $R$-module.
\end{theorem}

The proof of \Cref{mainthm} will be achieved via a ``double induction'' on the following statements. (We list how these imply one another below \Cref{Yrmk} in
\eqref{Astate} and \eqref{Bstate}.)
Let $\lambda$ and $\mu$ denote paths with $\lambda\leq \mu$ and $\bfP\in \Conf^1(\lambda,\mu)$.

\begin{labeling}{$B(\lambda,\mu,D,\bfR):=$}
	\item[$A(\lambda,\mu,\bfP):=$] For any $o'\in \adm(\bfP)$ we have
 \[\displaystyle		f_{\bfP,o'}\in cf_\bfP+\spa \langle f_{\bfQ} \mid \bfQ\prec \bfP\rangle\subset \Hom_{\not< \lambda}(B_\lambda^I,B_\mu^I)\]
	for some $c\in \bbQ^*$.\footnote{Unless otherwise specified, by $\spa\langle J\rangle $ we denote the left $R$-module generated by the set $J$. }
	\item[$A(\lambda,\mu):=$]
 $A(\lambda,\mu,\bfP)$ holds for any $\bfP\in \Conf^1(\lambda,\mu)$.

\end{labeling}

Let now $D$ be a Dyck strip that can be removed from $\mu$ and $\bfR\in \Conf^1(\lambda,\mu-D)$.

\begin{labeling}{$B(\lambda,\mu,D,\bfR):=$}
	\item[$B(\lambda,\mu,D,\bfR):=$]
	For any $\bfP\in \Conf^1(\lambda,\mu)$ such that $D\in \bfP$ and $\bfP\setminus \{D\}\succeq \bfR$ and for any $o'\in \adm(\bfR)$ we have 
	\[f_D\circ f_{\bfR,o'}\in \spa \langle f_{\bfQ} \mid \bfQ\preceq \bfP \rangle\subset \Hom_{\not< \lambda}(B_\lambda^I,B_\mu^I).\]
	Moreover, if $\bfP\setminus \{D\}\succ \bfR$ then
	\[f_D\circ f_{\bfR,o'}\in \spa \langle f_{\bfQ} \mid \bfQ\prec \bfP \rangle\subset \Hom_{\not< \lambda}(B_\lambda^I,B_\mu^I).\]
	\item[$B(\lambda,\mu):=$]
	$B(\lambda,\mu,D,\bfR)$ holds for any $D$ that can be removed from $\mu$ and any $\bfR\in \Conf^1(\lambda,\mu-D)$.
\end{labeling}

Notice that $B(\lambda,\mu,D,\bfR)$ implies that if $\bfR\preceq \bfS\in \Conf^1(\lambda,\mu-D)$ and $\bfS\cup \{D\}$ is of type 1, then 
	\[f_D\circ f_{\bfR,o'}\in \spa \langle f_{\bfQ} \mid \bfQ\preceq \bfS\cup \{D\} \rangle.\]
Moreover, if $\bfR\prec \bfS$ we have
\[f_D\circ f_{\bfR,o'}\in \spa \langle f_{\bfQ} \mid \bfQ\prec \bfS\cup \{D\} \rangle.\]

In \Cref{isabasis} we show that \Cref{mainthm} follows from $A(\lambda,\mu')$ for all $\mu'\leq \mu$. Hence, it is enough to show the statement $A(\lambda,\mu)$ for every pair of weights $\lambda,\mu$.
We  define two auxiliary sets of Dyck partitions.

\begin{definition}
Let $\lambda,\mu$ be paths and $D$ be a Dyck strip that can be removed from $\mu$ and such that $\lambda\leq \mu-D$. For $\bfR\in \Conf^1(\lambda,\mu-D)$ we define the following sets:
\[X(\bfR,D):=\{ \bfP\in \Conf^1(\lambda,\mu)\;\mid\; D\in \bfP\text{ and }\bfP\setminus \{D\}\succ \bfR\}\]
\[Y(\bfR,D):=\{ \bfQ\in \Conf^1(\lambda,\mu)\;\mid\; \bfQ\prec \bfP \text{ for all }\bfP\in X(\bfR,D)\}.\]
\end{definition}
We write $\bfP\succ Y(\bfR,D)$ if $\bfP\succ \bfQ$ for any $\bfQ\in Y(\bfR,D)$. Notice that if $\bfP\in X(\bfR,D)$, then $\bfP\succ Y(\bfR,D)$. Moreover, if $\bfR\prec \bfS$ then $Y(\bfR,D)\cug Y(\bfS,D)$.

\begin{remark}\label{Yrmk}
 The statement $B(\lambda,\mu,D,\bfR)$ implies that
	\begin{equation}\label{YB}f_D\circ f_{\bfR,o'}\in \spa \langle f_{\bfQ} \mid \bfQ\in Y(\bfR,D) \rangle\subset \Hom_{\not< \lambda}(B_\lambda^I,B_\mu^I).\end{equation}
	Moreover, if $\bfR\cup\{D\}$ is not of type 1 then $B(\lambda,\mu,D,\bfR)$ is equivalent to $\eqref{YB}$.
\end{remark}

In \Cref{Aproof} we show that for any paths $\lambda,\mu$ with $\lambda\leq \mu$ and $\bfP\in \Conf^1(\lambda,\mu)$ we have\\
\begin{minipage}{0.65\textwidth}
	\begin{enumerate}[i)]
	\item $A(\lambda,\mu')$ for all $\mu'<\mu$
\item $B(\lambda,\mu')$ for all $\mu'< \mu$
\item $A(\lambda,\mu,\bfQ)$ for all $\bfQ\in \Conf^1(\lambda,\mu)$ such that $\bfQ\prec \bfP$
\item $B(\lambda,\mu,D,\bfR)$ for all Dyck strips $D$ with $\hgt(D)=\hgt(\lambda,\mu)$ that can be removed from $\mu$ and for all $\bfR\in \Conf^1(\lambda,\mu-D)$ such that $\bfP\succ Y(\bfR,D)$
\end{enumerate}
\end{minipage} \begin{minipage}{0.35\textwidth}
	\begin{equation}\label{Astate}
\begin{tikzpicture}
	\draw [decorate,decoration={brace,amplitude=10pt,mirror,raise=4pt},yshift=0pt]
	(0,-2) -- (0,2);
	\draw node at (2,0) {$\implies A(\lambda,\mu,\bfP)$};

\end{tikzpicture}
\end{equation}
\end{minipage}
and in \Cref{Bproof} we show that, if $D$ a Dyck strip that can be removed from $\mu$, for any   $\bfR\in\Conf^1(\lambda,\mu-D)$ we have\\
\begin{minipage}{0.65\textwidth}
\begin{enumerate}[i)]
	\item $A(\lambda,\mu')$ for all $\mu'<\mu$
	\item $B(\lambda,\mu')$ for all $\mu'<\mu$
	\item $A(\lambda,\mu,\bfQ)$ if $\bfQ\in Y(\bfR,D)$
	\item $B(\lambda,\mu,D,\bfQ)$ for all $\bfQ\in \Conf^1(\lambda,\mu-D)$ with $\bfQ\prec \bfR$
	\item $B(\lambda,\mu,D',\bfQ)$ for all Dyck strips $D'$ with $\hgt(D')>\hgt(D)$ and for all $\bfQ\in \Conf^1(\lambda,\mu-D')$
	\item $B(\lambda,\mu,D',\bfQ)$ for all Dyck strips $D'$ with $\hgt(D')=\hgt(D)$ and $\ell(D')>\ell(D)$ and for all $\bfQ\in \Conf^1(\lambda,\mu-D')$ such that $Y(\bfQ,D')\subseteq Y(\bfR,D)$
\end{enumerate}
\end{minipage}\begin{minipage}{0.35\textwidth}
\begin{equation}\label{Bstate}
	\begin{tikzpicture}
		\draw [decorate,decoration={brace,amplitude=10pt,mirror,raise=4pt},yshift=0pt]
		(0,-3) -- (0,3);
		\draw node at (2.5,0) {$\implies B(\lambda,\mu,D,\bfR).$};

	\end{tikzpicture}
\end{equation}
\end{minipage}

\begin{lemma}\label{howinductionworks}
	Assume that the inductive steps \eqref{Astate} and \eqref{Bstate} hold.  Then, $A(\lambda,\mu)$ and $B(\lambda,\mu)$ hold for all paths $\lambda,\mu$ with $\lambda\leq \mu$.
\end{lemma}
\begin{proof}
	It is easy to see from \eqref{Astate} that $B(\lambda,\mu')$ for all $\mu'\leq \mu$ implies $A(\lambda,\mu')$ for all $\mu'\leq \mu$.

	Fix $\lambda$ and assume $B(\lambda,\mu)$ does not hold for some $\mu$.
	\begin{itemize}
		\item 	 Choose $\mu_0$ minimal such that $B(\lambda,\mu_0)$ does not hold. 
		\item Choose a pair $(\bfR_0,D_0)$ such that $B(\lambda,\mu_0,D_0,\bfR_0)$ does not hold for which $\hgt(D_0)$ is maximal.
		\item 	Among the pairs $(\bfR,D)$ such that $B(\lambda,\mu_0,D,\bfR)$ does not hold and $\hgt(D)=\hgt(D_0)$ choose $(\bfR_1,D_1)$ such that the set $Y(\bfR_1,D_1)$ is minimal.
		\item 	Among those pairs $(\bfR,D)$ for which $B(\lambda,\mu_0,D,\bfR)$ does not hold, $\hgt(D)=\hgt(D_1)$ and $Y(\bfR,D)=Y(\bfR_1,D_1)$ choose $(\bfR_2,D_2)$ such that $D_2$ is of maximal length.
	\end{itemize}
	Since $Y(\bfR_2,D_2)\supseteq Y(\bfR_2',D_2)$ if $\bfR_2\succ \bfR_2'$, we can choose $\bfR_2$ to be minimal among the partitions $\bfR\in \Conf^1(\lambda,\mu_0-D_2)$ for which $B(\lambda,\mu_0,D_2,\bfR)$ does not hold.
	
	By \eqref{Bstate}
	there exists $\bfQ\in Y(\bfR_2,D_2)$ such that $A(\lambda,\mu_0,\bfQ)$ does not hold. 	Let $\bfP$ be a minimal element in $Y(\bfR_2,D_2)$ for which $A(\lambda,\mu_0,\bfP)$ does not hold. Then by \eqref{Astate} there must exist $D'$ and $\bfR'$ with $\hgt(D')=\hgt(\lambda,\mu_0)\geq \hgt(D_2)$ and $Y(\bfR',D')\prec \bfP$ such that $B(\lambda,\mu_0,D',\bfR')$ does not hold. We have $Y(\bfR',D')\prec \bfP\in Y(\bfR_2,D_2)$ and $\bfP\not \in Y(\bfR',D')$, hence $Y(\bfR',D')\subset Y(\bfR_2,D_2)$. This contradicts the minimality of $Y(\bfR_1,D_1)$. It follows that $B(\lambda,\mu)$ and $A(\lambda,\mu)$ hold for all $\mu$.
\end{proof}

\begin{proof}[Proof of \Cref{mainthm}.]
In the next section, we will prove \eqref{Astate} and \eqref{Bstate} in \Cref{Aproof,Bproof}. It follows that \Cref{howinductionworks} holds without assumptions. Thus, $A(\lambda,\mu)$ and $B(\lambda,\mu)$ hold for any $\lambda$ and $\mu$ with $\lambda\leq \mu$. Finally, we conclude the proof by showing in  \Cref{isabasis} that
$A(\lambda,\mu')$ for all $\mu'\leq \mu$ implies that the set $\{f_\bfP\}_{\bfP\in \Conf^1(\lambda,\mu)}$ is a basis of $\Hom_{\not< \lambda}(B_\lambda^I,B_\mu^I)$.
\end{proof}

We can fully describe the Hom spaces between indecomposable objects in $\sbim^I$. For $\bfP\in \Conf^1(\lambda,\mu)$, we denote by $g_\bfP\in \Hom(B_\mu^I,B_\lambda^I)$ the map obtained by taking the flip of $f_\bfP$.
Recall the definition of a strictly object-adapted cellular (or SOACC) category from \cite[Definition 2.4]{ELa} (this is a rigidification of the original definition of cellular category given in \cite[Definition 2.1]{Wes}).

\begin{cor}\label{maincoro}	
	For any two paths $\lambda,\mu$ the set $\{f_\bfP\circ g_\bfQ\}$ where $\bfP\in \Conf^1(\lambda,\nu)$, $\bfQ\in \Conf^1(\nu,\mu)$ and $\nu\leq \lambda,\mu$ is a basis of $\Hom(B_\lambda^I,B_\mu^I)$ as a left $R$-module.
	
	The category $\sbim^I$ is a strictly object-adapted cellular category with cellular basis $\{f_\bfP\circ g_\bfQ\}$.
\end{cor}

\begin{proof}	
		Recall that, if $M$ is a finitely generated graded free $R$-module, every set $S\subset M$ of homogeneous elements which generates $M$ over $R$ and such that 
	\[\grrk M=\sum_{s\in S}v^{\deg(s)}\]
	is a $R$-basis of $M$.	So, by Soergel's Hom formula (\Cref{SHF}) and \Cref{SZJ}, it is enough to show that the set $\{f_\bfP\circ g_\bfQ\}$ generates $\Hom(B_\lambda^I,B_\mu^I)$. It follows directly from \Cref{mainthm} that, for $\nu\leq \lambda,\mu$ the subset $\{f_\bfP\circ g_\bfQ\}$, with $\bfP\in \Conf^1(\lambda,\nu)$, $\bfQ\in \Conf^1(\lambda,\nu)$, generates $\Hom_{\not< \nu}(B_\lambda^I,B_\mu^I)$. It is easy to see by induction on $\nu$ that the subset $\{f_\bfP\circ g_\bfQ\}$, with $\bfP\in \Conf^1(\lambda,\tau)$, $\bfQ\in \Conf^1(\lambda,\tau)$ and $\tau\leq \nu$ generates the subspace $\Hom_{\leq \nu}(B_\lambda^I,B_\mu^I)$. The first claim follows.
	
	The fact that $\sbim^I$ is  a strictly object-adapted cellular category with cellular basis $\{f_\bfP\circ g_\bfQ\}$ is immediate from the definitions.
\end{proof}

\subsubsection{The proofs of the inductive steps}

In this section we show the crucial Lemmas that are necessary to conclude he proof of \Cref{mainthm}. More precisely, we show \eqref{Astate} in \Cref{Aproof} and \eqref{Bstate} in \Cref{Bproof}. Moreover, we show in the following Lemma that \Cref{mainthm} follows $A(\lambda,\mu)$ for all paths $\lambda,\mu$.

\begin{lemma}\label{isabasis}
	Let $\lambda,\mu$ be paths with $\lambda\leq \mu$ and assume $A(\lambda,\mu')$ for all $\mu'\leq \mu$. Then $\{f_\bfP\}_{\bfP\in \Conf^1(\lambda,\mu)}$ is a left $R$-basis of $\Hom_{\not< \lambda}(B_\lambda^I,B_\mu^I)$.
\end{lemma}
\begin{proof}
	The case $\mu=\lambda$ is clear.	By induction we can assume the statement for any $\mu'$ with $\lambda\leq \mu'<\mu$.
	First notice that $\Hom_{\not< \lambda}(B_\lambda^I,B_\mu^I)=0$ if $\mu\not\geq \lambda$.

	Consider the Rouquier complex $E_\mu^I$. 
	Let \[M^1(\mu)=\{\nu \in \Lambda_{n,i} \mid \exists\text{ Dyck strip }D_{\nu}\text{ with }\nu+D_{\nu}=\mu\},\]
	so that we have
	\[{}^{-1}E_\mu^I\cong \bigoplus_{\nu\in M^1(\mu)} B_{\nu}^I(-1).\]
	From \cite[Lemma 4.15]{Pat4} and the induction hypothesis we see that
	\begin{multline*}\Hom_{\not< \lambda}(B_\lambda^I,B_\mu^I)= \spa \left\langle f_{D_\nu}\circ g \mid \nu \in M^1(\mu),\; g\in \Hom_{\not< \lambda}(B_\lambda^I,B_\nu^I)\right\rangle=\\
= \spa \left\langle f_{D_\nu}\circ f_\bfP\mid \nu\in M^1(\mu),\; \bfP\in \Conf^1(\lambda,\nu)\right\rangle. 
	\end{multline*}
	
	Let $V:=\spa \left\langle f_\bfQ \mid \bfQ \in \Conf^1(\lambda,\mu)\right\rangle \cug \Hom_{\not< \lambda}(B_\lambda^I,B_\mu^I) $. We want to show that
	$\Hom_{\not< \lambda}(B_\lambda^I,B_\mu^I)= V$.
	This will follow from the next claim.
	
	\begin{claim}\label{Claimchesembravafacile}
		Let $D$ be a Dyck strip that can be removed from $\mu$ and let $\bfP\in \Conf^1(\lambda,\mu-D)$. Then 
		$f_D\circ f_\bfP\in V$.
	\end{claim}
	\begin{proof}[Proof of the claim.]
		By induction, we can assume that the claim holds for all the pairs $(D',\bfQ)$ satisfying one of the following conditions:
		\begin{itemize}
			\item $D'$ is any Dyck strip that can be removed from $\mu$ with $\hgt(D')>\hgt(D)$, or $\hgt(D')=\hgt(D)$ and $\ell(D')>\ell(D)$, and $\bfQ\in \Conf^1(\lambda,\mu-D')$
			\item $D'=D$ and $\bfQ\in \Conf^1(\lambda,\mu-D)$ with $\bfQ\prec \bfP$.
		\end{itemize}
	
	If $\{D\}\cup \bfP$ is of type 1, then by $A(\lambda,\mu,\bfP)$ we have	$ f_D\circ f_\bfP \in \spa\langle f_{\bfS} \mid \bfS\preceq \{D\}\cup \bfP\rangle\subset V$.
	Hence, we can assume $\{D\}\cup \bfP$ not of type 1. This means there exists a strip $C\in \bfP$ which contains a box South of a box in $D$ and such that $\hgt(C)\geq \hgt(D)-1$.
	Let $h:=\hgt(D)$. We divide the rest of the proof into three cases.
	
	\begin{itemize}
	\item	
	\textbf{There exists a strip $D'\in \bfP$ with $D'\neq C$ that can be removed from $\mu$ such that $\hgt(D')>h$.}
	
	 Then there exists an admissible order $\tilde{o}$ ending in the strip $D'$. Notice that $D$ and $D'$ are distant, hence by $A(\lambda,\mu-D)$ we have 
	\[f_{\bfP,o}-cf_{\bfP,\tilde{o}}\in \spa\langle f_\bfQ \mid \bfQ\prec \bfP \rangle\]
	for some scalar $c\in \bbQ$. By induction, we can assume $o=\tilde{o}$.
	By \Cref{distantDp}.\ref{dist1}), we have	$f_{D}\circ f_{D'}=c f_{D'}\circ f_{D}$ for some scalar $c\in \bbQ$. Since $\hgt(D')>\hgt(D)$ we conclude by induction. 	
	
\item \textbf{$\hgt(C)\geq h$ and there does not exists a strip $C\neq D'\in \bfP$ that can be removed from $\mu$ with $\hgt(D')>h$.} 

	Since there is no Dyck strip $D'$ as above, we can remove $C$ from $\mu-D$ so there exists an admissible order $o'\in \adm(\bfP)$ ending in $C$. As before, we can assume by induction, that $o=o'$.
	By \Cref{overlying} there exist Dyck strips $C',D'$ with $C'\sqcup D'=C\sqcup D$, such that $\{C',D'\}$ is a partition of type 1 and that $f_D\circ f_C=f_{D'}\circ f_{C'}$ (see also \Cref{fromIItoI}).
	
	We have $f_{C'}\circ f_{\bfP\setminus\{C\},o'}\in \spa \langle f_{\bfR}\mid \bfR\in \Conf^1(\lambda,\mu-D')\rangle$.
	Since $\hgt(D')\geq \hgt(D)$ and $\ell(D')>\ell(D)$ from our induction hypothesis it follows that
	\[ f_D\circ f_C \circ f_{\bfP\setminus \{C\},o'}= f_{D'}\circ f_{\bfR}\in V.\]

\item \textbf{$\hgt(C)=h-1$ and there does not exists a strip $C\neq D'\in \bfP$ with $\hgt(D')>h$.} 
 By the assumption above there are no strips in $\bfP$ of height $>h$. There can be at most two strips $R_1,R_2\in \bfP$ with $\hgt(R_1)=\hgt(R_2)=\hgt(D)$ touching $C$, as in \Cref{touchingC}.
	\begin{figure}[!ht]
	\begin{center}
	\begin{tikzpicture}[x=\boxmini,y=\boxmini]
\tikzset{vertex/.style={}}
\tikzset{edge/.style={very thick}}
\tabpath{+,-,+,-,+,-,+,-,-,+,+,-,+,-}
\tabpath{-,-,+,-,-,-,+,+,-,+,+,-,+,+}
\tabpath{-,-,+,+,-,-,+,+,-,+,-}
\node at (6,-1) {$D$};
\node at (9,-2) {$C$};
\node at (2,-1) {$R_1$};
\node at (12,-1) {$R_2$};		
\end{tikzpicture}
	\end{center}
	\caption{In this example we have $\hgt(C)+1=\hgt(D)=\hgt(R_1)=\hgt(R_2)$. The strips $R_1$ and $R_2$ are of maximal height in $\bfP$.}\label{touchingC}
\end{figure}

	There exists an admissible order $o''\in \adm(\bfP)$ whose ending is $\ldots,C,R_2,R_1)$. By induction, it is enough to consider the case $o=o''$. As in \Cref{gybstrips}, we can find Dyck strips
	$S_2, S_T\subset C$ such that  both $T_1 = R_1 \sqcup D \sqcup T_2$ and $S_1 = R_2 \sqcup D \sqcup S_2$ are Dyck strips. 
	
\begin{figure}[!ht]	
	\begin{center}
		\begin{tikzpicture}[x=\boxmini,y=\boxmini]
		\draw[fill=yellow] (0,0) -- (1,1) -- (2,0) -- (3,1) -- (4,0) -- (5,1) -- (6,0) -- (7,1) -- (8,0) -- (6,-2) -- (5,-1) -- (4,-2) -- (3,-1) -- (2,-2) -- cycle;
		\draw[fill=blue!50] (4,0) -- (5,1) -- (6,0) -- (7,1) -- (8,0) -- (9,-1) -- (11,1) -- (12,0) -- (13,1) -- (14,0) --(12,-2) -- (11,-1) --(9,-3) -- (7,-1) -- (6,-2) -- cycle;
		\draw[fill=green](4,0) -- (5,1) -- (6,0) -- (7,1) -- (8,0) -- (6,-2) -- cycle;
		\tikzset{vertex/.style={}}
		\tikzset{edge/.style={very thick}}
		\tabpath{+,-,+,-,+,-,+,-,-,+,+,-,+,-}
		\tabpath{-,-,+,-,-,-,+,+,-,+,+,-,+,+}
		\tabpath{-,-,+,-,+,-,+,+,-,+,-}
		\tabpath{-,-,+,+,-,-,+,-}
		\node at (4,-1) {$T_2$};
		\node at (9,-2) {$S_2$};	
		\node[black!30!yellow] at (4,1) {$T_1$};	
		\node[blue] at (9,1) {$S_1$};	
		\node at (6,-1) {$D$};
		\node at (2,-1) {$R_1$};
		\node at (12,-1) {$R_2$};	
		\end{tikzpicture}
		\end{center}
	\caption{In this example $T_1=R_1\sqcup D\sqcup T_2$ is depicted in yellow and green while $S_1=R_2\sqcup D\sqcup S_2$ is depicted in green and blue. Moreover, the union of the colored strips forms a Dyck strip $U$.}
	\label{gybstrips}
\end{figure}

	 Applying \Cref{distantDp}.\ref{dist2}) twice, we see that there exist $c_1,c_2,c_3 \in \bbQ$ such that 
	\begin{equation}\label{longcommute}f_D\circ f_{R_1}\circ f_{R_2}=c_1 f_{R_1} \circ f_{R_2} \circ f_D + c_2 f_{T_1} \circ g_{T_2}\circ f_{R_2} + c_3 f_{R_1}\circ f_{S_1} \circ g_{S_2}. \end{equation}
	
	Let $\bfR=\bfP\setminus\{ R_1,R_2\}$ and $o''_\bfR$ be the restriction of $o''$ to $\bfR$. Notice that $o''_{\bfR}$ ends in $C$. We have $f_D \circ f_C= 0$, so when precomposing with $f_{\bfR,o''}$ the first term in the RHS of \eqref{longcommute} vanishes and the second term is in $V$ by induction since $\ell(T_1)>\ell(D)$. 
	 
	It remains to consider the third term in \eqref{longcommute}. Notice that there exists a Dyck strips $U$ with $U= R_1 \sqcup S_1 \sqcup T_2$ and 
	\[f_{R_1} \circ f_{S_1} = d_1f_{S_1} \circ f_{R_1} + d_2 f_{U} \circ g_{T_2}\]
	for some $d_1,d_2\in \bbQ$. 
	Since $\ell(U)>\ell(S_1)>\ell(D)$ and $\hgt(U)=\hgt(S_1)=\hgt(D)$, by induction we obtain also $ f_{R_1} \circ f_{S_1}\circ g_{S_2}\circ f_{\bfR,o''_\bfR}\in V$. \qedhere
	\end{itemize}
	\end{proof}
	
	We go back to the proof of \Cref{isabasis}.
	 From \Cref{Claimchesembravafacile} it follows that $\Hom_{\not< \lambda}(B_\lambda^I,B_\mu^I)$ is generated by the set
	$\{f_{\bfP}\}_{\bfP\in \Conf^1(\lambda,\mu)}$. 
	 Then since $\grrk \Hom_{\not< \lambda}(B_\lambda^I,B_\mu^I)=h_{\mu,\lambda}^I(v)$, from \Cref{SZJ} we see that the generating set $\{f_{\bfP}\}_{\bfP\in \Conf^1(\lambda,\mu)}$ has the right graded size, thus it must be a basis.
\end{proof}

We are now ready to prove the first of the two induction statements. Recall that $\hgt(\lambda,\mu)=\max\{\hgt(D) \mid D\in \bfP$ for $\bfP\in \Conf^1(\lambda,\mu)\}$.

\begin{lemma}\label{Aproof}
Let $\lambda,\mu$ be paths with $\lambda\leq \mu$ and let $\bfP\in \Conf^1(\lambda,\mu)$. Assume
\begin{enumerate}[i)]
	\item $A(\lambda,\mu')$ for all $\mu'<\mu$;
	\item $B(\lambda,\mu')$ for all $\mu'< \mu$;
	\item $A(\lambda,\mu,\bfQ)$ for all $\bfQ\in \Conf^1(\lambda,\mu)$ such that $\bfQ\prec \bfP$;
	\item \label{hypiv}$B(\lambda,\mu,D,\bfR)$ for all Dyck strips $D$ with $\hgt(D)=\hgt(\lambda,\mu)$ that can be removed from $\mu$ and for all $\bfR\in \Conf^1(\lambda,\mu-D)$ such that $\bfP\succ Y(\bfR,D)$.
\end{enumerate}
Then $A(\lambda,\mu,\bfP)$ holds.
\end{lemma}
\begin{proof}
	Let $o_\bfP=(D_1,\ldots,D_k)$ and $o'=(D_{\sigma(1)},\ldots,D_{\sigma(k)})\in \adm(\bfP)$. We need to show that for some $c \in \bbQ^*$ we have
	\begin{equation}\label{Aeq}
	f_{\bfP,o'}-c\cdot f_{\bfP}\in \spa \langle f_{\bfQ} \mid \bfQ\prec \bfP\rangle\subset \Hom_{\not< \lambda}(B_\lambda^I,B_\mu^I).
	\end{equation}

	Assume that $\sigma$ is not the trivial permutation. Then there exists $j$ such that $\sigma(j)>\sigma(j+1)$. This implies that the strips $D_{\sigma(j)}$ and $D_{\sigma(j+1)}$ can be added in any order, in particular they are distant. Hence $o'':=(D_{\sigma(1)},\ldots D_{\sigma(j-1)},D_{\sigma(j+1)},D_{\sigma(j)},D_{\sigma(j+2)},\ldots D_{\sigma(k)})\in \adm(\bfP)$.
	
	By induction on the length of the permutation $\sigma$ we can assume that \[f_{\bfP,o''}=d\cdot f_{\bfP}+ \spa \langle f_{\bfQ} \mid \bfQ\prec \bfP\rangle\subset \Hom_{\not< \lambda}(B_\lambda^I,B_\mu^I),\]
	for some $d\in \bbQ^*$.
	Now, up to replacing $o_\bfQ$ with $o''$, it is enough to prove the statement \eqref{Aeq} when $\sigma$ is the simple transposition $(j\; j+1)$, i.e we can assume 
	\[o'=(D_1,\ldots, D_{j-1},D_{j+1},D_j,D_{j+2},\ldots, D_k).\]

	If $D_j\sqcup D_{j+1}$ cannot be obtained as the difference of two Dyck strips then, from \Cref{distantDp}.\ref{dist1}), we have $f_{\bfP,o'}=d'f_{\bfP}$ for some $d'\in \bbQ^*$.
	So we can assume in the following that $h:=\hgt(D_j)=\hgt(D_{j+1})$ and that there exists Dyck strips $T$ and $T'$ such that $D_j\sqcup D_{j+1}\sqcup T'=T$.
	Moreover, due to \Cref{orderedheight}, we can also replace $o$ by $o_{\hgt}$ and $o'$ by $o'_{\hgt}$. Notice that $o_{\hgt}$ and $o'_{\hgt}$ still differ by a single transposition.

	Let $\bfP^{<j}=\{D_1,\ldots, D_{j-1}\}$, $o_\bfP^{<j}=(D_1,\ldots, D_{j-1})$. Define similarly $\bfP^{>j+1}$ and $o_\bfP^{>j+1}$. Since $o=o_{\hgt}$, all the strips in $\bfP^{<j}$ have height $\leq h$. Let $\bfP^{<j}(h)$ be the subset of $\bfP^{<j}$ of strips of height $h$ and $\bfP^{<j}(<h)$ the subset of strips of height $<h$.
	By \Cref{distantDp}.\ref{dist1}), there exist $c\in \bbQ^*$ and $c'\in \bbQ$ such that
	\begin{multline*}
	f_{\bfP,o'}-c\cdot f_{\bfP}=c'\cdot f_{\bfP^{>j+1},o_\bfP^{>j+1}}\circ f_{T}\circ g_{T'} \circ f_{\bfP^{<j},o_\bfP^{<j}}=\\=c'\cdot f_{\bfP^{>j+1},o_\bfP^{>j+1}}\circ f_{T}\circ g_{T'} \circ f_{\bfP^{<j}(h),o_\bfP^{<j}(h)}\circ f_{\bfP^{<j}(<h),o_\bfP^{<j}(<h)}.
	\end{multline*}
	Notice that $T'$ is distant from all the strips in $\bfP^{<j}(h)$, hence by \Cref{distneg} there exists $c''\in \bbQ^*$ such that 
	\[ g_{T'} \circ f_{\bfP^{<j}(h),o_\bfP^{<j}(h)}= c''\cdot f_{\bfP^{<j}(h),o_\bfP^{<j}(h)}\circ g_{T'}.\]
	Let $\nu=\lambda+\sum_{C\in \bfP^{<j}(<h)} C-T'$. Since $\nu< \mu$, we can use the hypothesis $A(\lambda,\kappa)$ for all $\kappa\leq \nu$ and \Cref{isabasis} to write
	\[g_{T'} \circ f_{\bfP^{<j}(<h),o^{<j}(<h)}\in \spa \langle f_\bfR \mid \bfR \in \Conf^1(\lambda,\nu)\rangle \subset \Hom_{\not< \lambda}(B_\lambda^I,B_\nu^I)\]
	and
\[f_{\bfP,o'}-c\cdot f_{\bfP}\in \spa\langle f_{\bfP^{>j+1},o_{\bfP}^{>j+1}} \circ f_T\circ f_{\bfP^{<j}(h),o^{<j}(h)}\circ f_\bfR \mid \bfR \in \Conf^1(\lambda,\nu)\rangle.\]
	
We need to ``bound'' each term of the form $f_{\bfP^{>j+1},o_{\bfP}^{>j+1}} \circ f_T\circ f_{\bfP^{<j}(h),o^{<j}(h)}\circ f_\bfR $ for $\bfR\in \Conf^1(\lambda,\nu)$, i.e. we need to show that 
\[f_{\bfP^{>j+1},o_{\bfP}^{>j+1}} \circ f_T\circ f_{\bfP^{<j}(h),o^{<j}(h)}\circ f_\bfR \in \spa\langle f_\bfS\mid \bfS\prec \bfP\rangle\]
	Let $\bfU$ be the maximal Dyck partition in $\Conf^1(\lambda,\nu)$, i.e. the partition consisting only of single boxes. Let $\nu'=\nu+\sum_{C\in \bfP^{<j}(h)}C+T$. Clearly $\{T\}\cup \bfP^{<j}(h) \cup \bfU\in \Conf^1(\lambda,\nu')$ and
	\[ \{T\}\cup \bfP^{<j}(h) \cup \bfU \prec \{D_j,D_{j+1}\}\cup \bfP^{<j}.\]
We divide the rest of the proof into two cases.
\begin{itemize}
	\item $\bfP^{>j}=\emptyset$.

We have $\nu'=\mu$ and $\bfP=\{D_j,D_{j+1}\}\cup \bfP^{<j}$. Furthermore, in this case we have $h=\hgt(T)=\hgt(\lambda,\mu)$. Then $\bfP\succ \{T\}\cup \bfP^{<j}(h) \cup \bfU$, so we can apply the hypothesis $A(\lambda,\mu,\{T\}\cup \bfP^{<j}(h) \cup \bfU)$ to deduce 
\begin{equation}\label{TU} f_T\circ f_{\bfP^{<j}(h),o^{<j}(h)}\circ f_\bfU \in \spa\langle f_\bfS\mid \bfS\prec \bfP\rangle.\end{equation}

If $\bfR\prec \bfU$, we have by $B(\lambda,\mu')$ for $\mu'\leq \mu-T$ that 
\begin{equation}\label{PR}f_{\bfP^{<j}(h),o^{<j}(h)}\circ f_\bfR\in \spa\langle f_{\bfQ}\mid \bfQ \prec \bfP^{<j}(h)\cup \bfU\rangle.
\end{equation}
If $\bfQ \prec \bfP^{<j}(h)\cup \bfU$, since $\{T\}\cup \bfP^{<j}(h)\cup \bfU \in X(\bfQ,T)$, we also have $\bfP\succ Y(\bfQ,T)$.
Because of the hypothesis (\ref{hypiv}), we can now apply $B(\lambda,\mu,T,\bfQ)$ and obtain
\begin{equation}\label{TQ}f_T\circ f_{\bfQ}\in \spa\langle f_\bfS\mid \bfS\prec \{T\}\cup \bfP^{<j}(h) \cup \bfU\rangle\cug \spa\langle f_\bfS\mid \bfS\prec \bfP\rangle.
	\end{equation}
Putting together \eqref{TU},\eqref{PR} and \eqref{TQ}, we obtain as desired that
\[f_T\circ f_{\bfP^{<j}(h),o^{<j}(h)}\circ f_\bfR \in  \spa\langle f_\bfS\mid \bfS\prec \bfP\rangle.\]

\item  $\bfP^{>j}\neq \emptyset$.
	
In this case we have $\nu'<\mu$. Applying repeatedly the hypothesis $B(\lambda,\mu')$ for some $\mu'<\mu$, we get 
\begin{multline*}
f_{\bfP^{>j+1},o_{\bfP}^{>j+1}} \circ f_T\circ f_{\bfP^{<j}(h),o^{<j}(h)}\circ f_\bfR \in \spa\langle f_{\bfP^{>j+1},o_{\bfP}^{>j+1}} \circ f_\bfV\mid \bfV\preceq \{T\}\cup \bfP^{<j}(h) \cup \bfU \rangle \\
\cug \spa\langle f_{\bfP^{>j+1},o_{\bfP}^{>j+1}} \circ f_\bfV\mid \bfV\prec \{D_j,D_{j+1}\}\cup \bfP^{<j} \rangle.
\end{multline*}

Let $D_k$ be the last strip in $\bfP$. Since $o=o_{\hgt}$ we have $\hgt(D_k)=\hgt(\lambda,\mu)$. Applying again repeatedly the hypothesis $B(\lambda,\mu')$ for $\mu'<\mu$ we get 
	\[f_{\bfP^{>j+1},o_{\bfP}^{>j+1}} \circ f_T\circ f_{\bfP^{<j}(h),o^{<j}(h)}\circ f_\bfR\in \spa\langle f_{D_k}\circ f_\bfV \mid \bfV\prec \bfP\setminus \{D_k\}\rangle.\]
	
Since $\bfV\prec \bfP\setminus \{D_k\}$, we have $\bfP\in X(\bfV,D_k)$ and $Y(\bfV,D_k)\prec \bfP$.
The claim finally follows from $B(\lambda,\mu,D_k,\bfV)$.\qedhere
\end{itemize}	
\end{proof}

\begin{lemma}\label{Bproof}
	Let $\lambda,\mu$ be paths with $\lambda\leq \mu$. Let $D$ a Dyck strip that can be removed from $\mu$ and let $\nu=\mu-D$. Let $\bfR\in\Conf^1(\lambda,\nu)$. Assume:	
	\begin{enumerate}[i)]
		\item $A(\lambda,\mu')$ for all $\mu'<\mu$;
		\item $B(\lambda,\mu')$ for all $\mu'<\mu$;
		\item\label{hyp3} $A(\lambda,\mu,\bfQ)$ if $\bfQ\in Y(\bfR,D)$;
		\item\label{hyp4} $B(\lambda,\mu,D,\bfQ)$ for all $\bfQ\in \Conf^1(\lambda,\nu)$ with $\bfQ\prec \bfR$;
		\item\label{hyp5} $B(\lambda,\mu,D',\bfQ)$ for all Dyck strips $D'$ with $\hgt(D')>\hgt(D)$ and for all $\bfQ\in \Conf^1(\lambda,\mu-D')$;
			\item\label{hyp6} $B(\lambda,\mu,D',\bfQ)$ for all Dyck strips $D'$ with $\hgt(D')=\hgt(D)$ and $\ell(D')>\ell(D)$ and for all $\bfQ\in \Conf^1(\lambda,\mu-D')$ such that $Y(\bfQ,D')\subseteq Y(\bfR,D)$.
	\end{enumerate}
	Then $B(\lambda,\mu,D,\bfR)$ holds.
\end{lemma}
\begin{proof}
		If $\bfR\cup \{D\}$ is of type 1, then $\bfR\cup \{D\}\in Y(\bfR,D)$. Hence the claim follows from $A(\lambda,\mu,\bfR\cup\{D\})$.
	We assume now that $\bfR\cup \{D\}$ is not of Type 1. 
	Let 
	\[V:=\spa \langle f_\bfQ \mid \bfQ \in Y(\bfR,D)\rangle \subset \Hom_{\not\leq \lambda} (B_\lambda^I,B_\mu^I).\] We need to show that $f_D\circ f_{\bfR,o'}\in V$ for all $o'\in \adm(\bfR)$. However, if $o',o''\in \adm(\bfR)$, by $A(\lambda,\nu)$ we have 
	$f_{\bfR,o'}-cf_{\bfR,o''}\in \spa \langle f_{\bfQ} \mid \bfQ \prec \bfR\rangle$ for some scalar $c\in \bbQ^*$. From the hypothesis \ref{hyp4}) it is enough to show that $f_D\circ f_{\bfR,o'}\in V$ for some $o'\in \adm(\bfR)$.

	Let $h:=\hgt(D)$. Since $\bfR\cup\{D\}$ is not of type 1, there exists a strip $C\in \bfR$ which contains a box South of a box in $D$ and such that $\hgt(C)\geq h-1$.

	The strategy is similar to the proof of \Cref{Claimchesembravafacile}.  We divide the proof into three cases.
	
	\begin{itemize}
		\item\textbf{Assume there exists a strip $D'\in \bfR$ with $D'\neq C$ that can be removed from $\nu$ such that $\hgt(D')>\max\{\hgt(C),\hgt(D)\}$.}
		 
		Then $D'$ and $D$ are distant. We can assume that $D'$ has maximal height among all the strips with this property, so $h':=\hgt(D')=\hgt(\lambda,\mu)$.
	
	\begin{claim}\label{Y}
For any $\bfS\in Y(\bfR\setminus\{D'\},D)$ we have
\begin{equation*} Y(\bfS,D') \subseteq Y(\bfR,D).
\end{equation*}
	\end{claim}
	\begin{proof}[Proof of the claim.]		
		Let $\bfP\in X(\bfR,D)$, i.e. $D\in \bfP$ and $\bfP\setminus\{D\} \succ \bfR$. 		
		We need to show that for any $\bfS \in Y(\bfR\setminus\{D'\},D)$ we have $Y(\bfS,D')\prec \bfP$.
		
		Assume that $D'\not\in \bfP$. Then, $\bfP(h')$ is finer than $\bfR(h')$. In particular, $D'$ is a union of strips in $\bfP$. We can construct a new partition $\bfQ$ containing $D,D'$, all the strips in $\bfP(h')$ not contained in $D'$ and such that all its other strips are single boxes.
		We have $\bfP\succ \bfQ$ and $\bfQ\in X(\bfR,D)$. So up to replacing $\bfP$ with $\bfQ$ we can assume $D'\in \bfP$.
		
		If $D'\in \bfP$, then $\bfP\setminus \{ D,D'\} \succ \bfR\setminus \{D'\}$, so $\bfP\setminus \{D'\}\in X (\bfR\setminus\{D'\},D)$ and $\bfS\prec \bfP\setminus \{D'\}$. Then also $\bfP \in X(\bfS,D')$ and $\bfP\succ Y(\bfS,D')$.	
	\end{proof}
	
	There exist $o'\in \adm(\bfR)$ ending in $D'$ and a scalar $c\in \bbQ$ for which
	 \[f_{D}\circ f_{D'}\circ f_{\bfR\setminus\{D'\},o'}=c f_{D'}\circ f_{D}\circ f_{\bfR\setminus\{D'\},o'}\in \spa\langle f_{D'}\circ f_{\bfS} \mid \bfS\in Y(\bfR\setminus\{D'\},D)\rangle.   \]
	 Since $\hgt(D')>\hgt(D)$ we can apply $B(\lambda,\mu,\bfS,D')$ by  the hypothesis (\ref{hyp5}) and conclude by \Cref{Y} and \Cref{Yrmk}.
	 
	 \item \textbf{$\hgt(C)= h-1$ and there does not exist  $D'\in \bfR$ with $\hgt(D')>\max\{\hgt(C),\hgt(D)\}$.}
	 
	 In this case there are no strips of height $>h$ in $\bfR$ and there are at most two strips $R_1,R_2\in \bfR$ with $\hgt(R_1)=\hgt(R_2)=h$ touching $C$ as in \Cref{touchingC}. We can find an admissible order $o'\in \adm(\bfR)$ whose last three elements are $C,R_2,R_1$.

 As in the proof of \Cref{Claimchesembravafacile} we have
\begin{equation}\label{trecommut}f_D\circ f_{R_1}\circ f_{R_2} =c_1 f_{R_1} \circ f_{R_2} \circ f_D + c_2 f_{T_1} \circ f_{R_2}\circ g_{T_2} + d_1f_{S_1} \circ f_{R_1}\circ g_{S_2} + d_2 f_U \circ g_{T_2} \circ g_{S_2}.
\end{equation}
for some scalars $c_1,c_2,d_1,d_2\in \bbQ$ (the Dyck strips $S_1,S_2,$ etc., have the same meaning as in \Cref{gybstrips}).

Let $o'_{\bfR\setminus\{R_1,R_2\}}$ be the restriction of $o'$ to $\bfR\setminus\{R_1,R_2\}$.
	We need to compute 
	\[f_D\circ f_{\bfR,o'}=f_{D}\circ f_{R_1}\circ f_{R_2}\circ f_{\bfR\setminus\{R_1,R_2\},o'_{\bfR\setminus\{R_1,R_2\}}}.\]
	Since $f_D\circ f_C=0$ after precomposing with $f_{\bfR\setminus\{R_1,R_2\},o'_{\bfR\setminus\{R_1,R_2\}}}$ the first term in the RHS of \eqref{trecommut} vanishes.
	
	Consider the term $f_{T_1}\circ f_{R_2}\circ g_{T_2}\circ f_{\bfR\setminus\{R_1,R_2\},o'_{\bfR\setminus\{R_1,R_2\}}}$. Let
	$\bfZ$ be the Dyck partition in $\Conf^1(\lambda,\nu- R_1-T_2)$ which contains all the Dyck strips in $\bfR\setminus\{R_1\}$ of height $h$ and such that all the other Dyck strips in $\bfZ$ are single boxes. In particular, $R_2\in \bfZ$ and we have $\bfR\setminus\bfR(h)\prec \bfZ \setminus \bfZ(h)$. Since $g_{T_2}$ commutes with all the morphisms attached to the strips in $\bfZ(h)$, applying repeatedly  $B(\lambda,\mu')$ for $\mu'\leq \mu-T$, we have
	\begin{equation}\label{R2T2}f_{R_2}\circ g_{T_2}\circ f_{\bfR\setminus\{R_1,R_2\},\tilde{o}_{\bfR\setminus\{R_1,R_2\}}}\in \spa\langle f_{\bfS} \mid \bfS\preceq \bfZ\rangle.\end{equation}

	\begin{claim}\label{claimpart2}
		We have $\bfZ\cup \{T_1\}\in Y(\bfR,D)$.
	\end{claim}
	\begin{proof}[Proof of the claim.]
		Let $\bfP\in X(\bfR,D)$. We have $\bfP(j)=\emptyset$ if $j>h$. Moreover, 
		$\bfP(h)$ is finer or equal than $\bfR(h)\cup \{D\}$ which in turn is finer than $\bfZ(h)\cup \{T_1\}$.
	\end{proof}

	Notice that $\bfZ\cup\{T_1\}$ is of type 1. By hypothesis (\ref{hyp3}) we have $f_{T_1}\circ f_\bfZ\in V$. If $\bfS \prec \bfZ$, then $Y(\bfS,T_1)\prec \bfZ\cup \{T_1\}$, hence $Y(\bfS,T_1)\cug Y(\bfR,D)$ and by the hypothesis (\ref{hyp6}) we deduce that
	$f_{T_1}\circ f_{\bfS} \in V$. Combining this with \eqref{R2T2} we conclude that
	\[f_{T_1}\circ f_{R_2}\circ g_{T_2}\circ f_{\bfR\setminus\{R_1,R_2\},o'_{\bfR\setminus\{R_1,R_2\}}}\in V.\]
	 The proof for the remaining terms in the RHS of \eqref{trecommut} is similar. Finally, we obtain as desired \[f_D\circ f_{R_1}\circ f_{R_2}\circ f_{\bfR\setminus\{R_1,R_2\},o'_{\bfR\setminus\{R_1,R_2\}}}\in V.\]

\item \textbf{$\hgt(C)\geq h$ and there does not exist  $D'\in \bfR$ that can be removed from $\nu$ with $\hgt(D')>\max\{\hgt(C),\hgt(D)\}$.}
 In this case, there exists an admissible order $o'\in \adm(\bfR)$ ending in $C$. 

Let $o'_{\bfR\setminus \{C\}}$ be the restriction of $o'$ to $\bfR\setminus\{C\}$.
	As in \Cref{fromIItoI}, there exists $\{C',D'\}$ of Type 1 such that $D'\sqcup C'=D\sqcup C$. By \Cref{overlying} we have 
	\[f_D\circ f_{C}\circ f_{\bfR\setminus\{C\},o'_{\bfR\setminus\{C\}}}=f_{D'}\circ f_{C'}\circ f_{\bfR\setminus\{C\},o'_{\bfR\setminus\{C\}}}\]
	
	Consider now the Dyck partition $\tilde{\bfZ}\in \Conf^1(\lambda,\nu-C)$ obtained by taking all the Dyck strips in $\bfR\setminus \{C\}$ of height $\geq \hgt(C)$, and by replacing every strip in $\bfR\setminus \{C\}$ of height $<\hgt(C)$ with the single boxes of which it consists. We have 	
	$\tilde{\bfZ}\succeq \bfR\setminus\{C\}$ and, moreover, $\tilde{\bfZ} \cup\{D',C'\}\in \Conf^1(\lambda,\mu)$ (see \Cref{finalpict}). 
	Applying $B(\lambda,\mu-D')$ we obtain 
	\begin{equation}\label{DCDC}f_D\circ f_{\bfR,o'}=f_{D'}\circ f_{C'}\circ f_{\bfR\setminus\{C\},o'_{\bfR\setminus\{C\}}}\in \spa \langle f_{D'}\circ f_\bfS \mid \bfS \preceq \tilde{\bfZ} \cup\{C'\} \rangle.
	\end{equation}	

	\begin{claim}\label{claimpart}
	We have $\tilde{\bfZ}\cup \{C',D'\}\in Y(\bfR,D)$.
\end{claim}

	\begin{proof}[Proof of the claim.]
	Let $\bfP\in X(\bfR,D)$ and let $\bfQ=\bfP \setminus\{D\}$. We need to show that $\tilde{\bfZ}\cup \{D',C'\}\prec \bfP$. Notice that there cannot be any strip in $\bfQ$ containing $C$, as this would contradict the hypotheses that $D\in \bfP$ and that $\bfP$ is of Type 1. In particular this means that $\bfQ(\hgt(C))\neq \bfR(\hgt(C))$.	Since $\bfQ\succ \bfR$, the maximal index $j$ such that $\bfQ(j)\neq\bfR(j)$ must be $j\geq \hgt(C)$. 
	
	If $j>\hgt(C)$ it is easy to see that 
	$(\tilde{\bfZ}\cup \{D',C'\})(j)=\tilde{\bfZ}(j)=\bfR(j)$ and also
	$\bfQ(j)=\bfP(j)$. Since $\bfQ(j)$ is finer than $\bfR(j)$ the claim follows.
	
	If $j=\hgt(C)$ then $(\tilde{\bfZ}\cup \{D',C'\})(\hgt(C))=(\bfR(\hgt(C))\setminus\{C\})\cup \{D'\}$ and 
	\[\bfP(\hgt(C))=\begin{cases} \bfQ(\hgt(C)) & \text{if }\hgt(D)<\hgt(C)\\
	\bfQ(\hgt(C)) \cup \{D\} & \text{if }\hgt(D)=\hgt(C).
	\end{cases}\]
	We have $D\subset D'$. We need to show that any other strip in $\bfQ$ of height $\hgt(C)$ that is contained in $C$ is also contained in $D'$. 
	
	Let $D_\bfQ$ be such a strip. Since $\bfP$ is of type 1, $D_\bfQ$ cannot contain any box South of a box in $D$, so it cannot contain any box in $C'$. Since $D_\bfQ\subset C\subset D'\sqcup C'$ it follows that $D_\bfQ$ must be contained in $D'$.
	 Hence $\bfP(\hgt(C))$ is finer than $(\tilde{\bfZ}\cup \{C',D'\})(\hgt(C))$, thus $\bfP\succ \tilde{\bfZ}\cup \{C',D'\}$.
\end{proof}

\begin{figure}[ht]
	\begin{center}
		\begin{tabular}{c}
			\begin{tikzpicture}[x=\boxmini,y=\boxmini]
			\tikzset{vertex/.style={}}
			\tikzset{edge/.style={very thick}}
				\draw[fill=yellow] (1,1) -- (2,2) -- (3,1) -- (4,2) -- (9,-3)-- (14,2) -- (15,1) -- (9,-5) -- (4,0) -- (3,-1) -- cycle; 
			\tabpath{+,+,-,+,-,-,-,-,-,+,+,+,+,+,-,-}
			\tabpath{-,-,-,+,-,-,-,-,-,+,+,+,+,+,+,+}
			\tabpath{+,-,-,+,-,-,-,-,-,+,+,+,+,+,-,+}
			\tabpath{-,-,-,+,+,-,-,-,-,+,+,+,+,-}
			\tabpath{-,-,-,+,-,-,-,-,-,+,+,+,+,+,+,+}
			\tabpath{-,+,-,+,-,-,-,-,-,+,+,+,+,+,+}
			\node at (6,-1) {$C$};
			\node at (9,1) {$\bfR$};
			\end{tikzpicture}\qquad
			\begin{tikzpicture}[x=\boxmini,y=\boxmini]

			\tikzset{vertex/.style={}}
			\tikzset{edge/.style={very thick}}
			\draw[fill=orange] (9,-3) -- (12,0) --(11,1) -- (9,-1) -- (7,1) -- (6, 0) --cycle;
			\tabpath{+,+,-,+,-,-,-,-,-,+,+,+,+,+,-,-}
			\tabpath{-,-,-,+,-,-,-,-,-,+,+,+,+,+,+,+}
			\tabpath{+,-,-,+,-,-,-,-,-,+,+,+,+,+,-,+}
			\tabpath{-,-,-,+,+,-,-,-,-,+,+,+,+,-}
			\tabpath{-,-,-,+,-,-,-,-,-,+,+,+,+,+,+,+}
			\tabpath{-,+,-,+,-,-,-,-,-,+,+,+,+,+,+}
			\tabpath{+,-,-,+,+,-,+,-,-,+,+,-,+,-}
			\tabpath{+,-,-,+,-,+,-,-,-,+,+,-,-}
			\tabpath{-,-,-,+,-,+,+,-,-,+,+,+,-}
			\tabpath{-,-,-,+,-,-,+,-,-,+,+,-}
	
			\node at (8,-1) {$D$};
			\node at (9,1) {$\bfP$};
			\node at (3,0.3) {$D_\bfQ$};		
			\end{tikzpicture}\\
			\begin{tikzpicture}[x=\boxmini,y=\boxmini]
			\tikzset{vertex/.style={}}
			\tikzset{edge/.style={very thick}}
			\draw[fill=orange] (1,1) -- (2,2) -- (3,1) -- (4,2) -- (6,0)-- (7,1) -- (9,-1) -- (11,1) -- (12,0) --	(14,2) -- (15,1) -- (12,-2) -- (11,-1) -- (9,-3) -- (7,-1)-- (6,-2) -- (4,0) -- (3,-1)-- cycle;
			\draw[fill=yellow] (9,-5) -- (12,-2) --(11,-1) -- (9,-3) -- (7,-1) -- (6, -2) --cycle; 
			\tabpath{+,+,-,+,-,-,+,-,-,+,+,-,+,+,-,-}
			\tabpath{-,-,-,+,-,-,-,-,-,+,+,+,+,+,+,+}
			\tabpath{+,-,-,+,-,-,+,-,-,+,+,-,+,+,-}
			\tabpath{+,-,-,+,-,-,-,-,-,+,+,+,+,+,-,+}
			\tabpath{-,-,-,+,+,-,-,-,-,+,+,+,+,-}
			\tabpath{-,-,-,+,-,-,-,-,-,+,+,+,+,+,+,+}
			\tabpath{-,+,-,+,-,-,-,-,-,+,+,+,+,+,+}
			\tabpath{-,-,+,-,-,+,-,-,-,-}
			\tabpath{-,-,+,-,-,-,+,-,-,+,-}
			\tabpath{-,-,+,-,-,-,-,+,-,+,+,-}
			\tabpath{-,-,+,-,-,-,-,-,+,+,+,+,-}
			
			\node at (8,-3) {$C'$};
			\node at (5,0) {$D'$};
			\node at (9,2) {$\tilde{\bfZ}\cup \{C',D'\}$};
			\end{tikzpicture}
	
		\end{tabular}
	\end{center}
	\caption{An illustration of \Cref{claimpart}.}\label{finalpict}
\end{figure}
Since $\tilde{\bfZ}\cup\{C',D'\}$ is of type 1,  by $A(\lambda,\mu,\tilde{\bfZ}\cup\{C',D'\})$, we have
\begin{equation}\label{DZ}f_{D'}\circ f_{\tilde{\bfZ}\cup\{C'\}}\in \spa \langle f_\bfV \mid \bfV \in Y(\bfR,D) \rangle.
\end{equation}

Notice that $\ell(D')>\ell(D)$. From \Cref{claimpart} it follows that 
for any $\bfS\prec \tilde{\bfZ}\cup \{C'\}$ we have $Y(\bfS,D')\prec \tilde{\bfZ}\cup \{C',D'\}\prec X(\bfR,D)$, hence $Y(\bfS,D')\cug Y(\bfR,D)$.	
We can now apply by hypothesis (\ref{hyp6}) $B(\lambda,\mu,D',\bfS)$ to obtain
\begin{equation}\label{DS}
	f_{D'}\circ f_{\bfS}\in \spa \langle f_\bfV \mid \bfV \in Y(\bfR,D) \rangle.
\end{equation}

Combining \eqref{DCDC},\eqref{DZ} and \eqref{DS} we finally obtain
\[f_D\circ f_{\bfR,\tilde{o}}\cug \spa \langle f_\bfV \mid \bfV \in Y(\bfR,D) \rangle.\qedhere\]
\end{itemize}
\end{proof}

\section{Bases of the Intersection Cohomology of Schubert Varieties}

In \Cref{mainthm} and \Cref{maincoro} we have computed bases of the Hom spaces between indecomposable objects in $\sbim^I$. In this section, we apply these results to obtain description of bases of the indecomposable bimodules $B_\lambda^I\in \sbim^I$.

Let $\lambda$ and $\mu$ be paths with $\lambda\leq \mu$ and let $\bfP\in \Conf^1(\lambda,\mu)$.
We define $F_\bfP:=f_\bfP(1^\otimes)\in B_\mu^I$. It is a homogeneous element of degree $|\bfP|-\ell(\lambda)$.

\begin{definition}
	Let $B^I\in \sbim^I$. We say that a $R$-basis $\{b_i\}$ of $B^I$ is \emph{compatible with the support filtration} if for any $\lambda\in \Lambda_{n,i}$ the set $\{b_i\}\cap \Gamma^I_{\leq\lambda}(B^I)$ is a basis of $\Gamma^I_{\leq\lambda}(B^I)$.
\end{definition}

\begin{prop}
	For a fixed path $\mu$, the set $\{F_\bfP\}$, where $\bfP$ runs over all the sets $\Conf^1(\lambda,\mu)$ with $\lambda\leq \mu$, is a $R$-basis of $B_\mu^I$ compatible with the support filtration.
\end{prop}
\begin{proof}
	By comparing the graded dimensions, it is enough to check that the set $\{F_\bfP\}$, with $\bfP\in \Conf^1(\nu,\mu)$ and $\nu\leq \lambda$ generates $\Gamma_{\leq \lambda}^IB_\mu^I$.

	By \cite[Lemma 4.7]{Pat4} and \Cref{maincoro} we know that 
	\begin{equation}\label{gammalambda}\Gamma_{\leq \lambda}^IB_\mu^I=\spa_R\left\langle f_\bfP\circ g_\bfQ({1^\otimes})\mid \bfP\in \Conf^1(\nu,\mu),\bfQ\in \Conf^1(\nu,\lambda)\text{ and }\nu\leq \lambda\right\rangle.\end{equation}
	
	If $D$ consists of a single box, then it follows by the explicit description of the morphisms given in \Cref{diagramsec} that $g_D(1^\otimes)=1^\otimes$. Otherwise, if $\ell(D)>1$ by degree reasons we have that $g_D(1^\otimes)=0$. It follows that 
	\begin{equation}\label{gcbot}
g_\bfQ({1^\otimes})=\begin{cases}{1^\otimes}\in B_\nu^I& \text{if }|\bfQ|=\ell(\lambda)-\ell(\nu)\\0& \text{otherwise.} \end{cases}
	\end{equation}
	Combining \eqref{gammalambda} and \eqref{gcbot}, the proposition follows.
\end{proof}

The basis $\{F_\bfP\}$, as the basis $\{f_\bfP\}$, critically depends on the admissible orders chosen. More generally, for any $o'\in \adm(\bfP)$ we can define $F_{\bfP,o'}=f_{\bfP,o'}(1^\otimes)$. We have
\[F_{\bfP,o'}\in cF_\bfP +\spa \langle F_\bfQ \mid \bfQ\prec \bfP\rangle.\] 
for some $c\in \bbQ^*$.
Let $\barF_{\bfP,o'}$ denote the element $1\otimes F_{\bfP,o}\in \bbQ\otimes_R B_\mu^I$.
Then we have 
\[\barF_{\bfP,o'}\in c\barF_\bfP +\spa_\bbQ \langle \barF_\bfQ \mid \bfQ\prec \bfP, |\bfQ|=|\bfP|\rangle.\] 

\begin{example}
The following is the smallest example of two Dyck partitions $\bfP$ and $\bfQ$ of type 1 such that $|\bfP|=|\bfQ|$ and $\bfP\succ \bfQ$.

\begin{center}
\begin{tikzpicture}[x=\boxmini,y=\boxmini]
\tikzset{vertex/.style={}}
\tikzset{edge/.style={very thick}}
\tabpath{+,+,-,+,-,+,-}
\tabpath{-,-,-,+,+,+,+}
\tikzset{edge/.style={}}
\tabpath{+,-,-,+,+,-}
\tabpath{-,+,-,+,-}
\node at (0,-3) {$\bfP$};
\end{tikzpicture}\qquad 
\begin{tikzpicture}[x=\boxmini,y=\boxmini]
\tikzset{vertex/.style={}}
\tikzset{edge/.style={very thick}}

\tabpath{+,+,-,+,-,+,-}
\tabpath{-,-,-,+,+,+,+}
\tikzset{edge/.style={}}
\tabpath{+,-,-,+,-,+}
\tabpath{-,+,-,-}
\tabpath{-,-,+}
\node at (0,-3) {$\bfQ$};
\end{tikzpicture}
\end{center}
\end{example}

\begin{lemma}
	If the tableau corresponding to $\mu$ has only $2$ rows (or only $2$ columns) then there exists no Dyck partitions $\bfP,\bfQ\in \Conf^1(\lambda,\mu)$ such that $\bfP\succ \bfQ$ and $|\bfP|=|\bfQ|$.
\end{lemma}
\begin{proof}
	If the tableau of $\mu$ has only two rows, then every Dyck partition $\bfP\in \Conf^1(\lambda,\mu)$ for some $\lambda\leq \mu$ has at most one strip of length $>1$. The proposition easily follows.
\end{proof}

\begin{cor}
	If the tableau corresponding to $\mu$ has only $2$ rows (or only $2$ columns) the element $\barF_{\bfP,o}\in \bar{B}_\mu^I=IH(X_\mu,\bbQ)$ does not depend on $o\in \adm(\bfP)$ up to a scalar.
\end{cor}

\subsection{Comparison with the Schubert basis}

For any $\lambda\in \Lambda_{n,i}$ the Schubert variety $X_\lambda\cug \Gr(i,n)$ is oriented and $T$-invariant, so it induces a class in the $T$-equivariant Borel--Moore homology (cf. \cite[\S 1]{Brion2}) \[[X_\lambda]\in H^{BM}_{2\ell(\lambda),T}(\Gr(i,n),\bbQ).\] 
The set $\{[X_\lambda]\}_{\lambda\in \Lambda_{n,i}}$ is a $R$-basis of $H^{BM}_{\bullet,T}(\Gr(i,n),\bbQ)$.
By taking the dual basis, we obtain a left $R$-basis $\{\calS_\lambda\}_{\lambda\in \Lambda_{n,i}}$ of $H^\bullet_T(\Gr(i,n),\bbQ)$, called the \emph{Schubert basis}.

For $\mu\in \Lambda_{n,i}$ consider the embedding $i_\mu:X_\mu\hookrightarrow \Gr(i,n)$. Then $\{i_\mu^*\calS_{\lambda}\}_{\lambda\leq \mu}$ is also a left $R$-basis of $H^\bullet_T(X_\mu,\bbQ)$.
 By abuse of notation, we denote $i_\mu^*\calS_\lambda$ simply by $\calS_\lambda$. Using the embedding $H_T^\bullet(X_\mu,\bbQ)(\ell(\mu))\cug IH_T^\bullet(X_\mu,\bbQ)=B_\mu^I$ we regard $\calS_\lambda$ as an element of $(B_\mu^I)^{2\ell(\lambda)-\ell(\mu)}$.

Let $\bfU(\lambda)$ be the maximal Dyck partition in $\Conf^1(\lambda,\mu)$, i.e. the one which consists only of single boxes. In the previous section we introduced the basis $\{F_\bfP\}$ of $B_\mu^I$. Our final goal is to show that $\calS_\lambda$ can be obtained as the dual of the basis element $F_{\bfU(\lambda)}$.

Since $B_\mu^I\cong \bbD B_\mu^I$ and $\End^0(B_\mu^I)\cong \bbQ$ there exists a unique up to a scalar left invariant form on $B_\mu^I$. Let $(\vec{I},\vec{J})$ be a reduced translation pair such that $B_\mu^{I}\cong BS(\vec{I},\vec{J})$. The intersection form $\langle -,-\rangle_{(\vec{I},\vec{J})}$ does not depend on the chosen reduced translation pair, so we can define $\langle - ,-\rangle_{\mu}:=\langle -,-\rangle_{(\vec{I},\vec{J})}$.

\begin{lemma}
	After having identified $B_\mu^I$ with $BS(\vec{I},\vec{J})$ via the isomorphisms sending $1^\otimes$ to $1^\otimes$ we have
	$F_{\bfU(\lambda^{id})}=\ctop(\vec{I},\vec{J})$ up to lower degree terms.
\end{lemma}
\begin{proof}
	Both $F_{\bfU(\lambda^{id})}$ and $\ctop(\vec{I},\vec{J})$ are in the image of a morphism of degree $\ell(\vec{I},\vec{J})$ form $R_I$. This means that there exists a scalar $q\in \bbQ$, such that $F_{\bfU(\lambda^{id})}=q\ctop$ up to lower terms in $\spa_R \langle F_\bfQ \mid \bfQ\prec \bfU(\lambda^{id})\rangle$.
	
	 More precisely, we have $F_{\bfU(\lambda^{id})}=f_{\bfU(\lambda^{id})}(1)$ and $\ctop(\vec{I},\vec{J})=f_{\carcd} \circ (f_{\carcd}\otimes 1)\circ \ldots \circ (f_{\carcd}\otimes 1)(1)$ as in \eqref{tantef}. 
	 	Notice that $f_{\bfU(\lambda^{id})}$ is the composition of morphisms of type 1a)-1d)  (cf. \Cref{diagramsec}), and for each of these morphisms, when we take the flip, we obtain a morphism which sends $1^\otimes$ to $1^\otimes$. The same holds for the morphisms of the form $f_{\carcd} \otimes 1$. Since the flip of both morphism sends $1^\otimes$ to $1^\otimes$ we deduce that $q=1$.
\end{proof}

  Define $\ctop(\mu):=F_{\bfU(\lambda^{id})}$. Notice that $\ctop(\mu)$ does not depend on the admissible order chosen for $\bfU(\lambda^{id})$ up to smaller terms. Moreover, $\form_\mu$ is the unique invariant form 
$\form_\mu: B_\mu^I\times B_\mu^I\raw R$
such that
\[\langle {1^\otimes},\ctop(\mu)\rangle_\mu =1. \]
Recall by \Cref{adjoint} that for any Dyck strip $D$ the maps $f_D$ and $g_D$ are adjoint to each other with respect to the form $\form_\mu$.

We look now more closely at Dyck strips which are single boxes. In this case we can give an alternative description of the corresponding degree one maps.

Let $\lambda,\mu$ be paths with $\mu=\lambda+C$. Assume that the Dyck strip $C$ is a single box, and let $j$ be the label of $C$ (cf. \Cref{boxdef}) and $s_j$ be the corresponding simple reflection.
 Let $w_\mu$ and $w_\lambda$ be the elements in $W^I$ corresponding to $\mu$ and $\lambda$, so that we have $w_\mu=s_jw_\lambda$.

 The bimodule $B_\mu^I$ is a direct summand of $B_{s_j}B_\lambda^I$, and since $B_{s_j}B_\lambda^I$ is perverse this summand is uniquely determined. We denote by $\iota_\mu: B_\mu^I\raw B_{s_j}B_\lambda^I$ and $\pi_\mu: B_{s_j}B_\lambda^I\raw B_\mu^I$ respectively the inclusion and the projection of this summand, so that $e_\mu=\iota_\mu\circ \pi_\mu\in \End^0(B_sB_\lambda^I)$ is the corresponding idempotent.

We can choose $\iota_\mu$ so that $\iota_\mu(1^\otimes)=1^\otimes$. This also forces $\pi_\mu(1^\otimes)=1^\otimes$. 
Since the vector space $\Hom^0(B_\mu,B_{s_j}B_\lambda)$ is one-dimensional, by \Cref{flippreserves1} we have $\bar{\iota_\mu}(1^\otimes)=1^\otimes$, and therefore $\bar{\iota_\mu}=\pi_\mu$.

\begin{lemma}
	We have \[f_C= \pi_\mu\circ (f_{\carcd}\otimes \Iden_{B_\lambda^I}):R\otimes_R B_\lambda^I=B_\lambda^I\rightarrow B_\mu^I(1)\]
	where $f_{\carcd}: R\raw B_{s_j}(1)$ is the morphism defined in \Cref{boxes}.
\end{lemma}
\begin{proof}

Let $\phi:= \pi_\mu\circ (f_{\carcd}\otimes \Iden_{B_\lambda^I})$.  The morphism $\phi$ is of degree $1$, hence it is a scalar multiple of $f_C$. 
After taking the adjoint, we have 
\[\bar{\phi}=(\bar{f_{\carcd}}\otimes \Iden_{B_\lambda^I})\circ \bar{\pi_\mu}=(f_{\carcu}\otimes \Iden_{B_\lambda^I})\circ \iota_\mu\]
where $f_{\carcu}:B_{s_j}\raw R(1)$ is the map defined by $f\otimes g\mapsto fg$.
It follows that $\bar{\phi}({1^\otimes})={1^\otimes}=g_C({1^\otimes})$, and we deduce $\bar{\phi}=g_C=\bar{f_C}$, hence $\phi=f_C$.
\end{proof}

\begin{definition}
	Let $\lambda$ be a path. We denote by $\boxplus(\lambda)$ the set of Dyck strips of length one (i.e. boxes) that can be added to $\lambda$ and by $\boxminus(\lambda)$ the set of boxes that can be removed from $\lambda$.
\end{definition}

Recall that the we have fixed an index $i$ so that $\{s_i\}=S\setminus I$.
The $R$-module structure on $H^\bullet_T(X_\mu,\bbQ)$ can be deduced from the equivariant Pieri's formula (see for example \cite[Proposition 2]{KnT}):
\begin{equation}\label{Pieri}
\calS_\lambda\cdot h=w_{\lambda}(h) \calS_\lambda+ \partial_{i}(h) \sum_{C\in \boxplus(\lambda)} \calS_{\lambda+C}.\end{equation}
There is a similar formula which holds in $B_\mu^I$:

\begin{prop}[Equivariant Pieri's formula in intersection cohomology]\label{newpieri}
Let $\nu,\mu$ be paths with $\nu\leq \mu$.	Let $\bfP\in\Conf^1(\nu,\mu)$, $o\in \adm(\bfP)$ and $h\in (\frh^*)^I$. Then we have 
	\[F_{\bfP,o}\cdot h=w_\nu(h)F_{\bfP,o}+\partial_{i}(h)\sum_{C\in \boxminus(\nu)} F_{ \{C\}\cup \bfP,(C,o)},\]
			where $(C,o)\in \adm(\{C\}\cup\bfP)$ is the order beginning with $C$ and continuing as in $o$.
\end{prop}
\begin{proof}	
	If $C$ is a Dyck strip we simply write $F_C$ instead of $F_{\{C\}}$. We first show by induction on $\nu$ that
	\begin{equation}\label{Pieribot}
	1^\otimes\cdot h=w_\nu(h){1^\otimes} +\partial_{i}(h)\sum_{C\in \boxminus(\nu)}F_{C}\in B_\nu^I.
	\end{equation}
	Let $D\in \boxminus(\nu)$ and let $\lambda:=\nu-D$. Let $s_j$ be the simple reflection corresponding to the label of $D$.
	 By the nil-Hecke relation \cite[(5.2)]{EW2} we have \[\partial_{j}(w_\nu(h))f_{\carcd}(1)=w_{\nu}(h)\otimes 1-1\otimes w_{\lambda}(h)\in B_{s_j},\]
	 and by induction 
	 \begin{equation}\label{nhrepeated}
	 \partial_{j}(w_\nu(h))(f_{\carcd}(1)\otimes 1^\otimes)=w_\nu(h){1^\otimes}-1^\otimes\cdot h+\partial_{i}(h)\sum_{C\in \boxminus(\lambda)}1\otimes F_{C}\in B_{s_j}B_\lambda^I.
	 \end{equation}
	 
	 For $C\in \boxminus(\lambda)$ we have $F_C\in \Gamma^I_{\leq \lambda-C} (B_\lambda^I)$.
	 Assume $j$ is not a valley for $\lambda-C$, or equivalently that $C\not \in \boxminus(\nu)$. Then also $1\otimes F_C\in \Gamma^I_{\leq \lambda-C} (B_{s_j}B_\lambda^I)$, thus 
	 $\pi_\nu(1\otimes F_C)\cug \Gamma^I_{\leq \lambda-C} (B_\nu^I)$. Since $\Gamma^I_{\leq \lambda-C} (B_\nu^I)$ is concentrated in degrees $\geq -\ell(\nu)+4$ we must have $\pi_\nu(1\otimes F_C)=0$.
	 
	 Assume now that $j$ is a valley for $\lambda-C$ and $C\in \boxminus(\nu)$. In this case the boxes $C$ and $D$ are distant, and we have the following diagram.
	 	 \begin{center} 
	 	\begin{tikzpicture}
	 	\matrix(m)[matrix of math nodes, column sep=50pt, row sep=20pt]
	 	{B_{\nu-C} & B_{s_j}B_{\lambda-C}^I\\
	 		B_\nu & B_{s_j}B_\lambda^I\\};
	 	\path[->] (m-1-1) edge node[above]{$\iota_{\nu-C}$}(m-1-2)
	 	(m-1-1) edge node[left] {$f_C$} (m-2-1)
	 	(m-1-2) edge node[left] {$1\otimes f_C$} (m-2-2)
	 	(m-2-2) edge node[above] {$\pi_\nu$} (m-2-1);
	 	\end{tikzpicture}
	 \end{center}
	 Since $\Hom^1(B_{\nu-C},B_\nu)\cong \bbQ$, the diagram is commutative up to a scalar. We can check that it does commute by taking the flip because we have
	 $\pi_{\nu-C}\circ (1\otimes g_C)\circ \iota_\nu(1^\otimes)=1^\otimes=g_C(1^\otimes)$. 
	 Thus, we have
	 $F_C=\pi_\nu(1\otimes F_C)\in B_\nu^I$. We can now apply $\pi_\nu$ to \eqref{nhrepeated} and get
	\[\partial_{j}(w_\nu(h))\pi_\nu(f_{\carcd}(1)\otimes 1^\otimes)=w_\nu(h){1^\otimes}-1^\otimes\cdot h+\partial_{i}(h)\sum_{C\in \boxminus(\nu)\setminus\{D\}}F_C\in B_{\nu}^I.\]

	Recall that $F_D=\pi_\nu(f_{\carcd}(1)\otimes 1^\otimes)$.
	To show \eqref{Pieribot}, it remains to check that $\partial_{j}(w_\nu(h))=-\partial_{i}(h)$. Consider the reflection $t:=w_\lambda^{-1} s_j w_\lambda\in W$. It is not contained in $W_I$, so we can write $t=ws_iw^{-1}$ for some $w\in W_I$. Now $\partial_{j}(w_\nu(h))=-\partial_{j}(w_\lambda(h))=-\partial_{i}(w(h))=-\partial_{i}(h)$.
	
	Finally, applying $f_{\bfP,o}$ to \eqref{Pieribot} we obtain the desired statement.
\end{proof}

\begin{lemma}
	Let $\lambda,\nu\leq \mu$ and $\bfP\in \Conf^1(\nu,\mu)$. We have 
	\[\langle F_{\bfP},\calS_\lambda\rangle_\mu=\begin{cases}
	1 & \text{if } \bfP=\bfU(\lambda)\\
	0 & \text{otherwise.}
	\end{cases}\]
\end{lemma}
\begin{proof}

	If $C$ is a Dyck strip that can be removed from $\mu$, then $g_C: B_\mu^I\raw B_{\mu-C}^I$ is a morphism of $(R,R^I)$-bimodules, hence it also commutes with the $H^\bullet_T(\Gr(i,n),\bbQ)$-action. By \eqref{gcbot} it follows that 
	 for $\bfP\in \Conf^1(\nu,\mu)$ we have 
 	\[g_{\bfP,o}(\calS_\lambda)=\begin{cases}\calS_\lambda& \text{if $\bfP$ only consists of single boxes and }\lambda\leq \nu\\
 	0& \text{otherwise.}\end{cases}\]

 	From the adjointness, we have $\langle F_{\bfP},\calS_\lambda\rangle_\mu = \langle {1^\otimes}, g_{\bfP}(\calS_\lambda)\rangle_\nu$. It remains to show that $\langle {1^\otimes}, \calS_\lambda\rangle_\nu=\delta_{\lambda,\nu}$. We show this by induction on $\nu$.
 	
 	Recall that $\calS_\lambda=0$ in $B_\nu^I$ unless $\lambda\leq \nu$. 
 	If $\lambda<\nu$ by degree reasons it follows $\langle 1^\otimes,\calS_\lambda\rangle_\nu=0$. Assume now $\lambda=\nu$ and let $h\in (\frh^*)^I$ be such that $\partial_{i}(h)=1$. 
 	Let $D\in \boxminus(\nu)$.
 	Then by \eqref{Pieri} and \eqref{Pieribot} we have
 	\begin{multline*} \langle 
 	{1^\otimes}, \calS_\nu\rangle_\nu =\langle {1^\otimes}, \calS_{\nu - D}\cdot h-w_{\nu - D}(h)\cdot \calS_{\nu - D}\rangle_\nu = \langle {1^\otimes}\cdot h -w_{\nu - D}(h)\cdot  {1^\otimes}, \calS_{\nu - D}\rangle_\nu = \\ 
 	=\left\langle\sum_{C\in \boxminus(\nu)} F_C, \calS_{\nu - D}\right\rangle_\nu= \sum_{C\in\boxminus(\nu)}\left\langle {1^\otimes}, g_C(\calS_{\nu - D})\right\rangle_{\nu - C}.
 	\end{multline*}
 	Now, $g_C(\calS_{\nu - D})=0\in B_{\nu-C}^I$ unless $D=C$. We obtain $\langle 
 	{1^\otimes}, \calS_\nu\rangle_\nu=\left\langle {1^\otimes}, g_D(\calS_{\nu - D})\right\rangle_{\nu - D}$ and we conclude by induction since because 
 	\[\langle  \langle {1^\otimes},g_D( \calS_{\nu-D})\rangle_{\nu-D}=\langle 1^\otimes,S_{\nu-D}\rangle_{\nu-D}=1. \qedhere\]
\end{proof}

\begin{cor}
	Let $\{F_{\bfP}^*\}$ be the dual basis of $\{F_{\bfP}\}$ with respect to the form $\langle-,-\rangle_\mu$. Then $\calS_\lambda=F_{\bfU(\lambda)}^*$, where $\bfU(\lambda)\in\Conf^1(\lambda,\mu)$ is the Dyck partition which consists only of single boxes.
\end{cor}

\appendix

\section{Appendix: Invariant forms on generalized Bott--Samelson bimodules}\label{appendix}

In this appendix we introduce invariant forms on generalized Bott-Samelson bimodules, generalizing the definitions from \cite[\S 3.1]{EW1}. We then  
show that flipping a singular diagram is the same as taking the adjoint with respect to the invariant forms.

Let $(W,S)$ be a Coxeter system and $I\subset S$ be finitary. Consider the category $\sbim^I$ of $I$-singular Soergel bimodules. Let $\frh^\vee$ be a representation of $W$ as in \cite[Proposition 2.1]{S4} with simple roots $\alpha_s\in \frh^\vee$ and simple coroots $\alpha_s\in \frh^\vee$.

For $s\in S$, let  $\partial_s:R\raw R(-2)$ denote the Demazure operator \[\partial_s(f)=\frac{f-s(f)}{\alpha_s}.\]
Note that if $f$ is of degree $2$ then $\partial_j(f)=\alpha_j^\vee(f)$. The Demazure operators satisfy the braid relations, hence for any $x\in W$ with reduced expression $\undx=s_1s_2\ldots s_k$ we can define $\partial_x:=\partial_{1}\partial_{2}\ldots \partial_{k}:R\raw R(-2k)$.

Let $B^I\in \sbim^I$ be a $I$-singular Soergel bimodule.
A left invariant bilinear form on $B$ is a graded bilinear map
\[\langle -,-\rangle: B^I\times B^I \raw R\]
such that for any $b,b'\in B^I$, $f\in R$ and $g \in R^I$ we have
\[\langle fb, b' \rangle = \langle b,fb'\rangle =f\langle b,b'\rangle \]
\[ \langle bg,b'\rangle=\langle b,b'g\rangle.\]

Let $(\vec{I},\vec{J})$ be a translation pair with $\vec{J}=(J_1,J_2,\ldots,,J_k)$ and $\vec{I}=(I_0,I_1,I_2 \ldots, I_k)$ and assume that $I_0=\emptyset$ and $I_k=I$. Consider the translation pair $(\vec{I'},\vec{J'})$ with \begin{equation}\label{IJprime}\vec{J'}=(J_2,J_3,\ldots,J_k)\qquad\text{ and }\qquad\vec{I'}=(\emptyset,I_2,I_3,\ldots,I_k).\end{equation} 
Recall the notation of \Cref{boxes}. The morphism $f_{\carcd}:R^{I_1}\raw R^{I_{1}}\otimes_{R^{J_1}}R^{I_1}(\ell(J_1)-\ell(I_1))$ induces a morphism 
\[ f_{\carcd}\otimes \Iden: BS(\vec{I'},\vec{J'})\raw BS(\vec{I},\vec{J}).\]

\begin{definition}
	We define, by induction on $k$, an element $\ctop(\vec{I},\vec{J})$ as $(f_{\carcd}\otimes \Iden)(\ctop(\vec{I'},\vec{J'}))$, where $\ctop((\emptyset),\emptyset)=1\in R$.
\end{definition}

If $\undw=s_1s_2\ldots s_k$ is a word with $s_i\in S$, then the ordinary Bott--Samelson bimodule $BS(\undw)$ can be realized as a generalized Bott--Samelson bimodule $BS(\vec{I},\vec{J})$, with $\vec{J}=(\{s_1\},\{s_2\},\ldots \{s_k\})$ and $\vec{I}=(\emptyset,\emptyset,\ldots,\emptyset)$. In this case, the corresponding element $\ctop(\undw):=\ctop(\vec{I},\vec{J})\in BS(\undw)$ coincides with the element $\ctop$ defined in \cite[\S 3.4]{EW1}.

Let $t\in W$ be a reflection. Then we can write $t=wsw^{-1}$ with $ws>w$ and $s\in S$. The root $\alpha_t:=w(\alpha_s)\in \frh^*$ is well defined and it is positive, i.e. it can be written as a positive linear combination of simple roots. We can consider the operator \[\partial_t^T:R \raw R(-2)\] 
defined by $\partial_t^T(f)=\frac{f-t(f)}{\alpha_t}$. (We use the non-standard notation $\partial_t^T$ to distinguish it from $\partial_t$, as they only coincide when $t\in S$.) If $t=wsw^{-1}$ with $ws>w$ and $s\in S$, then $\partial_t^T(f)=w\left(\partial_s({w}^{-1}(f))\right)$ for all $f\in R$. In particular, if $f\in R$ is of degree $2$, then $\partial_t^T(f)=\partial_s({w}^{-1}(f))=\alpha_s^\vee(w^{-1}(f))$.
We write $y\Rra{t}x$ if $t\in T$, $yt=x$ and $\ell(y)+1=\ell(x)$.

\begin{lemma}\label{demaformula}
	Let $x\in W$.
	Let $f,g\in R$ be such that $f$ is of degree $2$ and $g$ of degree $2\ell(x)-2$. Then
	\[ \partial_x(fg)= \sum_{y\Rra{t}x}\partial_t^T(f)\partial_y(g). \]
\end{lemma}
\begin{proof}
	The proof is by induction on $\ell(x)$. The case $\ell(x)\leq 1$ is trivial. Let $x=x's$ with $s\in S$ and $x>x'$. By induction on the length we obtain
	\[\partial_{x}(fg)=\partial_{x'}\partial_s(fg)=\partial_{x'}\left(g\partial_s(f)+s(f) \partial_s(g)\right)=\partial_s(f)\partial_{x'}(g)+\sum_{y'\Rra{u}x'} \partial_u^T(s(f))\partial_{y'}(\partial_s(g)).\]
	
	If $y's<y'$ the term $\partial_{y'}(\partial_s(g))$ vanishes, while if $y's>y'$ and $y'\Rra{u}x'$, then also $y's\Rra{sus}x$. In particular we have $s\neq u$, so $s(\alpha_u)$ is a positive root and $\partial_u^T(s(f))=\partial_{sus}^T(f)$. The claim now follows.
\end{proof}

Consider the right $R^I$-module $\bar{BS(\vec{I},\vec{J})}:=\bbR\otimes_R BS(\vec{I},\vec{J})$.
Its degree  $\ell(\vec{I},\vec{J})$ component is  of dimension one. Obvioulsy, in the bimodule $BS(\vec{I},\vec{J})$ the dimension of the degree $\ell(\vec{I},\vec{J})$ component is much bigger. However, among these, the element $\ctop(\vec{I},\vec{J})$ is basically unique up to lower degree terms. We make this precise with the following definition.

\begin{definition}
	We define $LT(BS(\vec{I},\vec{J}))$ to be the left $R$-submodule of $BS(\vec{I},\vec{J})$ generated by elements of degree less than $\ell(\vec{I},\vec{J})$. We refer to the elements of $LT(BS(\vec{I},\vec{J}))$ as \emph{lower degree terms}. 
\end{definition}
With this definition, every two elements of $BS(\vec{I},\vec{J})$ of degree $\ell(\vec{I},\vec{J})$ are multiple of each other up to lower degree terms.
Let $(\frh^*)^I$ be the subspace of $W_I$-invariant elements in $\frh^*$ (or, equivalently, $(\frh^*)^I$ is the subspace of elements of degree $2$ in $R^I$).
We say that $\rho\in (\frh^*)^I$ is \emph{ample} if $\partial_s(\rho)>0$ for every $s\in S\setminus I$. Since the simple coroots $\alpha_s^\vee$ are linearly independent there always exists an ample element $\rho\in (\frh^*)^I$.

\begin{lemma}\label{rhoample}
	Let $\rho\in (\frh^*)^I$ be ample. Let $w\in W^I$ and $w = s_1s_2\ldots s_l$ be a reduced expression for $w$. Then for any
	$i\leq l$ we have 
	\[\partial_{s_i}(s_{i+1}s_{i+2}\ldots s_l(\rho)) > 0\]
\end{lemma}
\begin{proof}
	See \cite[Lemma 4.2.1]{PatPhD}.
\end{proof}

We say that a word $\undw$ is $I$-reduced if $\undw$ is a reduced word for $w\in W^I$. 
If $B\in \sbim^J$ and $J\subset I$ we denote by $B_I=B\otimes_{R^I}R^I\in \sbim^I$ its restriction to a $(R,R^I)$-bimodule.

\begin{lemma}\label{formulaperN}
	Let $\undw$ be a word of length $k$, $\rho \in (\frh^*)^I$ and consider $1^\otimes:=1\otimes 1\otimes \ldots \otimes 1 \in BS(\undw)_I$. Then $1^\otimes \cdot\rho^{\ell(\undw)}=N\ctop(\undw)$ up to lower degree terms. If $\undw$ is not $I$-reduced then $N=0$ while if $\undw$ is a $I$-reduced word for $w$ then $N=\partial_w(\rho^k)$.
	
	If, moreover, $\rho$ is ample then $N>0$. 
\end{lemma}
\begin{proof}
	If $\undw$ is not $I$-reduced then 
	\begin{equation}\label{BSnotrex}BS(\undw)_I=\bigoplus_z B_z^I(m)\end{equation} and all the elements $z$ occurring in the RHS in \eqref{BSnotrex} satisfy $\ell(z)<k$. Then, by degree reasons $1^\otimes \cdot \rho^k$ has to be an element of $LT(BS(\undw)_I)$, hence $N=0$. 
	
	Assume now that $\undw$ is a $I$-reduced word for $w\in W^I$ with $\undw=s_1s_2\ldots s_k$. For any $i$, let $\undw_i=s_1s_2\ldots s_{i-1}s_{i+1}\ldots s_k$. Let $v_i=s_is_{i+1}\ldots s_k$. As in the proof of \cite[Lemma 3.10]{EW1} one can see that 
	\[ 1^\otimes \cdot \rho^k=N\ctop(\undw)\text{+ lower degree terms}\]
	with $N= \sum_{i}\partial_{s_i}(v_{i+1}(\rho))N_i$, where $N_i$ is the coefficient of $\ctop(\undw_i)$ in $1^\otimes\cdot \rho^{k-1}\in BS(\undw_i)$. 
	By induction we have 
	\[ N = \sum_{i\;:\; \undw_i\; I\text{-reduced}}\partial_{s_i}(v_{i+1}(\rho))\partial_{w_i}(\rho^{k-1})=\sum_{y\Rra{t}w}\partial^T_t(\rho)\partial_{y}(\rho^{k-1})\]
	Finally, applying \Cref{demaformula} we get $N=\partial_w(\rho^k)$. 
	
	Assume now that $\rho$ is ample. Then, by induction, for any $i$ such that $\undw_i$ is $I$-reduced we have $N_i>0$ and also $\partial_{s_i}(v_{i+1}(\rho))>0$
	by \Cref{rhoample}. Hence $N>0$. 
\end{proof}

Let $I\subset J$ be finitary subsets and fix a reduced expression $\underline{w_Jw_I}$ for $w_Jw_I$. Then
$BS((\emptyset,I),(J))=R\otimes_{R^{J}}R^{I}(\ell(J)-\ell(I))$ is a direct summand of $BS(\underline{w_Jw_I})_I$ and we have a natural embedding of $(R,R^I)$-bimodules 
\[\vartheta: R\otimes_{R^{J}}R^{I}(\ell(J)-\ell(I))\raw BS(\underline{w_Jw_I})_I\]
induced by sending $1^\otimes\in R\otimes_{R^{J}}R^{I}$ to $1^\otimes\in BS(\underline{w_Jw_I})_I$.
Recall the element $\ctop((\emptyset,I),(J))\in R\otimes_{R^{J}}R^{I}(\ell(J)-\ell(I))$. By definition, we have $\ctop((\emptyset,I),(J))=\sum_i x_i\otimes y_i$ where $\{x_i\}$ and $\{y_i\}$ are bases of $R^I$ over $R^J$ with $\partial_{w_Iw_J}(x_iy_j)=\delta_{ij}$.

\begin{lemma}\label{thetaembe}
	The embedding $\vartheta:R\otimes_{R^{J}}R^{I}(\ell(J)-\ell(I))\raw BS(\underline{w_Jw_I})_I$ sends $\ctop((\emptyset,I),(J))$ to $\ctop(\underline{w_Jw_I})$, up to lower degree terms.
\end{lemma}
\begin{proof}
	Since $\vartheta$ sends $1^\otimes$ to $1^\otimes$, it also sends $1^\otimes \cdot \rho^{\ell(w_Jw_I)}$ to $1^\otimes\cdot \rho^{\ell(w_Jw_I)}$ for every $\rho\in (\frh^*)^I$. Fix $\rho$ ample so that $\partial_{w_Jw_I}(\rho^{\ell(w_Jw_I)})\neq 0$. Then, by \Cref{formulaperN}, it suffices to show that 
	\[1\otimes \rho^{\ell(w_Jw_I)}=\partial_{w_Jw_I}(\rho^{\ell(w_Jw_I)})\ctop((\emptyset,I),(J)) + \text{lower degree terms}\]
	
	If $\ctop((\emptyset,I),(J))=\sum_i x_i\otimes y_i$, there exists a unique index $i_0$ such that $x_{i_0}\in R^I$ is of degree $0$ and $y_{i_0}\in R^I$ is of degree $2\ell(w_Jw_I)$. Then $\ctop((\emptyset,I),(J))$ is equal to $x_{i_0}\otimes y_{i_0}=1\otimes x_{i_0}y_{i_0}$ up to lower degree terms. On the other hand, when expressing $\rho^{\ell(w_Jw_I)}$ in the $\{y_i\}$ basis, the coefficient of $y_{i_0}$ is \[\frac{\partial_{w_Jw_I}(\rho^{\ell(w_Jw_I)})}{\partial_{w_Jw_I}(y_{i_0})}=x_{i_0}\partial_{w_Jw_I}(\rho^{\ell(w_Jw_I)}),\]
	so we have $1\otimes \rho^{\ell(w_Jw_I)} = \partial_{w_Jw_I}(\rho^{\ell(w_Jw_I)})\cdot 1\otimes x_{i_0}y_{i_0}$
	up to lower degree terms and the claim follows.	
\end{proof}

Reiterating \Cref{thetaembe}, we get a chain of embeddings
\begin{multline}\label{embintoBS}BS(\vec{I},\vec{J})=R \otimes_{R^{J_1}} R^{I_1}\otimes_{R^{J_2}} \ldots \otimes_{R_{J_k}}R^{I_k}(\ell(\vec{I},\vec{J}))\hookrightarrow\\
	\hookrightarrow BS(\underline{w_Jw_I})\otimes_{R^{J_2}}R^{I_2} \otimes_{R_{J_k}}R^{I_k}(\ell(\vec{I},\vec{J})-\ell(w_Jw_I))\hookrightarrow \ldots \hookrightarrow\\
	\hookrightarrow BS(\underline{w_{J_1}w_{I_1}}\;\underline{w_{J_2}w_{I_2}}\ldots \underline{w_{J_k}w_{I_k}})_I.
\end{multline}

\begin{lemma}\label{deltatoctop}
	The embedding in \eqref{embintoBS} sends $\ctop(\vec{I},\vec{J})$ to $\ctop\in BS(\underline{w_{J_1}w_{I_1}}\ldots \underline{w_{J_k}w_{I_k}})_I$ up to lower degree terms.
\end{lemma}
\begin{proof}
	Consider the following diagram, where $\vec{I'}$ and $\vec{J'}$ are defined as in \eqref{IJprime} and $\alpha$ and $\beta$ are embeddings as in \eqref{embintoBS}.
	\begin{center}
		\begin{tikzpicture}
			\node (a) at (4,3.5) {$BS(\vec{I},\vec{J})$};
			\node (b) at (0,2) {$BS(\vec{I'},\vec{J'})$};
			\node (c) at (9,2)
			{$BS(\underline{w_{J_1}w_{I_1}})\otimes_{R} BS(\vec{I'},\vec{J'})$};
			\node (e) at (0,0) {$BS(\underline{w_{J_2}w_{I_2}}\ldots \underline{w_{J_k}w_{I_k}})_I$};
			\node (f) at (9,0) {$BS(\underline{w_{J_1}w_{I_1}}\;\underline{w_{J_2}w_{I_2}}\ldots \underline{w_{J_k}w_{I_k}})_I$};
			\path[->] 
			(e) edge node[above] {$\ctop(\underline{w_{J_1}w_{I_1}})\otimes \Iden$} (f)
			(b) edge node[above, xshift=-5]
			{$f_{\carcd}\otimes \Iden$} (a)
			edge node[above] {$\ctop(\underline{w_{J_1}w_{I_1}})\otimes\Iden$} (c);
			\path[right hook->]
			(a) edge node[above] {$\alpha$} (c)
			(c) edge node[right] {$\Iden\otimes \beta$} (f) 
			(b) edge node[right] {$\beta$} (e);
		\end{tikzpicture}
	\end{center}
	Notice that the lower square is commutative. By \Cref{thetaembe} the top triangle is ``commutative up to lower degree terms,'' meaning that we have \[\alpha\circ (f_{\carcd}\otimes \Iden) \ctop(\vec{I'},\vec{J'})=\ctop(\underline{w_Jw_I})\otimes \ctop(\vec{I'},\vec{J'})\] up to lower degree terms. 
	Therefore we have
	\begin{multline*}(\ctop(\underline{w_Jw_I})\otimes\Iden)\circ \beta(\ctop(\vec{I'},\vec{J'}))=(\Iden\otimes \beta) \circ \alpha \circ (f_{\carcd}\otimes \Iden)(\ctop(\vec{I'},\vec{J'}))=\\=(\Iden\otimes \beta) \circ \alpha (\ctop(\vec{I},\vec{J}))\end{multline*}
	up to lower degree terms. We conclude since by induction we have $\beta(\ctop(\vec{I'},\vec{J'}))=\ctop(\underline{w_{J_2}w_{I_2}}\ldots \underline{w_{J_k}w_{I_k}})$ up to lower degree terms.
\end{proof}

\begin{lemma}
	The element $\ctop(\vec{I},\vec{J})\in BS(\vec{I},\vec{J})$ together with $LT(BS(\vec{I},\vec{J}))$ generates $BS(\vec{I},\vec{J})$ as a left $R$-module.
\end{lemma}
\begin{proof}
	By looking at $\grrk (BS(\vec{I},\vec{J}))$ it is enough to show that $\ctop(\vec{I},\vec{J})$ does not belong to $LT(BS(\vec{I},\vec{J}))$. This is clear by looking at the embedding in \eqref{embintoBS}, since the image of $\ctop(\vec{I},\vec{J})$ contains $\ctop(\underline{w_{J_1}w_{I_1}}\;\underline{w_{J_2}w_{I_2}}\ldots \underline{w_{J_k}w_{I_k}})$ with non-trivial coefficient, hence it is not contained in the submodule $LT(BS(\underline{w_{J_1}w_{I_1}}\;\underline{w_{J_2}w_{I_2}}\ldots \underline{w_{J_k}w_{I_k}})_I)$.
\end{proof}

In other words, we have the following decomposition of left $R$-bimodules.
\begin{equation}
	\label{deltadec} BS(\vec{I},\vec{J}) = R\cdot \ctop(\vec{I},\vec{J}) \oplus LT(BS(\vec{I},\vec{J})).
\end{equation}
We can finally define the intersection form of a generalized Bott--Samelson bimodule.

\begin{definition}
	Let $\Trace:BS(\vec{I},\vec{J})\raw R$ be the $R$-linear map which returns the coefficient of $\ctop(\vec{I},\vec{J})$ with respect to the decomposition \eqref{deltadec}.
	We denote by $(-)\cdot(-)$ the term-wise multiplication on $BS(\vec{I},\vec{J})$. We call \emph{intersection form} the invariant form $\langle-,-\rangle_{(\vec{I},\vec{J})}$ on $BS(\vec{I},\vec{J})$ defined by 
	\[\langle x,y\rangle_{(\vec{I},\vec{J})} =\Trace(x\cdot y).\]
\end{definition}

\subsection{Adjoint and flipped maps}

Our next goal is to show that taking the adjoint with respect of the intersection forms is the same as flipping the corresponding diagrams. It is enough to show this for the ``building boxes'' of \Cref{boxes}.

We consider first the clockwise cup $f_{\carcd}$ and its flip, the clockwise cap $f_{\carcu}$.
Let $\vec{I}$ and $\vec{J}$ be as in the previous section and let \[\vec{I'}=(\emptyset,I_1,\ldots,I_{h-1},I_h,I_h,I_{h+1},\ldots, I_k),\] \[\vec{J'}=(J_1,\ldots,J_{h-1},K,J_h,J_{h+1},\ldots, J_k)\] with $I_h\subset K$. Let
\[F:=\Iden \otimes f_{\carcd}\otimes \Iden:BS(\vec{I},\vec{J})\raw BS(\vec{I'},\vec{J'})\]
be the morphism induced by $f_{\carcd}: R^{I_h}\raw R^{I_h}\otimes_{R^K}R^{I_h}(\ell(K)-\ell(I_h))$ and let 
\[\bar{F}:=\Iden\otimes f_\carcu\otimes\Iden:BS(\vec{I'},\vec{J'})\raw BS(\vec{I},\vec{J})\] be the morphism induced by $f_{\carcu}: R^{I_h}\otimes_{R^K}R^{I_h}(\ell(K)-\ell(I_h))\raw R^{I_h}$.

\begin{lemma}
	The maps $F$ and $\bar{F}$ above are adjoint to each other with respect to the corresponding intersection forms, i.e. we have 
	\[\langle F(x),y\rangle_{(\vec{I'},\vec{J'})}=\langle x, \bar{F}(y)\rangle_{(\vec{I},\vec{J})}\;\text{ for all } x\in BS(\vec{I},\vec{J}),\; y \in BS(\vec{I'},\vec{J'}).\]
\end{lemma}

\begin{proof}
	
	In $R^{I_h}\otimes_{R^{K}}R^{I_h}$ we have 
	\[ f_{\carcd}(z)\cdot (y_1\otimes y_2)=f_{\carcd}(1) zy_1y_2=f_{\carcd}(z\cdot f_{\carcu}(y_1\otimes y_2))\;\text{ for all }z, y_1,y_2 \in R^{I_h}\]
	where the first equality follows from the fact that $g f_{\carcd}(1)=f_{\carcd}(g)= f_{\carcd}(1)g$ for any $g\in R^{I_h}$.
	It follows that for every $x\in BS(\vec{I},\vec{J})$ and $y\in BS(\vec{I'},vec{J'})$ 
	we have $F(x)\cdot y=F(x\cdot \bar{F}(y))$ and $\langle x,y\rangle_{(\vec{I},\vec{J})}=\Trace (F(x\cdot \bar{F}(y))$.
	Thus, to conclude that $F$ and $\bar{F}$ are adjoint to each other it is enough to show that, for all $b\in BS(\vec{I},\vec{J})$ we have $\Trace(F(b))=\Trace(b)$.
	
	Assume that $b=g \ctop(\vec{I},\vec{J}) +b' $ with $g\in R$ and $b'\in LT(BS(\vec{I},\vec{J}))$. By degree reasons, also $F(b')\in LT(BS(\vec{I'},\vec{J'}))$. We conclude that $\Trace(F(b))=\Trace(b)$ by showing $F(\ctop(\vec{I},\vec{J}))=\ctop(\vec{I'},\vec{J'})$. This follows since the following diagram is commutative: 
	\begin{center}
		\begin{tikzpicture}
			\node (a) at (0,2) {$R\otimes_{R^{J_h}}R^{I_h}$};
			\node (b) at (6,2) {$R\otimes_{R^{J_h}}R^{I_h}\otimes_{R^K}R^{I_h}$};
			\node (c) at (0,3) {$R\otimes_{R^{I_h}}R^{I_h}\otimes_{R^{J_h}}R^{I_h}$};
			\node (d) at (6,3) {$R\otimes_{R^{I_h}}R^{I_h}\otimes_{R^{J_h}}R^{I_h}\otimes_{R^K}R^I$};
			\node (e) at (0,5) {$R\otimes_{R^{I_h}}R^{I_h}$};
			\node (f) at (6,5) {$R\otimes_{R^{I_h}}R^{I_h}\otimes_{R^K}R^{I_h}$};
			\path[->] 
			(e) edge node[above] {$\Iden \otimes f_{\carcd}$} (f)
			edge node[left] {$ f_{\carcd}$} (c)
			(f) edge node[right] {$\Iden \otimes f_{\carcd}\otimes \Iden$} (d)
			(a) edge node[above] {$\Iden \otimes f_{\carcd}$} (b);
			\node at (0,2.5) {$\vertcong$};
			\node at (6,2.5) {$\vertcong$};
			
		\end{tikzpicture}
	\end{center}
	as it can be easily checked, for example, by looking at the corresponding diagrams.
\end{proof}

We consider next the counterclockwise cup $f_{\carad}$ and its flip, the counterclockwise cap $f_{\carau}$. 
Let now \[\vec{I''}=(\emptyset,I_1,\ldots,I_{h-1},K,I_{h+1},\ldots, I_k)\qquad\text{ and }\vec{J''}=(J_1,\ldots,J_{h-1},J_h,J_{h+1},\ldots, J_k)\] with $K\subset I_h$. Let \[G=\Iden\otimes f _{\carad}\otimes \Iden:BS(\vec{I},\vec{J})\raw BS(\vec{I''},\vec{J''})\] be the morphism induced by $f_{\carad}:R^{I_h}\raw {}_{I_h}(R^K)_{I_h}(\ell(I_k)-\ell(K))$ and let \[
\bar{G}=\Iden \otimes f_{\carau}\otimes\Iden:BS(\vec{I''},\vec{J''})\raw BS(\vec{I},\vec{J})\] be the morphism induced by $f_{\carau}: {}_{R^{I_h}}(R^K)_{R^{I_h}}(\ell(I_k)-\ell(K))\raw R^{I_h}$.

\begin{lemma}
	The maps $G$ and $\bar{G}$ above are adjoint with respect to the corresponding intersection forms.
\end{lemma}
\begin{proof}We have
	\[ z_1\cdot f_{\carau} (z_2) =
	f_{\carau}(z_1 z_2)=f_{\carau}(f_{\carad}(z_1)z_2)\;\text{ for all }z_1\in R^{I_h}\text{ and }z_2\in R^K.\]
	and so $x\cdot \bar{G}(y)=\bar{G}(G(x)\cdot y)$ for any $x\in BS(\vec{I},\vec{J})$ and $y\in BS(\vec{I''},\vec{J''})$. Therefore, it is enough to show that for every $b\in BS(\vec{I''},\vec{J''})$ we have $Tr(b)=Tr(\bar{G}(b))$.
	
	Let $b\in BS(\vec{I''},\vec{J''})$. We can write $b=g\ctop(\vec{I''},\vec{J''})+ b'$, with $g\in R$ and $b'\in LT( BS(\vec{I''},\vec{J''}))$, so that $\Trace(b)=g$.
	
	By degree reasons, we have $\bar{G}(b')\in LT(BS(\vec{I},\vec{J}))$. To conclude we just need to show that $\bar{G}(\ctop(\vec{I''},\vec{J''}))=\ctop(\vec{I},\vec{J})$. This follows from the commutativity of the following diagram.
	
	\[
	\includegraphics[scale=0.18]{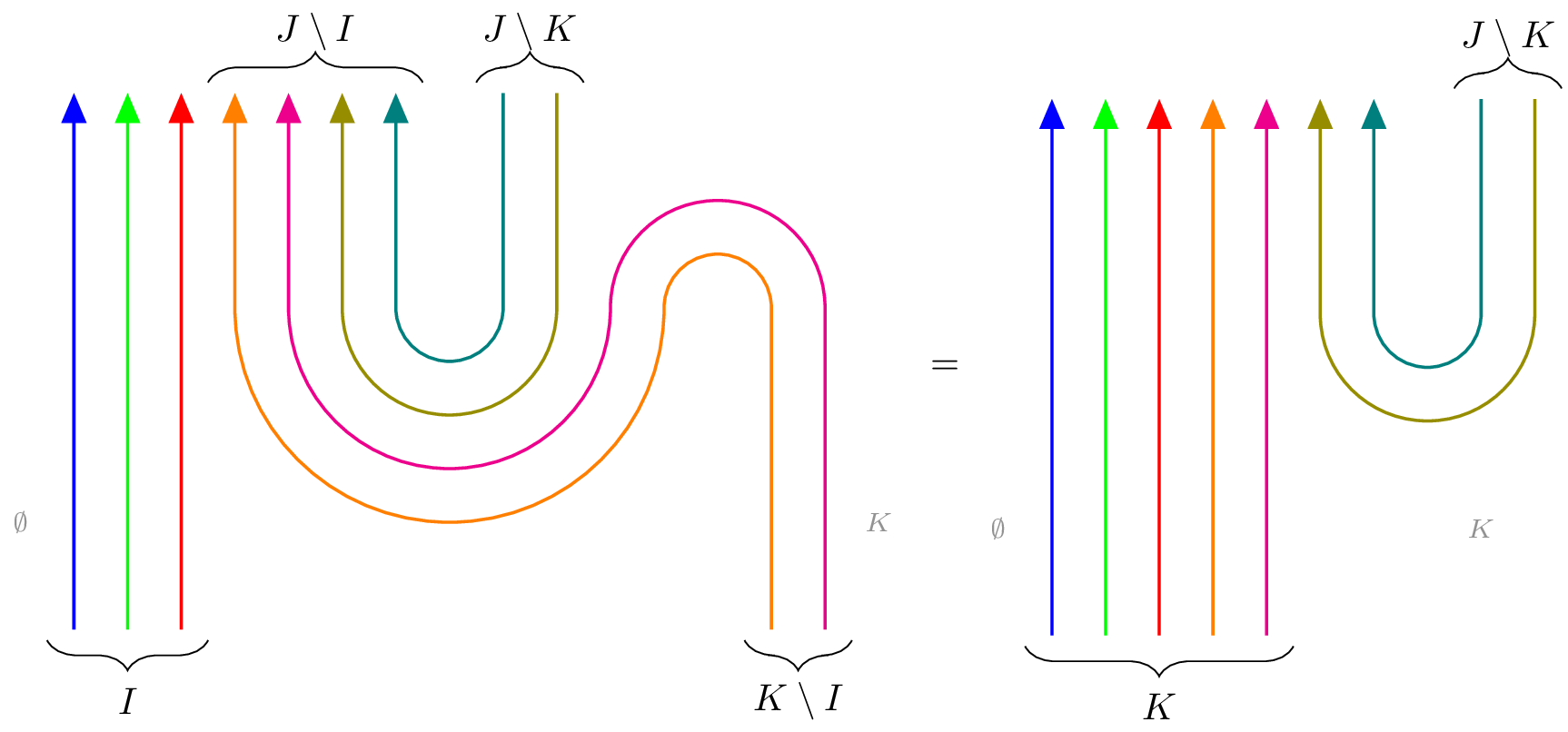}\qedhere
	\]
\end{proof}

Finally, it remains to consider the last building box: 
\begin{equation}\label{due1}
	\begin{tikzpicture}
		\draw[thick,red, -triangle 45] (0.76,0) to (-0.75,2);
		\draw[thick,blue, triangle 45-] (0.75,2) to (-0.75,0);
		\node at (0.5,1) {J};
		\node at (-0.5,1) {I};
		\node at (0,0.5) {K};
		\node at (0,1.5) {L};
		\node[red] at (0.8,0.3) {$s$};
		\node[blue] at (-0.8,0.3) {$t$};
	\end{tikzpicture}
\end{equation}

On bimodules $f_{\duarrow}$ induces the canonical identification 
\[R^I\otimes_{R^K} R^K\otimes_{R^J}R^J(\ell(J)-\ell(I))\cong R^I\otimes_{R^L} R^L\otimes_{R^J}R^J(\ell(J)-\ell(I)).\] 
By construction, $\Iden \otimes f_{\duarrow}\otimes \Iden $ identifies the corresponding $\ctop$ elements. In particular, $\Iden \otimes f_{\duarrow}\otimes \Iden $ is an isometry with respect to the corresponding intersection forms and if $f_{\ddarrow}$ denotes the inverse morphism of $f_{\duarrow}$ given by the following diagram 
\begin{equation}\label{due2}
	\begin{tikzpicture}
		\draw[thick,blue, -triangle 45] (0.76,0) to (-0.75,2);
		\draw[thick,red, triangle 45-] (0.75,2) to (-0.75,0);
		\node at (0.5,1) {J};
		\node at (-0.5,1) {I};
		\node at (0,0.5) {L};
		\node at (0,1.5) {K};
		\node[blue] at (0.8,0.3) {$t$};
		\node[red] at (-0.8,0.3) {$s$};
	\end{tikzpicture}
\end{equation}
then $\Iden \otimes f_{\ddarrow}\otimes \Iden$ is also the adjoint $\Iden \otimes f_{\duarrow}\otimes \Iden$. Notice that the diagram in \eqref{due2} is the flip of \eqref{due1}. 

We can condense the results of this section in the following proposition.
\begin{prop}
	Let $f:BS(\vec{I},\vec{J})\raw BS(\vec{I'},\vec{J'})$ be a morphism. Then the flipped morphism $\bar{f}:BS(\vec{I'},\vec{J'})\raw BS(\vec{I},\vec{J})$ is the adjoint of $f$ with respect to the corresponding intersection forms. 
\end{prop}

\subsection{Reduced translation pairs}

Reduced translating sequences \cite[Definition 1.3.1]{WPhD} are the analogue of reduced expression for double cosets of Coxeter groups. Since we are only interested in one-sided singular Soergel bimodules, we give a simpler definition which is valid in our setting.

\begin{definition}
	Let $(\vec{I},\vec{J})$ be a translation pair with $\vec{I}=(\emptyset,I_1,I_2,\ldots,I_k)$. Let 
	$v_1=w_{J_1}w_{I_1}$ and $v_{h}=v_{h-1}w_{J_h}w_{I_h}$ for every $h\leq k$.
	We call $v_k$ the \emph{end-point} of $(\vec{I},\vec{J})$.

	We say that $(\vec{I},\vec{J})$ is \emph{reduced} if for every $h<k$ we have 
	\[\ell(v_hw_{J_{h+1}}) = \ell(v_h)+\ell(w_{J_{h+1}}).\]
\end{definition}

If $(\vec{I},\vec{J})$ is a reduced translation pair with end-point $w\in W^I$, then $B_w^I$ is a direct summand of $BS(\vec{I},\vec{J})$ with multiplicity $1$. 

Recall the embedding of $BS(\vec{I},\vec{J})$ in $BS(\underline{w_{J_1}w_{I_1}}\;\underline{w_{J_2}w_{I_2}}\ldots \underline{w_{J_k}w_{I_k}})_I$ from \eqref{embintoBS}. 
If $(\vec{I},\vec{J})$ is a reduced translation pair, then
$\underline{w_{J_1}w_{I_1}}\;\underline{w_{J_2}w_{I_2}}\ldots \underline{w_{J_k}w_{I_k}}$ is a reduced word. Combining \Cref{deltatoctop} and \Cref{formulaperN} we have for every $\rho\in (h^*)^I$
\[1\otimes \rho^{\ell(\vec{I},\vec{J})}=\partial_{w}(\rho^{\ell(\vec{I},\vec{J})})\ctop(\vec{I},\vec{J}) + \text{lower degree terms}\]
and
\[\langle 1^\otimes,1^\otimes\cdot \rho^{\ell(\vec{I},\vec{J})}\rangle=\partial_w(\rho^{\ell(\vec{I},\vec{J})}).\]

\begin{lemma}
	Let $(\vec{I},\vec{J})$ and $(\vec{I'},\vec{J'})$ be two reduced translation pairs both having the same end-point $w\in W^I$. Let $\varphi :BS(\vec{I},\vec{J})\raw BS(\vec{I'},\vec{J'})$ be a morphism such that $\varphi(1^\otimes)= 1^\otimes$. Then also $\bar{\varphi}(1^\otimes)=1^\otimes$.
\end{lemma}

\begin{proof}
	Since $\varphi(1^\otimes)=1^\otimes$, the morphism $\varphi$ must be of degree $0$, and so does $\bar{\varphi}$. Then $\bar{\varphi}(1^\otimes)=c 1^\otimes$ for some scalar $c$. Let $\rho\in (\frh^*)^I$. Then we have
	\[\partial_w(\rho^k)=\langle \varphi(1^\otimes), 1^\otimes \rho^k\rangle_{(\vec{I'},\vec{J'})} =\langle 1^\otimes\rho^k, \bar{\varphi}(1^\otimes) \rangle_{(\vec{I},\vec{J})}=c\partial_w(\rho^k).\] 
	We can choose $\rho$ to be ample. Since $\partial_w(\rho^k)>0$, we obtain $c=1$.
\end{proof}

\begin{cor}\label{deltawelldef}
	Let $(\vec{I},\vec{J})$, $(\vec{I'},\vec{J'})$ and $\varphi :BS(\vec{I},\vec{J})\raw BS(\vec{I'},\vec{J'})$ be as above. Then $\varphi(\ctop(\vec{I},\vec{J}))=\ctop(\vec{I'},\vec{J'})$ up to lower degree terms. 
\end{cor}
\begin{proof}
	We have $\ctop(\vec{I},\vec{J})=\psi(1)$ where $\psi:=(f_{\carcd}\otimes \Iden)\circ (f_{\carcd}\otimes \Iden)\circ \ldots:R_I \raw BS(\vec{I},\vec{J})$, and similarly $\ctop(\vec{I'},\vec{J'})=\psi'(1)$ with $\psi':R_I\raw BS(\vec{I'},\vec{J'})$. The morphisms $\varphi\circ \psi$ and $\psi'$ are both of degree $\ell(\vec{I},\vec{J})$, so they coincide up to scalar and up to lower degree terms. 
	
	However, after taking the flip, we have
	\[ \bar{\psi}\circ \bar{\varphi}(1^\otimes) =1 =\bar{\psi'}(1^\otimes),\]
	hence $\varphi\circ \psi=\psi'$ up to lower degree terms.
\end{proof}

\bibliographystyle{alpha}

\Address
\end{document}